\documentclass[reqno,a4paper]{siamart250211}
\usepackage{amsmath,amssymb,amsfonts,mathtools,comment,mathrsfs}
\usepackage{tikz}
\usepackage{algorithm}
\usepackage{algpseudocode}
\usepackage{graphicx,epsfig,color,scalerel,subfigure,booktabs}
\usepackage[super]{nth}
\usepackage{xparse}
\usepackage{siunitx}
\usepackage{afterpage}
\usepackage{enumitem}
\usepackage{hyperref}
\tikzset{shift entire picture/.style n args={2}{execute at end picture={
\pgfmathtruncatemacro{\tmpx}{sign(#1)}
\pgfmathtruncatemacro{\tmpy}{sign(#2)}
\ifnum\tmpx=1
  \ifnum\tmpy=1
   \path[use as bounding box] ([xshift=-#1,yshift=-#2]current bounding box.south west) rectangle 
(current bounding box.north east);
  \else
   \path[use as bounding box] ([xshift=-#1]current bounding box.south west) rectangle 
([yshift=-#2]current bounding box.north east);
  \fi
\else  
  \ifnum\tmpy=1
   \path[use as bounding box] ([yshift=-#2]current bounding box.south west) rectangle 
([xshift=-#1]current bounding box.north east);
  \else
   \path[use as bounding box] (current bounding box.south west) rectangle 
([xshift=-#1,yshift=-#2]current bounding box.north east); 
  \fi
\fi}}}

\definecolor{lightgreen}{rgb}{0.2,0.6,0.2}
\definecolor{lightblue}{rgb}{0.15,0.15,0.85}
\definecolor{darkred}{rgb}{0.85,0.15,0.15}

\DeclareFontEncoding{OT2}{}{}
\DeclareFontSubstitution{OT2}{wncyr}{m}{n}

\DeclareSymbolFont{cyrletters}{OT2}{wncyr}{m}{n}

\DeclareMathSymbol{\CYRE}{\mathalpha}{cyrletters}{"03}

\newcommand{\bn}{\boldsymbol{n}}

\newcommand{\bu}{\boldsymbol{u}}
\newcommand{\bv}{\boldsymbol{v}}
\newcommand{\bw}{\boldsymbol{w}}

\newcommand{\bg}{\boldsymbol{g}}

\newcommand{\by}{\boldsymbol{y}}
\newcommand{\bz}{\boldsymbol{z}}

\newcommand\bxi{\boldsymbol{\xi}}

\newcommand\bdiv{\mathop{\mathbf{div}}\nolimits}

\newcommand\tr{\mathop{\mathrm{tr}}\nolimits}

\newcommand\qan{{\quad\hbox{and}\quad}}

\newcommand\qin{{\quad\hbox{in}\quad}}
\newcommand\qon{{\quad\hbox{on}\quad}}
\renewcommand\div{{\mathrm{div}}}
\newcommand\bta{{\boldsymbol\tau}}
\newcommand\bze{{\boldsymbol\zeta}}
\newcommand\bet{{\boldsymbol\eta}}
\newcommand\bsi{{\boldsymbol\sigma}}
\newcommand\brho{{\boldsymbol\rho}}

\newcommand\btet{{\boldsymbol\theta}}
\newcommand\bvatet{{\boldsymbol\vartheta}}
\newcommand\bvarr{{\boldsymbol\varrho}}

\newcommand\disp{\displaystyle}
\newcommand\Pcal{{\mathcal P}}
\newcommand{\Pcalbf}{{\boldsymbol{\mathcal P}}}

\newcommand{\Pibf}{\boldsymbol{\Pi}}
\newcommand{\Pibb}{\boldsymbol{\Pi}\hspace{-1.7ex}\boldsymbol{\Pi}}

\newcommand\bdev{{\mathbf{dev}}}

\numberwithin{equation}{section}
\numberwithin{remark}{section}
\numberwithin{assumption}{section}
\numberwithin{figure}{section}
\numberwithin{table}{section}

\allowdisplaybreaks

\headers{Mixed FEM for unsteady Brinkman-Forchheimer flow in desalination}{Gharibi, Abbaszadeh \& Dehghan}

\title{Mixed FEM 
for coupled unsteady fluid flow problems with $p$-type Brinkman-Forchheimer framework and its application for reverse-osmosis desalination\thanks{\textbf{Updated:} \today.\funding{This work has been partially supported by  Iran National Science Foundation (INSF) under project No. 4037159.}}}

\author{Zeinab Gharibi, Mostafa Abbaszadeh, Mehdi Dehghan \thanks{Department of Applied Mathematics, Faculty of Mathematics and Computer Sciences, Amirkabir University of Technology (Tehran Polytechnic), No. 424, Hafez Ave., 15914 Tehran, Iran (\email{z90gharibi@aut.ac.ir;m.abbaszadeh@aut.ac.ir;mdehghan@aut.ac.ir}).}
}
\date{\today}

\begin{document}

\maketitle
\begin{abstract}
This work analyzes a fully discrete mixed finite element method in a Banach space framework for solving nonstationary coupled fluid flow problems modeled by the Brinkman-Forchheimer equations, with applications to reverse osmosis. The model couples unsteady $p$-type convective Brinkman-Forchheimer and transport equations with nonlinear boundary conditions across a semi-permeable membrane.
A mixed formulation is used for the fluid equation (pseudostress-velocity) and for the transport equation (concentration, its gradient, and a Lagrange multiplier from the membrane condition). The continuous problem is reformulated in Banach spaces as a fixed-point problem, enabling a well-posedness analysis via differential-algebraic system theory.
Spatial discretization employs lowest-order Raviart-Thomas elements for fluxes and piecewise constants for primal variables, while linear elements are used for the Lagrange multiplier. A fully discrete Galerkin scheme with backward Euler time-stepping is proposed. Its well-posedness and stability are proven using a fixed-point argument, and optimal convergence rates are established. Numerical results confirm the theoretical error estimates and demonstrate the method's effectiveness.
\end{abstract}

\section{Introduction}\label{sec1}
Freshwater scarcity is a critical global challenge that has drawn increasing attention from governments, the public, and the scientific community. Among the various technologies developed to address this issue, reverse osmosis (RO) has emerged as a highly effective method for water purification due to its excellent selectivity and permeability \cite{bsckkksk-D-2023}. However, RO systems are inherently energy-intensive, primarily due to the high pressures required for operation. A key technical limitation that exacerbates energy consumption and undermines performance is concentration polarization--the accumulation of solutes near the membrane surface \cite{rslc-JMS-2006}. This phenomenon is complicated by unsteady mixing and transport processes, as well as poorly understood flow regimes within RO modules.
To better understand and mitigate concentration polarization, researchers have increasingly turned to computational fluid dynamics (CFD) for modeling mass transport and flow behavior in RO systems \cite{bmtdaz-JMS-2022,cllw-IJHMT-2022,damki-A-2023}.
Although few numerical methods have been specifically developed and analyzed for RO, notable contributions exist. Song et al.~\cite{mson-JSC-2004} used a Petrov-Galerkin FEM to address steep concentration gradients in membrane processes. Li et al.~\cite{lspa-JMS-2019} applied the lattice Boltzmann method for efficient microscale simulation in spacer-filled channels. Carro et al.~\cite{cmv-IJNMF-2024} introduced a Nitsche-based FEM for RO, while Khan et al.~\cite{kmrv-A-2024} improved it using a Lagrange multiplier and $H(\div)$-conforming elements to ensure well-posedness and pressure robustness.

Mixed finite element methods have become a key approach for solving linear and nonlinear fluid problems, especially via velocity-pressure and pseudostress-velocity formulations. They allow direct approximation of variables like velocity gradients and fluxes, avoiding less accurate post-processing via numerical differentiation.
In reverse osmosis modeling, Berm{\'u}dez et al.~\cite{bbcgos-arxive-2024} proposed a mixed FEM for coupled Navier-Stokes/transport and Brinkman-Forchheimer/transport equations with nonlinear membrane conditions, though without theoretical analysis. Later, Berm{\'u}dez et al.~\cite{bcos-CMAME-2025} developed and analyzed a nonlinear mixed variational formulation using auxiliary variables and Lagrange multipliers to handle low regularity, providing stable discretizations and error estimates.
In many practical cases-especially at high flow rates or in complex porous media—classical Darcy or Brinkman models fail to capture nonlinear inertial effects. The Brinkman-Forchheimer model \cite{f-ZVDI-1901} extends these by incorporating viscous diffusion and nonlinear drag. Its $p$-type variant \cite{cmc-JFM-1988,yl-CEJ-2012}, featuring a power-law drag term, offers a more accurate representation of flow in complex geometries and has gained attention for modeling nonlinear transport phenomena \cite{aco-SIAM-2023,cd-ANM-2023,gd-IMA-2025}.

Motivated by recent advances, this work develops and analyzes a mixed FEM for unsteady convective $p$-type Brinkman-Forchheimer flow coupled with a transport equation. The formulation introduces pseudostress, concentration gradient, and a Lagrange multiplier as additional unknowns alongside velocity and concentration. Raviart-Thomas elements and standard polynomials are used for spatial discretization, with backward Euler in time. The fully discrete scheme is high-order, implicit, and unconditionally stable. Under small data assumptions, we prove existence, uniqueness, and derive optimal error estimates. This appears to be the first analysis of a lowest-order mixed FEM for time-dependent Brinkman-Forchheimer flow models.

 \paragraph{Plan of the paper} The rest of this section introduces the notation and functional spaces used throughout the paper. Section~\ref{sec.2} presents the mathematical model, auxiliary variables, and variational formulation. Section~\ref{s:3} analyzes the solvability of the continuous problem in a Banach space setting. Section~\ref{s:4} develops the fully discrete mixed finite element scheme and proves its solvability using a discrete fixed-point approach. A priori error estimates are derived in Section~\ref{sec.5}, and Section~\ref{sec.7} provides numerical results demonstrating the method’s performance.

\paragraph{Sobolev and Banach Spaces}
Consider a Lipschitz-continuous domain $\mathcal{D}$ in $\mathbb{R}^2$ with boundary $\Gamma$. We use standard notation for Lebesgue spaces $L^p(\mathcal{D})$ and Sobolev spaces $W^{r,p}(\mathcal{D})$, where $ r \geq 0 $ and $ p \in [1,+\infty) $. The  associated  norms are denoted by $ \| \cdot \|_{0,p;\mathcal{D}} $ and $ \| \cdot \|_{r,p;\mathcal{D}} $, respectively. Note that $ W^{0,p}(\mathcal{D}) = L^p(\mathcal{D}) $, and when $ p = 2 $, we write $ H^r(\mathcal{D}) $ instead of $ W^{r,2}(\mathcal{D}) $. The norm and seminorm in this case are denoted as $ \| \cdot \|_{r,\mathcal{D}} $ and $ | \cdot |_{r,\mathcal{D}} $, respectively.  
Additionally, $ H^{1/2}(\Gamma) $ represents the space of traces of $ H^1(\mathcal{D}) $, while $ H^{-1/2}(\Gamma) $ denotes its dual space, equipped with the duality pairing $ \langle \cdot , \cdot \rangle_{\Gamma} $.  
On another note, for any given scalar functional space $ M $, we denote its vector- and tensor-valued counterparts as $ \mathbf{M} $ and $\mathbb{M}$. As customary, $ \mathbb{I} $ represents the identity tensor in $ \mathbb{R}^{2\times 2} $, and the notation $ | \cdot | $ refers to the Euclidean norm in $ \mathbb{R}^2 $.  
For $ s \in (1,+\infty) $, we define the Banach space by  $	\mathbf{H}(\text{div}_s; \mathcal{D}) := \Big\{ \mathbf{v} \in \mathbf{L}^2(\mathcal{D}) : \text{div}(\mathbf{v}) \in L^s(\mathcal{D}) \Big\}$, 
which is equipped with the norm  $	\|\mathbf{v} \|_{\text{div}_s; \mathcal{D}} := \|\mathbf{v} \|_{0,\mathcal{D}} + \|\text{div}(\mathbf{v})\|_{0,s;\mathcal{D}}$.
The space \( \mathbb{H}(\bdiv_s; \mathcal{D}) \) is defined to consist of matrix-valued functions whose rows belong to \( \mathbf{H}(\text{div}_s; \mathcal{D}) \), and it is endowed with the norm \( \| \cdot \|_{\bdiv_s; \mathcal{D}} \).
Here, the divergence operator $ \bdiv $ is understood to act row-wise on tensor. 
Moreover, for any tensor field \( \boldsymbol{\tau} \), the corresponding deviatoric part is defined as
$\bdev(\bta)\,:=\, \bta-\frac{1}{2}\tr(\bta)\mathbb{I}$.

Moreover, let $V$ be a separable Banach space endowed with the norm $\|\cdot\|_V$. Given $1\le p\le \infty$, we define $L^p(0, t_F; V)$ as the space of classes of functions $f$ : $(0, t_F) \rightarrow V$ that are Bochner measurable and such that $\|f\|_{L^p(0, t_F ; V)}<\infty$, with
\begin{equation*}
	\|f\|_{L^p(0, t_F; V)}^p  :=\int_0^{t_F}\|f(t)\|_V^p \mathrm{dt}\quad\forall \,1\le p<\infty \qan 	\|f\|_{L^{\infty}(0, t_F ; V)}:=\underset{t \in[0, t_F]}{\operatorname{ess} \sup }\|f(t)\|_V \,.
\end{equation*}
We present key concepts and auxiliary results pertaining to boundary conditions and the construction of extension operators (see \cite{gs-Elec-2007}). Given $\Gamma$, let $\widetilde{\Gamma} \subseteq \Gamma$ be a subset, and denote its complement and the unit outward normal vector on $\Gamma$ by $\widetilde{\Gamma}^c$
and $\mathbf{n}$, respectively.

\paragraph{Restriction of Functionals  in $H^{-1/2}(\Gamma)$ to $\widetilde{\Gamma}$}
Consider the extension operator $E_{0, \widetilde{\Gamma}} : H^{1/2}(\widetilde{\Gamma}) \to L^2(\Gamma)$, defined by $ 	E_{0,\widetilde{\Gamma}}(v) =v$ on $\widetilde{\Gamma} $, and $ 	E_{0,\widetilde{\Gamma}}(v)=0$ on $\widetilde{\Gamma}^c$.
We define the space $H^{1/2}_{00}(\widetilde{\Gamma})$ as
$	H^{1/2}_{00}(\widetilde{\Gamma}) := \big\{ v \in H^{1/2}(\Gamma) :\,\, E_{0,\widetilde{\Gamma}}(v) \in H^{1/2}(\Gamma) \big\}$, 
equipped with the norm 
\begin{equation}\label{def.normH12}
	\| v \|_{1/2,00,\widetilde{\Gamma}}\, :=\, \|E_{0,\widetilde{\Gamma}}(v)\|_{1/2,\Gamma}\,.
\end{equation}
Its dual space is denoted by $H^{-1/2}_{00}(\widetilde{\Gamma})$.
For any $\varphi \in H^{-1/2}(\Gamma)$, its restriction to $\widetilde{\Gamma}$, denoted as $\varphi|_{\widetilde{\Gamma}}$, is given by:
\begin{equation}\label{def.E0}
	\langle \varphi|_{\widetilde{\Gamma}}, v \rangle_{\widetilde{\Gamma}} := \langle \varphi, E_{0,\widetilde{\Gamma}}(v) \rangle_{\Gamma}\qquad \forall\, v \in H^{1/2}_{00}(\widetilde{\Gamma})\,.
\end{equation}
It follows that $\varphi|_{\widetilde{\Gamma}}\in H^{-1/2}_{00}(\widetilde{\Gamma})$.  Conversely, the boundary condition $ \varphi|_{\widetilde{\Gamma}}=0$ implies that $	\langle \varphi, E_{0,\widetilde{\Gamma}}(v) \rangle_{\Gamma} = 0$ for all $ v \in H^{1/2}_{00}(\widetilde{\Gamma})$.

\paragraph{Continuous Extension of $H^{1/2}(\widetilde{\Gamma})$-Functions}
Following \cite[Lemma 2.1]{gs-Elec-2007}, 
the restriction of \( \varphi \in H^{-1/2}(\Gamma) \) to the subset \( \widetilde{\Gamma} \) corresponds to an element in \( H^{-1/2}(\widetilde{\Gamma}) \), defined via $	\langle \varphi, v \rangle_{\widetilde{\Gamma}} := \langle \varphi, E_{\widetilde{\Gamma}}(v) \rangle_{\Gamma} $ for all $v \in H^{1/2}(\widetilde{\Gamma})$, 
where $E_{\widetilde{\Gamma}} : H^{1/2}(\widetilde{\Gamma}) \to H^{1/2}(\Gamma)$ is a bounded extension operator. For instance, if $v \in H^{1/2}(\widetilde{\Gamma})$, we define its extension as $E_{\widetilde{\Gamma}}(v) := z|_{\Gamma}$, where $z$ is the unique solution to
\[
	\Delta z = 0 \quad \text{in } \mathcal{D}\,, \quad
z = v \quad \text{on } \widetilde{\Gamma}\,,\quad
\nabla z \cdot \mathbf{n} = 0 \quad \text{on } \widetilde{\Gamma}^c \,,
\]
which satisfies $\| z \|_{1, \mathcal{D}} \leq C \| v \|_{1/2, \widetilde{\Gamma}}$, for some constant $C > 0$. This, in turn, ensures the continuity of $ E_{\widetilde{\Gamma}}$, meaning that $	\Vert E_{\widetilde{\Gamma}}(v)\Vert_{1/2,\Gamma}\,\le \, C \Vert v\Vert_{1/2,\widetilde{\Gamma}}$.
%

\paragraph{Decomposition of $H^{1/2}(\Gamma)$-Functions}
According to \cite[Lemma 2.2]{gs-Elec-2007}, for any $\zeta \in H^{1/2}(\Gamma)$, there exist unique $\zeta_{\widetilde{\Gamma}} \in H^{1/2}(\widetilde{\Gamma})$ and $\zeta_{\widetilde{\Gamma}^c} \in H^{1/2}_{00}(\widetilde{\Gamma}^c)$ satisfying $	\zeta = E_{\widetilde{\Gamma}}(\zeta_{\widetilde{\Gamma}}) + E_{0, \widetilde{\Gamma}^c}(\zeta_{\widetilde{\Gamma}^c})$. 
This yields the decomposition:
\begin{equation}
	\langle \varphi, \zeta \rangle_{\Gamma} = \langle \varphi, E_{\widetilde{\Gamma}}(\zeta_{\widetilde{\Gamma}}) \rangle_{\Gamma} + \langle \varphi, E_{0, \widetilde{\Gamma}^c}(\zeta_{\widetilde{\Gamma}^c}) \rangle_{\Gamma} \quad \forall \varphi \in H^{-1/2}(\Gamma)\,.
\end{equation}

\textbf{Remark.}
We use \(\mathbf{E}_{\widetilde{\Gamma}}\) and \(\mathbf{E}_{0, \widetilde{\Gamma}^c}\) for extension operators on vector-valued functions. The decomposition follows by applying them component-wise.
%
\section{Model Problem and Continuous Formulation}\label{sec.2}
Let \(\Omega \subset \mathbb{R}^2\) be a bounded polygonal domain with boundary \(\Gamma\) and outward unit normal \(\mathbf{n}\) defined almost everywhere. For the time interval \(\mathrm{J} := (0, t_F]\), we study the nonstationary incompressible convective Darcy-Brinkman-Forchheimer equation coupled with a transport equation via a convective term (see, e.g., \cite{kmm-A-2023}). The problem is to find the velocity \(\mathbf{u} : \Omega \times \mathrm{J} \to \mathbb{R}^2\), pressure \(p : \Omega \times \mathrm{J} \to \mathbb{R}\), and salt concentration \(\tilde{\varphi} : \Omega \times \mathrm{J} \to \mathbb{R}\) such that
\begin{subequations}\label{eq1}
	\begin{align}
		\partial_t\bu	-\operatorname{\textbf{div}}(\nu\,\boldsymbol{\nabla}\bu) + (\boldsymbol{\nabla}\mathbf{u})\mathbf{u}+\mathtt{F}\,|\bu|^{p-2}\bu+\nabla p &\,=\,\mathbf{0}\quad \qin\Omega\times \mathrm{J}\,,\label{eq:prob1}\\[1mm]
		\operatorname{div}(\mathbf{u})&\,=\,0\quad \qin\Omega\times \mathrm{J}\,,\label{eq:prob2}\\[1mm]
		\partial_t\tilde{\varphi}	-\operatorname{div}(\kappa\,\nabla\tilde{\varphi}) +\mathbf{u}\cdot\nabla\tilde{\varphi} &\,=\,0\quad \qin\Omega\times \mathrm{J}\,,\label{eq:prob3}
	\end{align}
\end{subequations}
where \(\nu\) and \(\kappa\) are the fluid viscosity and solute diffusivity coefficients, respectively, assumed bounded above and below.
Let \(\Gamma_{\mathtt{i}}\), \(\Gamma_{\mathtt{w}}\), and \(\Gamma_{\mathtt{o}}\) denote the inlet, wall, and outlet boundaries, respectively, with \(\bar{\Gamma}_{\mathtt{i}} \cup \bar{\Gamma}_{\mathtt{w}} \cup \bar{\Gamma}_{\mathtt{o}} = \Gamma\) and pairwise disjoint. Then problem \eqref{eq1} is equipped with the boundary conditions:
\begin{subequations}
	\begin{align}
		\mathbf{u}&\,=\,\mathbf{u}_{\mathtt{in}}\qan \tilde{\varphi}\,=\,\tilde{\varphi}_{\mathtt{in}}~~\quad \text{on}~\Gamma_{\mathtt{i}}\times \mathrm{J}\,,\label{eq:bc1}\\[1mm]
		\mathbf{u}&\,=\,(A\Delta P-AiRT \,\tilde{\varphi})\,\mathbf{n}\qan \kappa\nabla\tilde{\varphi}\cdot\bn\,=\,\tilde{\varphi}\,\bu\cdot\bn~~\quad \text{on}~\Gamma_{\mathtt{w}}\times \mathrm{J}\,,\label{eq:bc2}	\\[1mm]
		(\nu\,&\boldsymbol{\nabla}\bu-p\mathbb{I})\,\bn\,=\,\mathbf{0}\qan \kappa\nabla\tilde{\varphi}\cdot\bn\,=\,0~~\quad \text{on}~\Gamma_{\mathtt{o}}\times \mathrm{J}\,, \label{eq:bc3}	
	\end{align}
\end{subequations}
where \(A\), \(\Delta P\), \(i\), \(R\), and \(T\) are positive physical parameters, and \(\mathbf{u}_{\mathtt{in}}(t) \in \mathbf{H}^{1/2}(\Gamma_{\mathtt{i}})\), \(\tilde{\varphi}_{\mathtt{in}}(t) \in H^{1/2}(\Gamma_{\mathtt{i}})\) are given for almost all \(t \in \mathrm{J}\).  
Additionally, given \(\mathbf{u}_0 \in \mathbf{H}^1(\Omega)\) and \(\tilde{\varphi}_0 \in H^1(\Omega)\), the initial condition is defined as follows:
\begin{equation}\label{eq:ini}
	\bu(\cdot,0)\,=\, \bu_0 \qan \tilde{\varphi}(\cdot,0)\,=\, \tilde{\varphi}_0 \quad\qin \Omega\,.
\end{equation}
For convenience, we set \(\varphi := \tilde{\varphi} - \tilde{\varphi}_{\mathtt{in}}\) in \(\Omega \times \mathrm{J}\), introducing a new concentration variable \(\varphi\) that satisfies a homogeneous Dirichlet condition on \(\Gamma_{\mathtt{i}}\). Instead of \(\tilde{\varphi}\), we focus on \(\varphi\), so the boundary conditions \eqref{eq:bc1}–\eqref{eq:bc1} can be rewritten as
\begin{subequations}\label{eq:bc}
	\begin{align}
		\mathbf{u}&\,=\,\mathbf{u}_{\mathtt{in}}\qan \varphi\,=\,0~~\quad \text{on}~\Gamma_{\mathtt{i}}\times \mathrm{J}\,,\label{eq:bc1a}\\[1mm]
		\mathbf{u}&\,=\,(a_0-a_1 \,\varphi)\,\mathbf{n}\qan \kappa\nabla\varphi\cdot\bn\,=\,(\varphi+\tilde{\varphi}_{\mathtt{in}})\,\bu\cdot\bn~~\quad \text{on}~\Gamma_{\mathtt{w}}\times \mathrm{J}\,,\label{eq:bc2a}	\\[1mm]
		(\nu\,&\boldsymbol{\nabla}\bu-p\mathbb{I})\,\bn\,=\,\mathbf{0}\qan \kappa\nabla\varphi\cdot\bn\,=\,0~~\quad \text{on}~\Gamma_{\mathtt{o}}\times \mathrm{J}\,, \label{eq:bc3a}	
	\end{align}
\end{subequations}
where $a_0 = A\Delta P- AiRT\tilde{\varphi}_{\mathtt{in}}$ and $a_1 = AiRT$.

Next, to formulate the mixed variational approach via integration by parts, we introduce auxiliary unknowns corresponding to Neumann-type boundary conditions:
\begin{equation}\label{eq:defNewV2}
	\bsi\,:=\,\nu\, \boldsymbol{\nabla}\bu-p \,\mathbb{I}\qan \rho\,:=\, \kappa\nabla\varphi\qquad\qin \Omega\times\mathrm{J}	\,.
\end{equation}
By taking the matrix trace and using \(\div(\mathbf{u})=0\), the first equation in \eqref{eq:defNewV2} is equivalent to
\begin{equation}\label{eq:postpr}
	\bdev(\bsi)\,=\,\nu\boldsymbol{\nabla}\bu \qan 	p \,=\,-\dfrac{1}{2}\,\tr\big( \boldsymbol{\sigma}\big)
	\qin \Omega\times\mathrm{J} \,,
\end{equation}
Gathering \eqref{eq:defNewV2} and \eqref{eq:postpr}, problem \eqref{eq1} with initial-boundary conditions \eqref{eq:ini}, \eqref{eq:bc} is equivalently reformulated as:  
Find \((\bsi, \bu)\) and \((\brho, \varphi)\) in suitable spaces such that
\begin{equation}\label{sys1}
	\begin{array}{rcl}
		\dfrac{1}{\nu}\bdev(\bsi)&=&\boldsymbol{\nabla}\bu\quad \qin \Omega\times \mathrm{J}\,,\\
		\partial_t\bu-\bdiv(\bsi)-\dfrac{1}{\nu}\bdev(\bsi)\,\bu+\mathtt{F}|\bu|^{p-2}\bu&=&\mathbf{0}\quad\qin \Omega\times \mathrm{J}\,,\\
		\bu = \bu_{\mathtt{in}} \qon \Gamma_{\mathtt{i}}\times \mathrm{J}\,,\,\,\,
		\bu = (a_0 - a_1 \, \varphi)\bn\qon \Gamma_{\mathtt{w}}\times \mathrm{J}\,,\,\,\,
		\bsi\bn &=& \mathbf{0} \quad\qon \Gamma_{\mathtt{o}}\times \mathrm{J}\,,\\
		\bu (\cdot,0)&=&  \bu_0 \quad\qin\Omega\,,\qan
	\end{array}	
\end{equation}
\begin{equation}\label{sys2}
	\begin{array}{rcl}
		\dfrac{1}{\kappa}\,\brho &=& \nabla\varphi  \qin \Omega\times \mathrm{J}\,,\\
		\partial_t \varphi	-\div(\brho)-\dfrac{1}{\kappa}\,\brho \cdot\bu&=& 0 \qin \Omega\times \mathrm{J}\,,\\
		\varphi =0\,\,\,\,\text{on}\,\,\,\, \Gamma_{\mathtt{i}}\times \mathrm{J}\,,\,\,\,
		\brho\cdot\bn = \tilde{a}_0+(a_2 - a_1 \,\varphi)\,\varphi \,\,\,\,\text{on}\,\,\,\, \Gamma_{\mathtt{w}}\times \mathrm{J}\,,\,\,\,
		\brho\cdot\bn &=& 0  \qon \Gamma_{\mathtt{o}}\times \mathrm{J}\,,\\
		\varphi(\cdot,0) &=& \varphi_0 \qin \Omega\,, 
	\end{array}
\end{equation}
where $\tilde{a}_0 = a_0 \tilde{\varphi}_{\mathtt{in}}$ and $a_2 = a_0 - a_1 \tilde{\varphi}_{\mathtt{in}}$.

\subsection{Variational Formulation}
Let \(\Gamma_{\mathtt{i}}^{c} := \partial\Omega \setminus \Gamma_{\mathtt{i}}\) and \(\Gamma_{\mathtt{o}}^{c} := \partial\Omega \setminus \Gamma_{\mathtt{o}}\).  
Testing the first equations in \eqref{sys1} and \eqref{sys2} with tensor \(\bta\) and vector \(\bet\) yields:
\begin{equation}\label{eq2}
	\dfrac{1}{\nu}\int_{\Omega}\bdev(\bsi):\bta\,=\,\int_{\Omega}\boldsymbol{\nabla}\bu:\bta\qan \dfrac{1}{\kappa}\int_{\Omega}\brho\cdot\bet  \,=\,\int_{\Omega}\nabla\varphi\cdot\bet \quad\text{for a.e}\,\, t\in \mathrm{J}\,.
\end{equation}
The terms on the left-hand side are well-defined for \(\bsi(t), \bta \in \mathbb{L}^2(\Omega)\) and \(\brho(t), \bet \in \mathbf{L}^2(\Omega)\) for almost all \(t \in \mathrm{J}\).  
Using vector- and scalar-valued test functions for the second rows of \eqref{sys1} and \eqref{sys2} yields
\begin{subequations}\label{eq2a}
	\begin{alignat}{2}
		\int_{\Omega} \partial_t \bu \cdot \bv	-\int_{\Omega}\bdiv (\bsi)\cdot\bv	-\frac{1}{\nu}\int_{\Omega}\bdev(\bsi)\,\bu\cdot\bv
		+\mathtt{F}\int_{\Omega}|\bu|^{p-2}\bu\cdot \textbf{v}\,&=\,0 &\quad\text{for a.e}\,\, t\in \mathrm{J}\,,\label{o1}\\
		\int_{\Omega}	\partial_t \varphi\,\psi	-\int_{\Omega}\div(\brho) \,\psi -\frac{1}{\kappa}\,\int_{\Omega}(\brho \cdot\bu)\, \psi\,&=\,0&\quad\text{for a.e}\,\, t\in \mathrm{J}\,.\label{o2}
	\end{alignat}
\end{subequations}
Given that $\bsi$ and $\brho$ belong to the space of $L^2$-functions, and applying the Cauchy--Schwarz and H\"older inequalities, we obtain that for all $r,s\in (1,\infty)$ satisfying $\frac{1}{r}+\frac{1}{s}=1$, the following holds
\begin{subequations}
	\begin{alignat}{2}
		\Big|\int_{\Omega}\bdev(\bsi)\,\bu\cdot\bv\Big|&\,\le\, \Vert\bsi\Vert_{0,\Omega}\,\Vert\bu\Vert_{0,2r;\Omega}\,\Vert\bv\Vert_{0,2s;\Omega}\,,\label{wel1}\\
		\Big|\int_{\Omega}|\bu|^{p-2}\bu\cdot \textbf{v}\Big|&\,\le \, \Vert \bu\Vert_{0,2(p-2);\Omega}^{p-2}\Vert\bu\Vert_{0,2r;\Omega}\,\Vert\bv\Vert_{0,2s;\Omega}\,,\label{wel2}\\
		\Big|\int_{\Omega}(\brho \cdot\bu)\, \psi\Big|&\,\le\, \Vert\brho\Vert_{0,\Omega}\,\Vert\bu\Vert_{0,2r;\Omega}\,\Vert\psi\Vert_{0,2s;\Omega}\,,\label{wel3}
	\end{alignat}
\end{subequations}
From \eqref{wel1} and \eqref{wel3}, it follows that the third terms on the left-hand sides of \eqref{o1} and \eqref{o2} are well-defined for $\bu(t)\in \mathbf{L}^{2r}(\Omega)$, $\bv\in\mathbf{L}^{2s}(\Omega)$ and $\psi\in\mathrm{L}^{2s}(\Omega)$. 
In turn, noting that for \( p \in [3,4] \), we have \( p - 2 \leq r \), we use the embedding \( \mathbf{L}^{2r}(\Omega) \hookrightarrow \mathbf{L}^{2(p-2)}(\Omega) \), which yields an upper bound of \( |\Omega|^{(r - p + 2)/2r(p - 2)} \), in \eqref{wel2} to obtain
\begin{equation}\label{eq:bound.tr2}
	\Big|\int_{\Omega}|\bu|^{p-2}\bu\cdot \textbf{v}\Big|\,\leq \, |\Omega|^{(r-p+2)/2r}\Vert \bu\Vert_{0,2r;\Omega}^{p-2}\,\Vert\bu\Vert_{0,2r;\Omega}\,\Vert\bv\Vert_{0,2s;\Omega}\,,
\end{equation}
which shows the fourth term in \eqref{o1} is well-defined. To ensure the second terms in \eqref{o1} and \eqref{o2} are well-defined, we require \(\bdiv(\bsi) \in \mathbf{L}^{(2s)'}(\Omega)\) and \(\div(\brho) \in L^{(2s)'}(\Omega)\), where \((2s)' = \frac{2s}{2s - 1}\).  
Next, to handle the right-hand side of the second equation in \eqref{eq2}, we introduce the Lagrange multiplier \(\lambda := -\varphi|_{\Gamma_{\mathtt{i}}^c} \in H_{00}^{1/2}(\Gamma_{\mathtt{i}}^{c})\). By integration by parts, we have
\begin{equation}\label{eq3a}
	\dfrac{1}{\kappa}\int_{\Omega} \brho\cdot\bet +\int_{\Omega} \varphi\,\div(\bet)\,=\,-\langle \bet\cdot\bn ,\, \lambda \rangle_{\Gamma_{\mathtt{i}}^{c}}\,.
\end{equation}
This step uses that \(\varphi = 0\) on \(\Gamma_{\mathtt{i}}\) and that \(\bet \cdot \bn\) is well-defined since the embedding \(H^1(\Omega) \hookrightarrow L^{2s}(\Omega)\) is continuous (see \cite[Section 3.1]{cgm-ESAIM-2020}), which holds in 2D for all \(2s \in [1, \infty)\).  
Similarly, to apply integration by parts to the first equation of \eqref{eq2} and derive
 \begin{equation}\label{eq3}
	\dfrac{1}{\nu}\int_{\Omega}\bdev(\bsi):\bta +\int_{\Omega}\bu\cdot\bdiv(\bta)\,=\,\langle \bta\bn, \, \bu\rangle_{\Gamma}\,,
\end{equation}
we require \(\bdiv(\bsi) \in \mathbf{L}^{(2r)'}(\Omega)\) and the continuous embedding \(\mathbf{H}^1(\Omega) \hookrightarrow \mathbf{L}^{2r}(\Omega)\), where \((2r)' = \frac{2r}{2r - 1}\). This ensures \(\bta \bn\) is well defined. In 2D, the embedding holds for all \(2r \in [1, \infty)\).  
The term \(\langle \bta \bn, \bu \rangle_{\Gamma}\) will be handled using the boundary condition from the third row of \eqref{sys1}.
More specifically, the following auxiliary functions are introduced.
\begin{equation}
	\mathbf{g}\,:=\,\left\{\begin{matrix}
		\bu_{\mathtt{in}} & \qon & \Gamma_{\mathtt{i}}\,,\\
		a_{0}\bn & \qon & \Gamma_{\mathtt{w}}\,,
	\end{matrix}\right.
	\quad\qan \quad
	\mathbf{g}^{\lambda}\,:=\,\left\{\begin{matrix}
		\mathbf{0} & \qon & \Gamma_{\mathtt{in}}\,,\\
		a_1 \lambda \bn & \qon & \Gamma_{\mathtt{w}}\,,
	\end{matrix}\right.
\end{equation}
Note that \(\mathbf{g} \in \mathbf{H}^{1/2}(\Gamma_{\mathtt{o}}^{c})\) when \(\bu_{\mathtt{in}}\) satisfies the compatibility condition \(\bu_{\mathtt{in}} = \mathbf{0}\) on \(\bar{\Gamma}_{\mathtt{i}} \cap \bar{\Gamma}_{\mathtt{w}}\).  
Also, \(\mathbf{g}^\lambda \in \mathbf{H}^{1/2}(\Gamma_{\mathtt{o}}^{c})\) since \(\lambda \in H_{00}^{1/2}(\Gamma_{\mathtt{i}}^{c})\).  
Using the third boundary condition in \eqref{sys1} and testing with \(\bta \in \mathbb{H}_{\Gamma_{\mathtt{o}}}(\bdiv_{(2r)^{'}};\Omega)\,:=\, \big\{\bta\in\mathbb{H}(\bdiv_{(2r)^{'}};\Omega):\ \quad \bta\mathbf{n}\,=\,\mathbf{0}\qon\Gamma_{\mathtt{o}}   \big\}\),
 we obtain from \eqref{eq3} that
\begin{equation}\label{eq4a}
	\dfrac{1}{\nu}\int_{\Omega}\bdev(\bsi):\bta +\int_{\Omega}\bu\cdot\bdiv(\bta)\,=\,\langle \bta\bn, \, \mathbf{g}+\mathbf{g}^{\lambda}\rangle_{\Gamma_{\mathtt{o}}^{c}}\qquad\forall\,\bta\in\mathbb{H}_{\Gamma_{\mathtt{o}}}(\bdiv_{(2r)^{'}};\Omega) \,.
\end{equation}
%
Seeking \(\bsi\) and \(\bta\) in the same space leads to \(s = r = 2\) and \((2s)' = (2r)' = \frac{4}{3}\) from now on.  Using the second equation in the third row of \eqref{sys2} and \(\lambda \in H_{00}^{1/2}(\Gamma_{\mathtt{i}}^{c})\), we conclude that
\begin{equation}\label{eq5a}
	\langle \brho\cdot \bn,\, \xi\rangle_{\Gamma_{\mathtt{i}}^{c}}\,=\,\tilde{a}_0 \int_{\Gamma_{\mathtt{w}}}\xi \,+\, a_2 \int_{\Gamma_{\mathtt{w}}}\lambda\, \xi \,+\,a_1 \int_{\Gamma_{\mathtt{w}}}\lambda^2 \, \xi\qquad\forall\, \xi\in H_{00}^{1/2}(\Gamma_{\mathtt{i}}^{c})\,.
\end{equation}

According to the foregoing discussion, and aiming to conveniently rewrite the system of equations \eqref{eq2a},\eqref{eq3a} and \eqref{eq4a}, we now introduce the spaces
$\mathbb{H}\,:=\, \mathbb{H}_{\Gamma_{\mathtt{o}}}(\bdiv_{4/3};\Omega)\,,\,\,\,\mathbf{V}\,:=\,\mathbf{L}^{4}(\Omega)\,,\,\,\,\mathbf{H}\,:=\, \mathbf{H}(\div_{4/3};\Omega)\,,\,\,\, \mathrm{V}\,:=\, \mathrm{L}^{4}(\Omega)\,,\,\,\, Q \,:=\,H_{00}^{1/2}(\Gamma_{\mathtt{i}}^{c})\,,$
which are endowed, respectively, with the norms
$\Vert\bta\Vert_{\mathbb{H}}\,:=\,\Vert\bta\Vert_{\bdiv_{4/3};\Omega}\,,\,\,\, \Vert\bv\Vert_{\mathbf{V}}\,=\,\Vert\bv\Vert_{0,4;\Omega}\,,\,\,\,\Vert\bet\Vert_{\mathbf{H}}\,:=\, \Vert\bet\Vert_{\div_{4/3};\Omega}\,,\,\,\,\Vert\psi\Vert_{\mathrm{V}}\,:=\, \Vert\psi\Vert_{0,4;\Omega}\,,\,\,\,\Vert\xi\Vert_{Q}\,:=\, \Vert\xi\Vert_{1/2,00,\Gamma_{\mathtt{i}}^{c}}\,.$

In this way, defining the global space $\mathcal{V}\,:=\, \mathrm{V}\times Q$, setting notation
\[\vec{\varphi}\,:=\, (\varphi,\lambda)\in\mathcal{V}\,,\quad
\vec{\psi}\,:=\, (\psi,\xi)\in\mathcal{V}\,,\quad \vec{\phi}\,:=\,(\phi,\chi)\in\mathcal{V}\,,\]
and equipping with the norm
$\Vert\vec{\psi}\Vert_{\mathcal{V}}\,:=\,\Vert\psi\Vert_{\mathrm{V}}+\Vert\xi\Vert_{Q}\,,$
Eqs. \eqref{eq2a},\eqref{eq3a} and \eqref{eq4a} can be reformulated as: Find $(\bsi,\bv)\in L^{2}(\mathrm{J};\mathbb{H})\times L^{2}(\mathrm{J};\mathbf{V})$ and $(\brho,\vec{\varphi})\in L^{2}(\mathrm{J};\mathbf{H})\times L^{2}(\mathrm{J};\mathcal{V})$ such that
\begin{equation}\label{eq:v1a-eq:v4}
	\begin{split}
		\mathscr{A}^{F}(\bsi , \bta)+\mathscr{B}^{F}(\bta, \bu)&= \mathscr{F}_{\lambda}^{F}(\bta)\quad \text{for a.e.}\,  t\in \mathrm{J} \,,\\
		\disp \int_{\Omega}\partial_t \bu \cdot\bv-\mathscr{B}^{F}(\bsi,\bv)-\mathscr{O}^{F}_1(\bu;\bsi, \bv)+\mathscr{O}^{F}_2(\bu;\bu, \bv)&=0  \qquad \text{for a.e.}\,  t\in \mathrm{J} \,,\\
		\mathscr{A}^{C}(\brho , \bet)+\mathscr{B}^{C}(\bet,\vec{\varphi})&= 0\quad \text{for a.e.}\,  t\in \mathrm{J} \,,\\
		\hspace{-.25cm}\disp \int_{\Omega}\partial_t \varphi \, \psi-	\mathscr{B}^{C}(\brho,\vec{\psi})+\mathscr{C}^{C}(\vec{\varphi}, \vec{\psi})-\mathscr{O}^{C}_1 (\bu;\brho,\vec{\psi})+\mathscr{O}^{C}_2(\lambda;\vec{\varphi}, \vec{\psi})&=\mathscr{F}^C(\vec{\psi}) \quad \text{for a.e.}\,  t\in \mathrm{J}  \,,
	\end{split}
\end{equation}
for any $ (\bta, \bv)\in\mathbb{H}\times\mathbf{V} $ and $ (\bet, \vec{\psi})\in \mathbf{H}\times\mathcal{V}$, 
where the bilinear forms 
$ \mathscr{A}^F: \mathbb{H}\times \mathbb{H}\rightarrow \mathrm{R} $, 
$ \mathscr{B}^F: \mathbb{H}\times \mathbf{V}\rightarrow \mathrm{R} $,  
$\mathscr{A}^{C} : \mathbf{H}\times \mathbf{H}\rightarrow \mathrm{R}$,  
$ \mathscr{B}^T: \mathbf{H}\times \mathcal{V}\rightarrow \mathrm{R} $  and
$\mathscr{C}^{C} : \mathrm{Q}\times\mathrm{Q}\rightarrow \mathrm{R}$ are defined as
\begin{equation}\label{def-a-b-tilde-a-tilde-b}
	\begin{array}{rcll}
		\mathscr{A}^F( \bze , \bta )&:=& \disp\frac{1}{\nu}\int_{\Omega}\bdev(\bze) : \bdev(\bta)  \,,\quad &	\mathscr{A}^{C}( \bxi , \bet ):= \disp\frac{1}{\kappa}\int_{\Omega}\bxi \cdot \bet  \,, \\
		\mathscr{B}^F(\bta,\bv)&:=& \disp  \int_{\Omega}\textbf{v}\cdot\bdiv(\bta)\,,   	\quad &
		\mathscr{B}^{C}(\bet, \vec{\psi}):=
		\disp\int_{\Omega}\psi\, \div(\bet)+\langle\bet\cdot\bn,\, \xi\rangle_{\Gamma_{\mathtt{i}}^{c}}  \,,\\
		\mathscr{C}^{C}(\vec{\phi}, \vec{\psi})&:=&a_2\disp\int_{\Gamma_{\mathtt{w}}}\chi\,\xi\,,
	\end{array}
\end{equation}

whereas, for each $ (\bz,\chi)\in\mathbf{V}\times\mathrm{Q} $, $\mathscr{O}_1^{F}(\bz;\cdot,\cdot):\mathbb{H}\times\mathbf{V}\rightarrow\mathbb{R}$, $\mathscr{O}_2^{F}(\bz;\cdot,\cdot):\mathbf{V}\times\mathbf{V}\rightarrow\mathbb{R} $, $\mathscr{O}_1^{C}(\bz;\cdot,\cdot):\mathbf{H}\times\mathrm{Q}\rightarrow\mathbb{R}$ and $\mathscr{O}_2^{C}(\chi;\cdot,\cdot):\mathrm{Q}\times\mathrm{Q}\rightarrow\mathbb{R}$ are the bilinear forms given by
\begin{equation}\label{def-c-tilde-c}
	\begin{array}{rcll}
		\mathscr{O}_1^F(\bz;\bta,\bv) &:=& \disp\frac{1}{\nu} \int_{\Omega}
		(\bdev(\bta)\,\bz)\cdot \bv \,,\quad & 	\mathscr{O}^{C}_1(\bz; \bet , \psi) := \disp\dfrac{1}{\kappa} \int_{\Omega}(\bet\,\bz)\, \psi\,,\\
		\mathscr{O}_2^F(\bz;\bw,\bv)&:=& \disp \mathtt{F}\int_{\Omega}
		|\bz|^{p-2}\,\bw\, \cdot\bv \,,\quad & \mathscr{O}^{C}_2(\chi;\vec{\varphi}, \vec{\psi})\,:=\, a_1 \disp\int_{\Gamma_{\mathtt{w}}}\chi\,\lambda\,\xi \,.
	\end{array}
\end{equation}
Finally, 
given $ \chi\in \mathrm{Q} $, the linear functionals $ \mathscr{F}^F_{\chi}\in \mathbb{H}^{'} $ and 
$\mathscr{F}^C \in  \mathrm{Q}^{'} $ are  defied by
\begin{equation}\label{def-f-tilde-f}
	\begin{array}{c}
		\mathscr{F}_{\chi}^F(\bta) \,:=\,  \langle \bta\bn\,, \, \mathbf{g}+\mathbf{g}^{\chi} \rangle_{\Gamma_{\mathtt{o}}^{c}}\quad\forall\,\bta \in \mathbb{H}\qan
		\mathscr{F}^C(\xi) \,:=\,-\tilde{a}_0 \disp\int_{\Gamma_{\mathtt{w}}}\xi\quad\forall\, \xi\in \mathrm{Q}\,.
	\end{array}
\end{equation}

\section{Continuous Solvability Analysis}\label{s:3}
We analyze the solvability of \eqref{eq:v1a-eq:v4} using \cite[Theorem 2.1]{gg-JCAM-2024}, \cite[Theorem 3.2]{cov-Calcolo-2020}, \cite[Lemma 2.4]{hjjkr-SIAMNA-2013}, and the Banach theorem.

%
\subsection{Preliminaries}

We first state the boundedness of the variational forms in \eqref{eq:v1a-eq:v4}. Using Cauchy-Schwarz, the bilinear forms \(\mathscr{A}^F\), \(\mathscr{A}^C\), and \(\mathscr{B}^F\) are bounded with constants \(c_{\mathscr{A}^F} = \nu^{-1}\), \(c_{\mathscr{A}^C} = \kappa^{-1}\), and \(c_{\mathscr{B}^F} = 1\).  
For any \(\bz \in \mathbf{V}\), from the definitions of \(\mathscr{O}_1^F\), \(\mathscr{O}_2^F\), and \(\mathscr{O}_1^C\) (cf. \eqref{def-c-tilde-c}) and \eqref{eq:bound.tr2}, we have
\begin{subequations}
	\begin{alignat}{2}
		\big| \mathscr{O}_1^F(\bz; \bze, \bv)\big| &\,\leq\, \dfrac{1}{\nu}\Vert \bz\Vert_{\mathbf{V}}\, \Vert\bze\Vert_{\mathbb{H}}\, \Vert\bv\Vert_{\mathbf{V}}\qquad\forall\, (\bze, \bv)\in \mathbb{H}\times\mathbf{V}\,,\label{eq1:bound.c.ct}\\	
		\big| \mathscr{O}_2^F(\bz; \bw, \bv)\big| &\,\leq\, |\Omega|^{(4-p)/4}\Vert \bz\Vert_{\mathbf{V}}^{p-2}\, \Vert\bw\Vert_{\mathbf{V}}\, \Vert\bv\Vert_{\mathbf{V}}\qquad\forall\, \bw, \bv\in \mathbf{V}\,,\label{eq2:bound.c.ct}\\	
		\big|\mathscr{O}^{C}_1(\bz; \bxi, \psi)\big|& \, \leq \, \dfrac{1}{\kappa}\Vert \bz\Vert_{\mathbf{V}}\, \Vert \bxi\Vert_{\mathbf{H}}\, \Vert\psi\Vert_{\mathrm{V}}\qquad\forall\, (\bxi, \psi)\in \mathbf{H}\times\mathrm{V}\,.\label{eq3:bound.c.ct}
	\end{alignat}
\end{subequations}
The boundedness of $\mathscr{B}^{C}$, $\mathscr{C}^{C}$, $\mathscr{O}_2^{C}$, and the functionals $\mathscr{F}_{\lambda}^F$, $\mathscr{F}^C$ is shown in the next lemma.
\begin{lemma}
	There exist positive constants $c_{\mathscr{B}^C}$, $c_{\mathscr{C}^{C}}$, $c_{\mathscr{O}_2^{C}}$, $c_{\mathscr{F}^F}$ and $c_{\mathscr{F}^C}$ such that
	\begin{subequations}\label{bund.BCO2}
		\begin{alignat}{2}
			\big|\mathscr{B}^{C}(\bet,\vec{\psi})\big|\,&\le\, c_{\mathscr{B}^{C}}\,\Vert \bet\Vert_{\mathbf{H}}\,\Vert\vec{\psi}\Vert_{\mathcal{V}}\qquad & \forall\, (\bet,\vec{\psi})\in\mathbf{H}\times\mathcal{V}\,,\label{bund.Bc}\\
			\big|\mathscr{C}^C(\vec{\phi},\vec{\psi})\big|\,&\le\, c_{\mathscr{C}^C}\,\Vert\vec{\psi}\Vert_{\mathcal{V}}\Vert\vec{\psi}\Vert_{\mathbf{V}}\qquad & \forall\, \vec{\phi},\vec{\psi}\in\mathcal{V}\,,\label{bund.Cc}\\
			\big|\mathscr{O}_2^C(\chi;\vec{\lambda},\vec{\xi})\big|\,&\le\, c_{\mathscr{O}_2^C}\,\Vert\chi\Vert_{\mathrm{Q}}\,\Vert\vec{\lambda}\Vert_{\mathcal{V}}\,\Vert\vec{\xi}\Vert_{\mathcal{V}} \qquad & \forall\,\vec{\chi},\vec{\lambda},\vec{\xi}\in\mathcal{V}\,,\label{bund.O2c}
		\end{alignat}
	\end{subequations}
	and
	\begin{subequations}\label{bund.F}
		\begin{alignat}{2}
			\big|\mathscr{F}_{\chi}^F(\bta)\big|\,&\le\, c_{\mathscr{F}^F}\,\Big(\Vert\mathbf{g}\Vert_{1/2,\Gamma_{\mathtt{o}}^c}+a_1\,\Vert\chi\Vert_{\mathrm{Q}}\Big)\Vert\bta\Vert_{\mathbb{H}}\qquad & \forall\, \bta\in\mathbb{H}\,,\label{bund.Ff}\\
			\big|\mathscr{F}^C(\vec{\psi})\big|\,&\le\, c_{\mathscr{F}^C}\,\Vert\vec{\psi}\Vert_{\mathcal{V}}& \forall\, \vec{\psi}\in\mathcal{V}\,.\label{bund.Fc}
		\end{alignat}
	\end{subequations}
\end{lemma}
\begin{proof}
Let \(\bet \in \mathbf{H}\) and \(\vec{\psi} = (\psi, \xi) \in \mathcal{V}\). Using \eqref{def.E0}, the second term in \(\mathscr{B}^C\) rewrites as
	\begin{equation*}
		\begin{array}{c}
			\langle \bet\cdot\bn,\,\xi\rangle_{\Gamma_{\mathtt{i}}^c}\,=\,\langle \bet\cdot\bn,\, E_{0,\Gamma_{\mathtt{i}}^c}(\xi)\rangle_{\Gamma}
			\,=\,\disp\int_{\Omega}\bet\cdot\nabla\tilde{\gamma}_{0}^{-1}\big(E_{0,\Gamma_{\mathtt{i}}^c}(\xi)\big)\,+\, \int_{\Omega}\tilde{\gamma}_{0}^{-1}\big(E_{0,\Gamma_{\mathtt{i}}^c}(\xi)\big)\,\div(\bet)\,,
		\end{array}
	\end{equation*}
where \(\tilde{\gamma}_{0}^{-1}: H^{1/2}(\Gamma) \to [H_0^1(\Omega)]^{\perp}\) is the right inverse of the trace operator \(\gamma_0: H^{1}(\Omega) \to H^{1/2}(\Gamma)\) (see \cite[Section 1.3.4]{g-springer-2014}).  
Using H\"older’s inequality, the continuous embedding \(i_4: H^1(\Omega) \hookrightarrow L^4(\Omega)\) (cf. \cite[Theorem 4.12]{af-book}), and the continuity of \(\tilde{\gamma}_0^{-1}\) (cf. \cite[Lemma 1.3]{g-springer-2014}) yields
	\begin{equation*}
		\begin{array}{c}
			\big|\langle \bet\cdot\bn,\,\xi\rangle_{\Gamma_{\mathtt{i}}^c}\big|
			\,\le\, \Vert\bet\Vert_{0,\Omega}\,\Vert\nabla\tilde{\gamma}_{0}^{-1}\big(E_{0,\Gamma_{\mathtt{i}}^c}(\xi)\big)\Vert_{0,\Omega}\,+\,\Vert\tilde{\gamma}_{0}^{-1}\big(E_{0,\Gamma_{\mathtt{i}}^c}(\xi)\big)\Vert_{0,4;\Omega}\,\Vert\div(\bet)\Vert_{0,4/3;\Omega}\\[1ex]
			\,\le\, \Vert\bet\Vert_{0,\Omega}\,\Vert\nabla\tilde{\gamma}_{0}^{-1}\big(E_{0,\Gamma_{\mathtt{i}}^c}(\xi)\big)\Vert_{0,\Omega}\,+\,\Vert i_{4}\Vert\, \Vert\tilde{\gamma}_{0}^{-1}\big(E_{0,\Gamma_{\mathtt{i}}^c}(\xi)\big)\Vert_{1,\Omega}\,\Vert\div(\bet)\Vert_{0,4/3;\Omega}\\[1ex]
			\,\le\, \Vert\bet\Vert_{0,\Omega}\,\Vert\, \Vert E_{0,\Gamma_{\mathtt{i}}^c}(\xi)\Vert_{1/2,\Gamma}\,+\,\Vert i_{4}\Vert\, \Vert E_{0,\Gamma_{\mathtt{i}}^c}(\xi)\Vert_{1/2,\Gamma}\,\Vert\div(\bet)\Vert_{0,4/3;\Omega}\,,
		\end{array}
	\end{equation*}
	which thanks to \eqref{def.normH12} concludes \eqref{bund.Bc} with $c_{\mathscr{B}^C}\,:=\,\max\{1,\Vert i_4\Vert\}$. 
	
For \(\mathscr{C}^C\), taking \(\vec{\phi}, \vec{\psi} \in \mathcal{V}\) and applying Cauchy-Schwarz and the Sobolev embedding \(H^{1/2}(\Gamma) \hookrightarrow L^2(\Gamma)\) with constant \(c_2\) (see \cite[Theorem B.46]{eg-book-2004}), we get
	\begin{equation}\label{kk}
		\begin{array}{c}
			\Big|a_{2}\disp\int_{\Gamma_{\mathtt{w}}}\chi\,\xi\Big|\,\le\, |a_2|\,\Vert\chi\Vert_{0,\Gamma_{\mathtt{i}}^c}\,\Vert\xi\Vert_{0,\Gamma_{\mathtt{i}}^c}\,=\, |a_2|\, \Vert E_{0,\Gamma_{\mathtt{i}}^{c}}(\chi)\Vert_{0,\Gamma}\, \Vert E_{0,\Gamma_{\mathtt{i}}^{c}}(\xi)\Vert_{0,\Gamma} \\
			\,\le\, |a_2|c_2^2\, \Vert E_{0,\Gamma_{\mathtt{i}}^{c}}(\chi)\Vert_{1/2,\Gamma}\, \Vert E_{0,\Gamma_{\mathtt{i}}^{c}}(\xi)\Vert_{1/2,\Gamma}\,.
		\end{array}
%
\end{equation}
This, together with \eqref{def.normH12}, yields \eqref{bund.Cc} with \( c_{\mathscr{C}^c} := |a_2| c_2^2 \).  
Similarly, applying H\"older’s inequality and the Sobolev embedding \( H^{1/2}(\Gamma) \hookrightarrow L^{3}(\Gamma) \) with constant \( c_3 \) gives \eqref{bund.O2c} with \( c_{\mathscr{O}_2^C} := a_1 c_3^3 \).
To prove \eqref{bund.Ff}, from the definition of \( \mathscr{F}_{\lambda}^F \) (see second column of \eqref{def-f-tilde-f}), we have
	\begin{equation}\label{lk}
		\mathscr{F}_{\lambda}^F(\bta) \,=\,  \langle \bta\bn\,, \, \mathbf{g} \rangle_{\Gamma_{\mathtt{o}}^{c}}\,+\, \langle \bta\bn\,, \, \mathbf{g}^{\chi} \rangle_{\Gamma_{\mathtt{o}}^{c}}\,.
	\end{equation}
Regarding the first term on the right hand side of \eqref{lk}, we define the extension $\mathbf{E}_{\Gamma_{\mathtt{o}}^{c}}(\mathbf{g})\,:=\,\bw|_{\Gamma}$, where $\bw\in\mathbf{H}^{1}(\Omega)$ is the unique solution to $\Delta \bw \,=\, \mathbf{0}  $ in $\Omega$, $\bw \,= \,\mathbf{g}$ on $ \Gamma_{\mathtt{o}}^{c}$, $\bw \cdot \mathbf{n} \,=\, 0$ on $\Gamma_{\mathtt{o}}$,
	which satisfies
	\begin{equation}\label{gg}
		\Vert \mathbf{E}_{\Gamma_{\mathtt{o}}^{c}}(\mathbf{g})\Vert_{1/2,\Gamma}\,\le\, C_1 \, \Vert\mathbf{g}\Vert_{1/2,\Gamma_{\mathtt{o}}^c}\,.
	\end{equation}
Recalling from the end of Section~\ref{sec1} that \( \mathbf{E}_{\Gamma_{\mathtt{o}}^{c}}(\mathbf{g}) \in \mathbf{H}^{1/2}(\Gamma) \), we note that there exist unique elements \( \bvarr_{\Gamma_{\mathtt{o}}^c} \in \mathbf{H}^{1/2}(\Gamma_{\mathtt{o}}^c) \) and \( \bvarr_{\Gamma_{\mathtt{o}}} \in \mathbf{H}^{1/2}(\Gamma_{\mathtt{o}}) \) such that $	\langle \bta\bn,\,\mathbf{E}_{\Gamma_{\mathtt{o}}^c}(\mathbf{g})\rangle_{\Gamma}:=\langle \bta\bn,\,\mathbf{E}_{\Gamma_{\mathtt{o}}^c}(\bvarr_{\Gamma_{\mathtt{o}}^c})\rangle_{\Gamma}+\langle \bta\bn,\,\mathbf{E}_{0,\Gamma_{\mathtt{o}}}(\bvarr_{\Gamma_{\mathtt{o}}})\rangle_{\Gamma}$ and $	\mathbf{E}_{\Gamma_{\mathtt{o}}^c}(\mathbf{g})|_{\Gamma_{\mathtt{o}}^c}=\bvarr_{\Gamma_{\mathtt{o}}^c}$, $\mathbf{E}_{\Gamma_{\mathtt{o}}^c}(\mathbf{g})|_{\Gamma_{\mathtt{o}}}=\bvarr_{\Gamma_{\mathtt{o}}}$.
	Additionally, knowing that  $\mathbf{g}\in\mathbf{H}^{1/2}(\Gamma_{\mathtt{o}}^c)$ and using uniqueness, we deduce that $\bvarr_{\Gamma_{\mathtt{o}}^c}\,=\,\mathbf{g}$, which means that
	\begin{equation}\label{ll}
		\langle \bta\bn,\,\mathbf{E}_{\Gamma_{\mathtt{o}}^c}(\mathbf{g})\rangle_{\Gamma}\,=\,\langle \bta\bn,\,\mathbf{g}\rangle_{\Gamma_{\mathtt{o}}^c}\qquad\forall\,\bta\in\mathbb{H}\,.
	\end{equation}
As a result, applying the previous reasoning to bound $\mathscr{B}^C$, now using the continuous embedding $\mathbf{i}_4\colon \mathbf{H}^1(\Omega) \hookrightarrow \mathbf{L}^4(\Omega)$, the right inverse $\boldsymbol{\tilde{\gamma}}_0^{-1}$ of the trace operator $\boldsymbol{\gamma}_0$, and its continuity, we obtain
	\begin{equation}\label{gf}
		\begin{array}{c}
			\big|\langle \bta\bn,\,\mathbf{g}\rangle_{\Gamma_{\mathtt{o}}^c}\big|\,=\, \big|\langle \bta\bn,\,\mathbf{E}_{\Gamma_{\mathtt{o}}^c}(\mathbf{g})\rangle_{\Gamma}\big|
			\,=\, \Big|\disp\int_{\Omega}\bta:\nabla\boldsymbol{\tilde{\gamma}}_0^{-1}\big(\mathbf{E}_{\Gamma_{\mathtt{o}}^c}(\mathbf{g})\big)\Big|\,+\, \Big|\int_{\Omega}\boldsymbol{\tilde{\gamma}}_0^{-1}\big(\mathbf{E}_{\Gamma_{\mathtt{o}}^c}(\mathbf{g})\big)\cdot\bdiv(\bta)\Big|\\
			\,\le\, \Vert\bta\Vert_{0,\Omega}\,\Vert\nabla\boldsymbol{\tilde{\gamma}}_0^{-1}\big(\mathbf{E}_{\Gamma_{\mathtt{o}}^c}(\mathbf{g})\big)\Vert_{0,\Omega}\,+\,\Vert\boldsymbol{\tilde{\gamma}}_0^{-1}\big(\mathbf{E}_{\Gamma_{\mathtt{o}}^c}(\mathbf{g})\big)\Vert_{0,4;\Omega}\,\Vert\bdiv(\bta)\Vert_{0,4/3;\Omega}\\[1ex]
			\,\le\, \max\{1,\Vert \mathbf{i}_4\Vert\}\,\Vert\mathbf{E}_{\Gamma_{\mathtt{o}}^c}(\mathbf{g})\Vert_{1/2,\Gamma}\,\Vert\bta\Vert_{\mathbb{H}}\,,
		\end{array}
	\end{equation}
	which, thanks to the boundedness of $\mathbf{E}_{\Gamma_{\mathtt{o}}^c}$ given in \eqref{gg}, yields 
	\begin{equation}\label{dd}
		\big|\langle \bta\bn,\,\mathbf{g}\rangle_{\Gamma_{\mathtt{o}}^c}\big|\,\le\,\max\{1,\Vert \mathbf{i}_4\Vert\}\,C_1 \,\Vert\mathbf{g}\Vert_{1/2,\Gamma_{\mathtt{o}}^c}\,\Vert\bta\Vert_{\mathbb{H}}\,.
	\end{equation}
To handle the second term on the right-hand side of \eqref{lk}, we define $\Xi_{\lambda} := E_{0,\Gamma_{\mathtt{i}}^c}(\lambda)a_1 \bn \in \mathbf{H}^{1/2}(\Gamma)$ so that its restriction to $\Gamma_{\mathtt{o}}^c$ equals $\mathbf{g}^{\lambda} \in \mathbf{H}^{1/2}(\Gamma_{\mathtt{o}}^c)$. As with \eqref{ll}, we find $\langle \bta\bn,\, \Xi_\lambda\rangle_{\Gamma} = \langle \bta\bn,\, \mathbf{g}^{\lambda} \rangle_{\Gamma_{\mathtt{o}}^c}$. Then, applying the $H_{00}^{1/2}(\Gamma_{\mathtt{i}}^c)$-norm from \eqref{def.normH12} and reasoning as in \eqref{gf}, we obtain
	\begin{equation}\label{df}
		\big|\langle \bta\bn,\, \mathbf{g}^{\lambda}\rangle_{\Gamma_{\mathtt{o}}^c}\big|\,\le\, a_1 \, \max\{1,\Vert\mathbf{i}_4\Vert\}\,\Vert\vec{\lambda}\Vert_{\mathcal{V}}\Vert\bta\Vert_{\mathbb{H}}\,.
	\end{equation}
Substituting \eqref{dd} and \eqref{df} into \eqref{lk} yields \eqref{bund.Ff} with $c_{\mathscr{F}} := \max\{1,\|\mathbf{i}_4\|\} \max\{C_1,1\}$.  
Similarly, we obtain the final bound with $c_{\mathscr{F}^C} := |\tilde{a}_0||\Gamma|^{1/2}c_2$ by following the same steps as in \eqref{kk}.
\end{proof}

Moreover, adapting the proofs of \cite[Lemmas 2.3 and 2.5]{g-springer-2014}, one shows that there exist constants $c_{1,\Omega}, c_{2,\Omega} > 0$ depending only on $\Omega$ and $\Gamma_{\mathtt{o}}$ such that
\begin{subequations}\label{kkk}
	\begin{alignat}{2}
		c_{1,\Omega}\Vert\bta_0\Vert_{0,\Omega}^2\,\le\, \Vert\bdev(\bta_0)\Vert_{0,\Omega}^2+\Vert\bdiv(\bta_0)\Vert_{0,4/3;\Omega}^2\qquad\forall\,\bta_0\in\mathbb{H}_{0}(\bdiv_{4/3};\Omega)\,,\label{kll}\\[1ex]
		c_{2,\Omega}\Vert\bta\Vert_{\mathbb{H}}^2\,\le\,\Vert\bta_0\Vert_{\mathbb{H}}^2\quad\forall\,\bta=\bta_0+d\mathbb{I}\in \mathbb{H}=\mathbb{H}_{\mathtt{o}}(\bdiv_{4/3};\Omega)\,.\label{kjj}
	\end{alignat}
\end{subequations}
\subsection{A fixed point strategy}\label{s:fix}
We begin by  reformulating the variational formulation \eqref{eq:v1a-eq:v4} as a fixed point problem.  For this, we define the operator  $ \mathscr{L}^F:  \mathbf{V} \times \mathrm{Q}\rightarrow \mathbb{H}\times\mathbf{V}$ by
\begin{equation}\label{eq:opr1}
	\mathscr{L}^F(\bz(t), \chi(t))\, =\,\big(\mathscr{L}^F_1(\bz(t), \chi(t)), \mathscr{L}^F_2(\bz(t), \chi(t))\big)\, :=\,
	(\bsi_{\star}(t), \bu_{\star}(t))\,\quad\text{for a.e.}\,  t\in \mathrm{J} \,,
\end{equation}
for each $ (\bz(t), \chi(t))\in \mathbf{V} \times \mathrm{Q} $, where $ (\bsi_{\star}(t), \bu_{\star}(t))\in \mathbb{H}\times\mathbf{V} $ 
satisfying
\begin{equation}\label{eq:subp1}
	\begin{array}{rcll}
		\hspace{-.5cm}	\mathscr{A}^F(\bsi_\star , \bta)+\mathscr{B}^F(\bta, \bu_\star)&=& \mathscr{F}_{\chi}^F(\bta)\quad \text{for a.e.}\,  t\in \mathrm{J}\,,\\
		\hspace{-.5cm}	\disp \int_{\Omega}\partial_t \bu_\star \cdot\bv-\mathscr{B}^F(\bsi_\star,\bv)-\mathscr{O}_1^F(\bz;\bsi_\star, \bv)+\mathscr{O}_2^F(\bz;\bu_\star, \bv)&=&0\quad \text{for a.e.}\,  t\in \mathrm{J}
		\,, 
	\end{array}
\end{equation}
with initial condition $ \bu_{\star}(\cdot,0)=\bu_0 $, for all $(\bta,\bv)\in\mathbb{H}\times\mathbf{V}$.

\smallskip

In turn, we also introduce the operator $ \mathscr{L}^C: \mathbf{V}\times \mathrm{Q}\rightarrow  \mathbf{H}\times\mathcal{V} $ as
\begin{equation}\label{eq:opr2}
	\mathscr{L}^C(\bz(t) , \chi(t))\, =\, \big(\mathscr{L}^C_1(\bz(t) , \chi(t)), \mathscr{L}^C_2(\bz(t) , \chi(t)), \mathscr{L}^C_3(\bz(t) , \chi(t))\big)\,:=\,
	(\brho_{\star}(t), \vec{\varphi}_{\star}(t))\,, 
\end{equation}
for a.e. $ t\in \mathrm{J} $ and each $ (\bz(t), \chi(t))\in\mathbf{V}\times\mathrm{Q} $,	where $ (\brho_{\star}(t), \vec{\varphi}_{\star}(t))=(\brho_{\star}(t), (\varphi_{\star}(t), \lambda_{\star}(t)))\in  \mathbf{H}\times\mathcal{V} $ satisfying
\begin{equation}\label{eq:subp2}
	\begin{array}{rcll}
		\hspace{-.3cm}	\mathscr{A}^{C}(\brho_\star , \bet)+\mathscr{B}^{C}(\bet,\vec{\varphi}_\star)&=&0\quad \text{for a.e.}\,  t\in \mathrm{J}\,,\\
			\hspace{-.3cm}	\disp \int_{\Omega}\partial_t \varphi_\star \, \psi-	\mathscr{B}^{C}(\brho_\star,\vec{\psi})+\mathscr{C}^{C}(\vec{\varphi}_\star, \vec{\psi})-\mathscr{O}^{C}_1 (\bz;\brho_\star,\vec{\psi})
		+\mathscr{O}^{C}_2(\chi;\lambda_\star, \vec{\psi})&=&\mathscr{F}^C(\vec{\psi})\quad \text{for a.e.}\,  t\in \mathrm{J}  \,,
	\end{array}
\end{equation}
with initial condition $ \varphi_{\star}(\cdot,0)=\varphi_{0} $ and each $(\bet,\vec{\psi})\in\mathbf{H}\times\mathcal{V}$. 

Thus, for a.e. $ t\in \mathrm{J} $, defining the operator $ \mathscr{T} :\mathbf{V}\times \mathrm{Q}\rightarrow \mathbf{V}\times \mathrm{Q} $ as
\begin{equation}\label{eq:def.opr.E}
	\mathscr{T}(\bz(t) , \chi(t))\, :=\,  \big( \mathscr{L}_2^F(\bz(t) , \mathscr{L}_3^T(\bz(t), \chi(t))),\, \mathscr{L}_3^T(\bz(t), \chi(t))\big)\quad\forall\, (\bz(t), \chi(t))\in \mathbf{V}\times \mathrm{Q} \,,
\end{equation}

we find out that solving  \eqref{eq:v1a-eq:v4} is equivalent to finding a fixed point of $ \mathscr{T} $, that is, given initial data
$ (\bu_{0}, \varphi_{0})\in \mathbf{V}\times\mathrm{V} $, find $ (\bu(t) , \lambda(t))\in \mathbf{V}\times \mathrm{Q} $ such that
\begin{equation}\label{eq:fixp}
	\begin{array}{l}
		\mathscr{T}(\bu , \lambda)\, =\,(\bu , \lambda)  \qquad \text{for a.e.}\,  t\in \mathrm{J}\,.
	\end{array}
\end{equation}
\subsection{Well-posedness of the uncoupled problems}\label{sec.2.2}
We aim to show that \( \mathscr{T} \) has a unique fixed point by first proving the well-posedness of the uncoupled problems \eqref{eq:subp1} and \eqref{eq:subp2}, i.e., verifying that \( \mathscr{L}^F \) and \( \mathscr{L}^C \) are well-defined. From the definitions of \( \mathscr{A}^F \) and \( \mathscr{A}^C \), we note that both are elliptic with constants \( \alpha_F := 1/\nu \) and \( \alpha_C := 1/\kappa \), respectively:
\begin{equation}\label{eq:ell.a.at}
		\mathscr{A}^F(\bta , \bta)\, \geq\, \alpha_F \, \Vert \bdev(\bta)\Vert_{0,\Omega}^{2}\quad\forall\,\bta\in \mathbb{H}\qan
		\mathscr{A}^{C}(\bet , \bet)\, \geq\, \alpha_C \, \Vert \bet\Vert_{0,\Omega}^{2}\quad\forall\,\bet\in \mathbf{H}\,. 
\end{equation}

Thus, it remains to verify the continuous inf-sup condition for \( \mathscr{B}^F \) and \( \mathscr{B}^C \), adapting the argument from \cite[Lemma 3.3]{cgo-NMPDE-2021} to our setting.
\begin{lemma}\label{l_inf.b.bt}
	There exist constants $ \beta_F >0 $ and $ \beta_{C}>0 $ such that there hold
	\begin{subequations}
		\begin{alignat}{2}
			\sup_{\mathbf{0}\neq \bta\in \mathbb{H}}\dfrac{\mathscr{B}^F(\bta, \bv)}{\Vert \bta\Vert_{\mathbb{H}}}\,\geq\,\beta_F\,\Vert \bv\Vert_{\mathbf{V}}\qquad\forall\, \bv\in \mathbf{V} \,,\label{eq:inf.bi}\\
			\sup_{\mathbf{0}\neq \bet\in \mathbf{H}}\dfrac{\mathscr{B}^C(\bet, \vec{\psi})}{\Vert \bet\Vert_{\mathbf{H}}}\,\geq\,\beta_{C}\,\Vert \vec{\psi}\Vert_{\mathcal{V}}\qquad\forall\, \psi\in \mathcal{V} \,,\label{eq:inf.bit}
		\end{alignat}
	\end{subequations}
\end{lemma}
\begin{proof}
	Give $\bv\in \mathbf{V}$, we set $ \widetilde{\bv}:=|\bv|^{2}\bv$, and let $\bw\in \mathbf{H}^{1}(\Omega) $ be the unique solution to the following problem
	\[
	\Delta\bw \,=\, \widetilde{\bv}\qin\Omega\,,\quad(\nabla\bw)\bn\,=\, \mathbf{0}\qon\Gamma_{\mathtt{o}}\,,\qan \bw\,=\,\mathbf{0}\qon\Gamma_{\mathtt{o}}^c \,.
	\]
Defining $\widetilde{\bta}:=-\nabla\bw \in \mathbb{L}^2(\Omega)$ with $\widetilde{\bta}\in \mathbb{H}$ and using an argument as in \cite[Lemma 3.3]{cgo-NMPDE-2021}, we obtain \eqref{eq:inf.bi} with 
$\beta_F := 1/(c_p \|\mathbf{i}_4\| + 1)$, where $c_p$ is a Poincaré constant depending on $|\Omega|$.  
For the second part, letting $(\bet, \vec{\psi}) \in \mathbf{H} \times \mathcal{V}$, we proceed similarly to $\mathscr{B}^F$, yielding
	\begin{equation}\label{fd}
		\sup_{\mathbf{0}\neq \bet\in \mathbf{H}}\dfrac{\mathscr{B}^C(\bet, \vec{\psi})}{\Vert \bet\Vert_{\mathbf{H}}}\,\geq\,	\sup_{\mathbf{0}\neq \bet\in \mathbf{H}}\dfrac{\mathscr{B}^C(\bet, (\psi,0))}{\Vert \bet\Vert_{\mathbf{H}}}\,\ge\, \beta_{1,C}\Vert\psi\Vert_{\mathrm{V}}\,,
	\end{equation}
	where $\beta_{1,C}\,=\,1/(c_p\Vert\mathrm{i}_4\Vert+1)$. 
Given $\upsilon \in \mathrm{H}_{00}^{-1/2}(\Gamma_{\mathtt{o}})$,  we let $z\in\mathrm{H}^{1}(\Omega)$ be the solution to
	\[
	\Delta z\,=\, 0\qin\Omega\,,\quad z \,=\,0\qon\Gamma_{\mathtt{o}}^{c}\qan \nabla z\cdot\bn\,=\,\upsilon \qon\Gamma_{\mathtt{o}}\,,
	\]
Set $\tilde{\bet} := \nabla z \in \mathbf{H}$. By a similar argument to \cite[Theorem 2.1]{bg-NMPDE-2003}, there exists $\beta_{2,C}$ such that
	\begin{equation}\label{part2.difBc}
		\sup_{\mathbf{0}\neq \bet\in \mathbf{H}}\dfrac{\mathscr{B}^C(\bet, \vec{\psi})}{\Vert \bet\Vert_{\mathbf{H}}}\,\geq\,	\sup_{\mathbf{0}\neq \bet\in \mathbf{H}}\dfrac{\mathscr{B}^C(\bet, (0,\xi))}{\Vert \bet\Vert_{\mathbf{H}}}\,\geq\,	\beta_{2,C}\Vert\xi\Vert_{\mathrm{Q}}\,.
	\end{equation}
	The preceding result, together with \eqref{fd}, leads to \eqref{eq:inf.bit} with $\beta_{C}\,:=\,\frac{1}{2}\min\big\{\beta_{1,C},\beta_{2,C}\big\}$.
\end{proof}

Next, we adopted the approach taken from \cite{hjjkr-SIAMNA-2013} to the uncoupled problems \eqref{eq:subp1} and \eqref{eq:subp2}. 
We first need to introduce some notations:  Let
\begin{equation}\label{n1}
	\Big\{\big(\bta_k,\bv_k\big):\quad k\in \mathbb{N}  \Big\} \qan \Big\{\big(\bet_k,(\psi_k,\xi_k)\big):\quad k\in \mathbb{N}  \Big\}\,, 
\end{equation}
be a Hilbert bases
of $ \mathbb{H}\times\mathbf{V}$ and $ \mathbf{H}\times(\mathrm{V}\times\mathrm{Q}) $, respectively and denote $ \mathbb{H}_n \subseteq\mathbb{H} $, $\mathbf{V}_m \subseteq \mathbf{V} $,
and
$ \mathbf{H}_{\tilde{n}} \subseteq\mathbf{H} $, $\mathrm{V}_{\tilde{m}} \subseteq \mathrm{V} $, $ \mathrm{Q}_{\tilde{s}} \subseteq\mathrm{Q}$ 
as finite dimensional subspaces spanned by $ \{\bta_i\}_{i=1}^{n} $, $ \{\bv_j\}_{j=1}^{m} $
and $ \{\bet_k\}_{k=1}^{\tilde{n}} $, $ \{\psi_l\}_{l=1}^{\tilde{m}} $, $ \{\xi_p\}_{p=1}^{\tilde{s}} $ respectively. For each positive integers $ n, m, s, \tilde{n}, \tilde{m}, \tilde{s}  $, and $ (\bz , \chi)\in \mathbf{V}_m \times\mathrm{Q}_{\tilde{s}} $ we let $ \bsi_{\star,n}: \mathrm{J}\rightarrow \mathbb{H}_{n} $, $ \bu_{\star,m}: \mathrm{J}\rightarrow \mathbf{V}_{m} $ and $ \brho_{\star,s}: \mathrm{J}\rightarrow \mathbf{H}_{s} $, $ \vec{\varphi}_{\star,\tilde{m}\tilde{s}}=(\varphi_{\star,\tilde{m}},\lambda_{\star,\tilde{s}}): \mathrm{J}\rightarrow\mathcal{V}_{\tilde{m}\tilde{s}} =\mathrm{V}_{\tilde{m}}\times\mathrm{Q}_{\tilde{s}} $ be the solution of the following problems, respectively,
\begin{equation}\label{eq:subp1.a}
	\begin{array}{rcll}
		\hspace{-.3cm}	\mathscr{A}^F(\bsi_{\star,n} , \bta_i)+\mathscr{B}^F(\bta_i, \bu_{\star,m})&=& \mathscr{F}_{\chi}^F(\bta_i)\quad \text{for a.e.}\,  t\in \mathrm{J}
		\,,
		\\
	\hspace{-.5cm}	\disp \int_{\Omega}\partial_t \bu_{\star,m} \cdot\bv_j-\mathscr{B}^F(\bsi_{\star,n},\bv_{j})-\mathscr{O}_1^F(\bz;\bsi_{\star,n}, \bv_j)
		\,+\,\mathscr{O}_2^F(\bz;\bu_{\star,m}, \bv_j)&=&0  \quad \text{for a.e.}\,  t\in \mathrm{J} \\
		\hspace{-.3cm}	\disp \int_{\Omega}\bu_{\star,m}^{0}\cdot \bv_j&=&\disp \int_{\Omega}\bu_{0}\cdot\bv_j \quad \text{for a.e.}\,  t\in \mathrm{J}\,,
	\end{array}	
\end{equation}
for any $ i=1,\cdots,n $, $ j=1,\cdots,m $ and
\begin{equation}\label{eq:subp2.a}
	\begin{array}{rcll}
		\mathscr{A}^{C}(\brho_{\star,\tilde{n}} , \bet_i)+\mathscr{B}^{C}(\bet_i,\vec{\varphi}_{\star,\tilde{m}\tilde{s}})&=& 0\quad \text{for a.e.}\,  t\in \mathrm{J}
		\,,\\
		\disp \int_{\Omega}\partial_t \varphi_{\star,\tilde{m}} \, \psi_j-	\mathscr{B}^{C}(\brho_{\star,\tilde{n}},\vec{\psi}_{j,k})+\mathscr{C}^{C}(\vec{\varphi}_{\star,\tilde{m}\tilde{s}}, \vec{\psi}_{j,k})&&\\
		-\mathscr{O}^{C}_1 (\bz;\brho_{\star,\tilde{n}},\vec{\psi}_{j,k})+\mathscr{O}^{C}_2(\chi;\lambda_{\star,\tilde{s}}, \vec{\psi}_{j,k})&=&\mathscr{F}^C(\vec{\psi}_{j,k})\quad \text{for a.e.}\,  t\in \mathrm{J}  \,,\\
		\disp \int_{\Omega}\varphi_{\star,\tilde{m}}^{0}\, \psi_j&=&\disp \int_{\Omega} \varphi_{0}\, \psi_j	\,,
	\end{array}	
\end{equation}
for any $ i=1,\cdots,\tilde{n} $, $ j=1,\cdots,\tilde{m} $, $ k=1,\cdots,\tilde{s} $.
\begin{lemma}
Given \( (\bz, \chi) \in \mathbf{V}_m \times \mathrm{Q}_{\tilde{s}} \), and for each \( n, m, s \in \mathbb{N} \), there exist unique solutions \( (\bsi_{\star,n}, \bu_{\star,m}) \) and \( (\brho_{\star,\tilde{n}}, \vec{\varphi}_{\star,\tilde{m}\tilde{s}}) \) to problems~\eqref{eq:subp1.a} and~\eqref{eq:subp2.a}, respectively.
\end{lemma}
\begin{proof}
	First, we introduce the following notations
	\begin{equation}\label{eq:not1}
		\begin{array}{c}
(\mathbf{A}_n)_{ij} \, :=\, \mathscr{A}^F(\bsi_j, \bta_i),\,\,
(\mathbf{B}_{nm})_{ik}\,:=\, \mathscr{B}^F(\bta_i, \bu_{k}),\,\,(\mathbf{O}_{1,mn})_{ik}\,:=\, \mathscr{O}_1^F(\bz;\bsi_{k},\bv_i)\ \,,\\
(\mathbf{D}_m)_{kl}:= \disp\int_{\Omega}\bu_l \cdot \bv_k ,\,\,
(\mathbf{O}_{2,m})_{kl}:= \mathscr{O}_2^F(\bz;\bu_{l},\bv_k)\,,\,\,
(\mathbf{F}_n(t))_i := \mathscr{F}_{\chi}^F(\bta_i),\,\, (\mathbf{U}_0)_k  := \disp\int_{\Omega}\bu_{0}\cdot \bv_k \,\,,
		\end{array}
	\end{equation}
for $i,j=1,\cdots, n,\,\,$  $k,l=1,\cdots,m$	and
	\begin{equation}\label{eq:not2}
		\begin{array}{c}
(\widetilde{\mathbf{A}}_{\tilde{n}})_{ij} \, :=\, \mathscr{A}^{C}(\brho_j, \bet_i),\,\,
(\widetilde{\mathbf{B}}_{1,\tilde{n}\tilde{m}})_{ik}\,:=\, \mathscr{B}^{C}_1(\bet_i, \varphi_k),
\,\,
(\widetilde{\mathbf{D}}_{\tilde{m}})_{kl}\,:=\, \disp\int_{\Omega}\varphi_l \, \psi_k ,\quad
\\[1ex]
(\widetilde{\mathbf{O}}_{1,\tilde{m}\tilde{n}})_{ki}\,:=\, \mathscr{O}^{C}_{1}(\bz;\brho_i,\psi_k),\quad(\widetilde{\mathbf{C}}_{\tilde{s}})_{pq}\,=\, \mathscr{C}^{C}(\lambda_{q}, \xi_p),\,\,\,
(\widetilde{\mathbf{O}}_{2,\tilde{s}})_{pq}\,=\, \mathscr{O}^{C}_2(\chi;\lambda_{q}, \xi_p)\,,\\
(\mathbf{F}_{1,\tilde{n}}(t))_i \,:=\, \mathscr{F}_1^C(\bet_i),\,\,			 \,\,
(\widetilde{\mathbf{B}}_{2,\tilde{n}\tilde{s}})_{iq}\,:=\, \mathscr{B}^{C}_2(\bet_i, \lambda_q),\,\,
(\mathbf{F}_{2,\tilde{s}}(t))_p \,:=\,\mathscr{F}_2^C(\xi_p),\,\,(\boldsymbol{\Phi}_0)_l \, :=\, \disp\int_{\Omega}\varphi_{0}\, \psi_l \,,
		\end{array}
	\end{equation}
	for $i,j = 1,\cdots,\tilde{n},\,\,$ $k,l=1,\cdots,\tilde{m},\,\,$ $p,q=1,\cdots,\tilde{s} $, 
	also we denote by $ \boldsymbol{\Sigma}_{n}(t) $, $ \mathbf{U}_m(t) $, and $ \boldsymbol{\Upsilon}_{\tilde{n}}(t) $, $ \boldsymbol{\Phi}_{\tilde{m}}(t) $, $ \boldsymbol{\Lambda}_{\tilde{s}}(t) $ the vectors of degrees of freedom of $ \bsi_{\star,n}(t) $, $ \bu_{\star,m}(t) $ and $ \brho_{\star,\tilde{n}}(t) $, $ \varphi_{\star,\tilde{m}}(t) $, $ \lambda_{\star,\tilde{s}}(t) $, respectively,  with respect to the bases $ \{\bta_i\}_{i=1}^{n} $, $ \{\bv_j\}_{j=1}^{m} $ and $ \{\bet_k\}_{k=1}^{\tilde{n}} $, $ \{\psi_l\}_{l=1}^{\tilde{m}} $, $ \{\xi_p\}_{p=1}^{\tilde{s}} $, respectively.
	Hence, thanks to the above notations, the matrix form of problems \eqref{eq:subp1.a} and \eqref{eq:subp2.a} may be rewritten as
	\begin{subequations}
		\begin{align}
\mathbf{A}_n\,	\boldsymbol{\Sigma}_{n}(t)+\mathbf{B}_{nm}\,\mathbf{U}_{m}(t)&\,=\, \mathbf{F}_n(t)\,,\label{eq:sub.1a}\\
\mathbf{D}_m\,\dfrac{d\mathbf{U}_m (t)}{dt} -\left(\mathbf{B}_{1,nm}^{\mathtt{t}}+\mathbf{O}_{1,mn}\right)\,\boldsymbol{\Sigma}_{n}(t)-\mathbf{O}_{2,ms}\boldsymbol{\vartheta}_{s}(t)+\mathbf{O}_{3,m}\mathbf{U}_m(t)
&\,=\,	0	\,,\label{eq:sub.1b}\\
\mathbf{U}_m(0) &\,=\, \mathbf{U}_0\,,\label{eq:sub.1c}
		\end{align}
	\end{subequations}
	and
	\begin{subequations}
		\begin{align}
\widetilde{\mathbf{A}}_{\tilde{n}}\,	\boldsymbol{\Upsilon}_{\tilde{n}}(t)+\widetilde{\mathbf{B}}_{1,\tilde{n}\tilde{m}}\,\boldsymbol{\Phi}_{\tilde{m}}(t)+\widetilde{\mathbf{B}}_{2,\tilde{n}\tilde{s}}\,\boldsymbol{\Lambda}_{\tilde{s}}(t)&\,=\, \mathbf{0}\,,\label{eq:sub.2a}\\
\widetilde{\mathbf{D}}_{\tilde{m}}\,
\dfrac{d\boldsymbol{\Phi}_{\tilde{m}} (t)}{dt} -\left(\widetilde{\mathbf{B}}_{1,\tilde{n}\tilde{m}}^{\mathtt{t}}+\widetilde{\mathbf{O}}_{1,\tilde{m}\tilde{n}}\right)\,\boldsymbol{\Upsilon}_{\tilde{n}}(t)&\,=\,\textbf{0}\,,	\label{eq:sub.2b}\\
\widetilde{\mathbf{B}}_{2,\tilde{n}\tilde{s}}^{\mathtt{t}}\, \boldsymbol{\Upsilon}_{\tilde{n}}(t)+
\left(\widetilde{\mathbf{O}}_{2,\tilde{s}}+\widetilde{\mathbf{C}}_{\tilde{s}}\right)\,\boldsymbol{\Lambda}_{\tilde{s}}(t)
&\,=\,	\widetilde{\mathbf{F}}_{\tilde{s}}(t)\label{eq:sub.2d}\\
\boldsymbol{\Phi}_r(0) &\,=\, \boldsymbol{\Phi}_0\,.\label{eq:sub.2c}
		\end{align}
	\end{subequations}
Since $ \mathbf{A}_n $ and $ \widetilde{\mathbf{A}}_{\tilde{n}} $ are symmetric positive definite (by \eqref{eq:ell.a.at}), they are invertible. Hence, \eqref{eq:sub.1a} and \eqref{eq:sub.2a} imply
	\begin{subequations}
		\begin{alignat}{2}
		\boldsymbol{\Sigma}_{n}(t)
		\,=\,\mathbf{A}_n^{-1}\Big(\mathbf{F}_n(t)-\mathbf{B}_{nm}\,\mathbf{U}_{m}(t)
		\Big)\label{eq:sig}	\qan\\	
		\boldsymbol{\Upsilon}_{\tilde{n}}(t)\,=\,\widetilde{\mathbf{A}}_{\tilde{n}}^{-1}\,\left( \widetilde{\mathbf{F}}_{1,\tilde{n}}(t)-\widetilde{\mathbf{B}}_{1,\tilde{n}\tilde{m}}\,\boldsymbol{\Phi}_{\tilde{m}}(t)-\widetilde{\mathbf{B}}_{2,\tilde{n}\tilde{s}}\,\boldsymbol{\Lambda}_{\tilde{s}}(t)\right)\,,\label{eq:rho}
		\end{alignat}
	\end{subequations}
	
	Now, substituting \eqref{eq:sig} and \eqref{eq:rho} into \eqref{eq:sub.1b} and \eqref{eq:sub.2b}-\eqref{eq:sub.2d}, respectively, we deduce that there exist square matrices $\widehat{\mathbf{A}}_m^F$, $\widehat{\mathbf{A}}_{\tilde{m}}^C$ and vectors $\widehat{\mathbf{F}}_m^F$, $\widehat{\mathbf{F}}_{\tilde{m}}^C$ satisfying
	\begin{subequations}
		\begin{alignat}{2}
			\dfrac{d\mathbf{U}_m (t)}{dt}+\widehat{\mathbf{A}}_{m}^F\mathbf{U}_{m}(t)\,=\,\widehat{\mathbf{F}}_{m}^F(t)\,,\label{eq:sub.1b.}\\
			\dfrac{d\boldsymbol{\Phi}_{\tilde{m}} (t)}{dt}+\widehat{\mathbf{A}}_{\tilde{m}}^C\boldsymbol{\Phi}_{\tilde{m}}(t)\,+\,\widehat{\mathbf{B}}_{\tilde{m}\tilde{s}}^C\boldsymbol{\Lambda}_{\tilde{s}}(t)\,=\,\mathbf{0}\,.\label{eq:sub.2b.}
		\end{alignat}
	\end{subequations}	
Substituting \eqref{eq:rho} into \eqref{eq:sub.2d} gives the matrix $\widehat{\mathbf{D}}_{\tilde{s}}^C$ and vector $\widehat{\mathbf{G}}_{\tilde{s}}^C$ such that
	\begin{equation}\label{eq:sub.2b-}
		\boldsymbol{\Lambda}_{\tilde{s}}(t)\,=\, \widehat{\mathbf{G}}_{\tilde{s}}^C (t)+\widehat{\mathbf{D}}_{\tilde{s}}^C\,\boldsymbol{\Phi}_{\tilde{m}}(t)\,.
	\end{equation}
	
	Thus, using \eqref{eq:sub.1b.} and substituting \eqref{eq:sub.2b-} into \eqref{eq:sub.2b.}, we obtain for a.e. $t \in \mathrm{J}$
	\begin{equation}\label{n2}
			\dfrac{d\mathbf{U}_m (t)}{dt}+\widehat{\mathbf{A}}_{m}^F\mathbf{U}_{m}(t)\,=\,\widehat{\mathbf{F}}_{m}^F(t)\qan
			\dfrac{d\boldsymbol{\Phi}_{\tilde{m}} (t)}{dt}+\big(\widehat{\mathbf{A}}_{\tilde{m}}^C+\widehat{\mathbf{B}}_{\tilde{m}\tilde{s}}^C\widehat{\mathbf{D}}_{\tilde{s}}^C\big)\boldsymbol{\Phi}_{\tilde{m}}(t)\,=\,-\widehat{\mathbf{G}}_{\tilde{s}}^C (t)
	\end{equation}
	The two systems above are linear first-order ODEs with initial conditions \eqref{eq:sub.1c} and \eqref{eq:sub.2c}, respectively, ensuring the existence of unique solutions $ \mathbf{U}_m(t)\in (C(\mathrm{J}))^{m} $ and $ \boldsymbol{\Phi}_{\tilde{m}}(t)\in (C(\mathrm{J}))^{\tilde{m}} $ with $ \partial_t\mathbf{U}_m \in (L^{2}(\mathrm{J}))^{m}$ and $ \partial_t\boldsymbol{\Phi}_{\tilde{m}}\in (L^{2}(\mathrm{J}))^{\tilde{m}} $ (see e.g. \cite[Theorem 3.8-2]{c-SIAM-2025}). From \eqref{eq:sig}, \eqref{eq:rho} and \eqref{eq:sub.2b-}, the functions $ \boldsymbol{\Sigma}_{n}\in (C(\mathrm{J}))^{n} $, $ \boldsymbol{\Upsilon}_{\tilde{n}}\in (C(\mathrm{J}))^{\tilde{n}} $ and $\boldsymbol{\Lambda}_{\tilde{s}}\in (C(\mathrm{J}))^{\tilde{s}}$ are obtained, with their time derivatives in $(L^{2}(\mathrm{J}))^{n} $,  $(L^{2}(\mathrm{J}))^{\tilde{n}} $ and $(L^{2}(\mathrm{J}))^{\tilde{s}} $, respectively. Consequently, the pairs 	$ (\bsi_{\star,n},\bu_{\star,m}) $ and $ (\brho_{\star,\tilde{n}},\vec{\varphi}_{\star,\tilde{m}\tilde{s}}) $ uniquely solve
	\eqref{eq:subp1.a} and \eqref{eq:subp2.a}.
\end{proof}
In the next step, we derive suitable a {\it priori} estimates for the solution of \eqref{eq:subp1.a}.
\begin{lemma}\label{l_help}
	Assume that for a.e $ t\in \mathrm{J} $,  $ (\bz , \chi)\in \mathbf{V}_m \times\mathrm{Q}_{\tilde{s}} $ be given and satisfy
	\begin{equation}\label{eq:as.z.var}
		\Vert\bz(t)\Vert_{\mathbf{V}}\,\le\,\max\Big\{ \left(\dfrac{\beta_F^2\nu}{8}|\Omega|^{(p-4)/4}\right)^{1/(p-2)},\,\dfrac{\beta_F\nu}{2\sqrt{2}} \Big\}\,,
	\end{equation}
	Then, there exists a constant $ \mathcal{C}_1 $, depending on $\nu$, $\beta_F$, $\Omega$, $c_{\mathscr{F}^F}$ but independent of $ n $ and $ m $, such that 
	\begin{equation}\label{ky}
		\begin{array}{c}
\hspace{-.2cm}\Vert\bsi_{\star,n}\Vert_{L^{2}(\mathrm{J};\mathbb{H})}+\Vert\bdev(\bsi_{\star,n})\Vert_{L^{\infty}(\mathrm{J};\mathbb{L}^{2})}+\Vert\bu_{\star,m}\Vert_{L^{\infty}(\mathrm{J};\mathbf{L}^{2}(\Omega))}+\Vert\partial_t\bu_{\star,m}\Vert_{L^{2}(\mathrm{J};\mathbf{L}^{2}(\Omega))}	+\Vert\bu_{\star,m}\Vert_{L^{2}(\mathrm{J};\mathbf{V})}\\[1ex]
		\,\leq\, \mathcal{C}_1\,
\Big\{\Vert \bu_{\star}(0)\Vert_{\mathbf{H}^{1}(\Omega)}+\Vert\partial_t\mathbf{g}\Vert_{L^{2}(\mathrm{J};H^{1/2}(\Gamma_{\mathtt{o}}^c))}+\Vert\mathbf{g}\Vert_{L^{2}(\mathrm{J};H^{1/2}(\Gamma_{\mathtt{o}}^c))}\\[1ex]
\,+\, \Vert\chi\Vert_{L^{2}(\mathrm{J};\mathrm{Q})} +\Vert\partial_t\chi\Vert_{L^{2}(\mathrm{J};\mathrm{Q})} \Big\}\,:=\, \mathcal{C}_1\mathcal{N}_F(\bu_{\star}(0),\mathbf{g},\chi)
		\end{array}
	\end{equation}
\end{lemma}
\begin{proof}
	%
	Letting $n,m \ge 1$ and taking $(\bta_i, \bv_j) := (\bsi_{\star,n}(t), \bu_{\star,m}(t)) \in \mathbb{H}_n \times \mathbf{V}_m$ in \eqref{eq:subp1.a}, then summing the first two rows and applying the coercivity (cf. \eqref{eq:ell.a.at}) and continuity bounds for $\mathscr{O}_1^F$, $\mathscr{O}_2^F$ (cf. \eqref{eq1:bound.c.ct}, \eqref{eq2:bound.c.ct}) and $\mathscr{F}_{\chi}^F$ (cf. \eqref{bund.Ff}), we get
	\begin{equation}\label{eq:j1}
		\begin{array}{l}
			\dfrac{1}{2}\partial_t \Vert \bu_{\star,m}\Vert_{0,\Omega}^{2}\,+\,\dfrac{1}{\nu}\, \Vert \bdev(\bsi_{\star,n})\Vert_{0,\Omega}^{2}\,\leq\, c_{\mathscr{F}^F}\left(\Vert\mathbf{g}\Vert_{1/2,\Gamma_{\mathtt{o}}^c}+a_1\Vert\chi\Vert_{\mathrm{Q}}\right)\,\Vert\bsi_{\star,n}\Vert_{\mathbb{H}}\\
			\qquad\qquad\qquad	\qquad\qquad\qquad\,+\, \left\{\dfrac{\Vert\bz\Vert_{\mathbf{V}}}{\nu}\Vert\bdev(\bsi_{\star,n})\Vert_{0,\Omega}+|\Omega|^{(4-p)/4}\, \Vert\bz\Vert_{\mathbf{V}}^{p-2}\Vert\bu_{\star,m}\Vert_{\mathbf{V}}\right\}\Vert\bu_{\star,m}\Vert_{\mathbf{V}}\,.
		\end{array}
	\end{equation}
	The inf-sup condition of $ \mathscr{B}^F $ (cf. \eqref{eq:inf.bi}), along with the first row of equation \eqref{eq:subp1.a} and the continuity bounds of  $ \mathscr{F}_{\chi}^F $ and $ \mathscr{A}^F $, allows us to derive
	\begin{equation}\label{eq:h3}
		\dfrac{\beta_F^{2}\nu}{2}\,\Vert \bu_{\star,m}\Vert_{\mathbf{V}}^{2}\,\leq \, 2\nu c_{\mathscr{F}^F}^2\left(\Vert\mathbf{g}\Vert_{1/2,\Gamma_{\mathtt{o}}^c}+a_1\Vert\chi\Vert_{\mathrm{Q}}\right)^{2}+\dfrac{1}{2\nu}\Vert\bdev(\bsi_{\star,n})\Vert_{0,\Omega}^{2}\,.
	\end{equation}	
	Combining with \eqref{eq:j1}, applying Young's inequality and assumption \eqref{eq:as.z.var}, then integrating over $[0,t]$, we deduce a constant $\mathcal{C}_1 > 0$ depending on $c_{\mathscr{F}^F}$, $\nu$, $\beta_F$, and $C(\epsilon)$ such that
	\begin{equation}\label{eq:f4}
		\begin{array}{c}
			\Vert\bu_{\star,m}(t)\Vert_{0,\Omega}^{2}+\disp\int_{0}^{t}\Vert \bdev(\bsi_{\star,n}(s))\Vert_{0,\Omega}^{2}\, ds +\disp\int_{0}^{t}\Vert \bu_{\star,m}(s)\Vert_{\mathbf{V}}^{2}\, ds
			\,\leq\,\mathcal{C}_1 \, \bigg\{\Vert\bu_{\star,m}(0)\Vert_{0,\Omega}^{2}\\
			\,+\,\disp\int_{0}^{t}\left(\Vert\mathbf{g}(s)\Vert_{1/2,\Gamma_{\mathtt{o}}^c}+a_1\Vert\chi(s)\Vert_{\mathrm{Q}}\right)^2 \,ds  +\epsilon\disp\int_{0}^{t}\Vert\bsi_{\star,n}(s)\Vert_{\mathbb{H}}^{2}\,ds \bigg\}\,,
		\end{array}
	\end{equation}
	
	On the other hand, using $ \bdiv(\mathbb{H}) = \mathbf{V}' $ and the second row of \eqref{eq:subp1.a}, we deduce that
	\begin{equation*}
		\begin{array}{c}
			\big\Vert\bdiv(\bsi_{\star,n})\big\Vert_{0,4/3;\Omega}
	\, \leq\, |\Omega|^{1/4}\Vert\partial_t \bu_{\star,m}\Vert_{0,\Omega}+\dfrac{1}{\nu}\Vert\bz\Vert_{\mathbf{V}}\Vert\bdev(\bsi_{\star,n})\Vert_{0,\Omega}+|\Omega|^{(4-p)/4}\, \Vert\bz\Vert_{\mathbf{V}}^{p-2}\Vert\bu_{\star,m}\Vert_{\mathbf{V}}\,.
		\end{array}
	\end{equation*}
	
	The above inequality and \eqref{eq:as.z.var} ensure the existence of a constant
	$ \widetilde{C} $ depending on $ \Omega $, $ \nu $, $ \beta_F $, such that 
	\begin{equation}\label{eq:f7}
		\begin{array}{c}
			\disp\int_{0}^{t}\big\Vert\bdiv(\bsi_{\star,n}(s))\big\Vert_{0,4/3;\Omega}^{2}\, ds \, \leq\, \widetilde{C}\, \bigg\{\disp\int_{0}^{t}\Big(\Vert \bdev(\bsi_{\star,n}(s))\Vert_{0,\Omega}^{2} +\Vert \bu_{\star,m}(s)\Vert_{\mathbf{V}}^{2}+\Vert\partial_t  \bu_{\star,m}(s)\Vert_{0,\Omega}^{2}\Big)\, ds\bigg\}\,.
		\end{array}
	\end{equation}

To bound the last term in \eqref{eq:f7}, we differentiate the first equation of \eqref{eq:subp1.a} in time, test with $(\bsi_{\star,n}, \partial_t \bu_{\star,m})$, and apply Young's and inverse inequalities, yielding
	\begin{equation*}
		\begin{array}{c}
			\Vert\partial_t \bu_{\star,n}\Vert_{0,\Omega}^{2}+\dfrac{1}{2\nu}\partial_t\Vert\bdev(\bsi_{\star,n})\Vert_{0,\Omega}^{2}\,\leq\,
			\widetilde{C}_1 \, \Big\{
			\left(\Vert\partial_t\mathbf{g}\Vert_{1/2,\Gamma_{\mathtt{o}}^c}+a_1 \Vert\partial_t\chi\Vert_{\mathrm{Q}}\right)\,\Vert\bsi_{\star,n}\Vert_{\mathbb{H}}\\
			\,+\,  \, \left(\Vert\bdev(\bsi_{\star,n})\Vert_{0,\Omega}+\Vert\bu_{\star,m}\Vert_{\mathbf{V}}\right)\Vert\partial_t\bu_{\star,m}\Vert_{\mathbf{V}}\Big\}\\
			\,\leq\, \widetilde{C}_1^{2}(\epsilon_1)\left(\Vert\partial_t\mathbf{g}\Vert_{1/2,\Gamma_{\mathtt{o}}^c}+a_1 \Vert\partial_t\chi\Vert_{\mathrm{Q}}\right)^{2}\,+\dfrac{\epsilon_1}{2}\Vert\bsi_{\star,n}\Vert_{\mathbb{H}}^{2}\\
			\,+\, \widetilde{C}_1 ^{2}(\epsilon_2)
			\left(\Vert\bdev(\bsi_{\star,n})\Vert_{0,\Omega}^{2}+\Vert\bu_{\star,m}\Vert_{\mathbf{V}}\right)^{2}+\dfrac{\epsilon_2}{2}|\Omega|^{-1/4}\Vert\partial_t\bu_{\star,m}\Vert_{0,\Omega}^{2}\,,
		\end{array}
	\end{equation*}
	
	from which, by choosing $ \epsilon_2\leq |\Omega|^{1/4} $ and integrating from 
	$ 0 $ to $ t $,  we arrive at 
	\begin{equation}\label{eq:f6}
		\begin{array}{c}
			\disp\int_{0}^{t}\Vert\partial_t \bu_{\star,n}\Vert_{0,\Omega}^{2}\, ds +\Vert\bdev(\bsi_{\star,n}(t))\Vert_{0,\Omega}^2 \\\,\leq\,\mathcal{C}_2 \, \bigg\{ \Vert \bdev(\bsi_{\star,n}(0))\Vert_{0,\Omega}^{2}
			+\disp\int_{0}^{t}\left(\Vert\bdev(\bsi_{\star,n}(s))\Vert_{0,\Omega}^{2}+\Vert\bu_{\star,m}(s)\Vert_{\mathbf{V}}^{2}\right)\, ds
			\\
			\,+\,
			\disp\int_{0}^{t}\left(\Vert\partial_t\mathbf{g}(s)\Vert_{1/2,\Gamma_{\mathtt{o}}^c}^2+a_1^2 \Vert\partial_t\chi(s)\Vert_{\mathrm{Q}}^2+ \epsilon_1 \Vert\bsi_{\star,n}(s)\Vert_{\mathbb{H}}^{2}\right)
			\bigg\}\,.
		\end{array}
	\end{equation}
Substituting \eqref{eq:f6} into \eqref{eq:f7} and combining with \eqref{eq:f4}, \eqref{kll}, and \eqref{kjj}, we get
	\begin{equation*}
		\begin{array}{c}
			c_{2,\Omega}^2\min\big\{\dfrac{c_{1,\Omega}^2}{2},\dfrac{1}{2}\big\}	\disp\int_{0}^{t}\Vert\bsi_{\star,n}(s)\Vert_{\mathbb{H}}^{2}\, ds \,\\ \leq\,\widetilde{C}\,\mathcal{C}_2\, \left(\Vert \bdev(\bsi_{\star,n}(0))\Vert_{0,\Omega}^{2}+\disp\int_{0}^{t}\big(\Vert\partial_t\mathbf{g}(s)\Vert_{1/2,\Gamma_{\mathtt{o}}^c}^2+a_1^2 \Vert\partial_t\chi(s)\Vert_{\mathrm{Q}}^2\big) \, ds\right)\\
			\,+\, \mathcal{C}_1\big(1+ \,\widetilde{C}(1+\mathcal{C}_2)\big)\, \bigg\{\Vert\bu_{\star,m}(0)\Vert_{0,\Omega}^{2}+\disp\int_{0}^{t}\left(\Vert\mathbf{g}(s)\Vert_{1/2,\Gamma_{\mathtt{o}}^c}+a_1\Vert\chi(s)\Vert_{\mathrm{Q}}\right)^2\,ds \bigg\}\\
			 \,+\,\left(\epsilon\,\mathcal{C}_1 \,\big(1+\widetilde{C}(1+\mathcal{C}_2)\big)+\epsilon_1\widetilde{C}\,\mathcal{C}_2\right)\,\disp\int_{0}^{t}\Vert\bsi_{\star,n}(s)\Vert_{\mathbb{H}}^{2}\,ds \,,
		\end{array}
	\end{equation*}
	which, by choosing the appropriate  $ \epsilon $ and $ \epsilon_1  $ yields
	\begin{equation}\label{ke}
		\begin{array}{c}
			\disp\int_{0}^{t}\Vert\bsi_{\star,n}(s)\Vert_{\mathbb{H}}^{2}\, ds \, \leq\,\mathcal{C}_3\,\bigg\{\Vert \bdev(\bsi_{\star,n}(0))\Vert_{0,\Omega}^{2}+\Vert\bu_{\star,m}(0)\Vert_{0,\Omega}^{2}\\
			\,+\,  \,\disp\int_{0}^{t}\big(\Vert\partial_t\mathbf{g}(s)\Vert_{1/2,\Gamma_{\mathtt{o}}^c}^2+a_1^2 \Vert\partial_t\chi(s)\Vert_{\mathrm{Q}}^2\big)\,ds+\disp\int_{0}^{t}\left(\Vert\mathbf{g}(s)\Vert_{1/2,\Gamma_{\mathtt{o}}^c}+a_1\Vert\chi(s)\Vert_{\mathrm{Q}}\right)^2 \,ds \bigg\}\,.
		\end{array}
	\end{equation}
	To bound $ \Vert \bdev(\bsi_{\star,n}(0))\Vert_{0,\Omega} $, we set $ \bta_i:=\bsi_{\star,n} $ at the first equation of \eqref{eq:subp1.a} and let $ t=0 $, leading to
	\begin{equation}\label{eq:a3}
		\begin{array}{c}
			\mathscr{A}^F(\bsi_{\star,n}(0), \bsi_{\star,n}(0))\,=\, -\disp\int_{\Omega}\bdiv(\bsi_{\star,n}(0))\cdot \bu_{\star,m}(0)+\mathscr{F}_{\chi}^F(\bsi_{\star,n}(0))\,.
		\end{array}
	\end{equation}	
	Noting that \eqref{eq:a3} holds for all $ n,m\geq 1 $, we bound the right-hand side as before, let $ m\rightarrow \infty $, and use Green's formula, given that $ \bu_{\star,m}(0)\rightarrow \bu_{0} $ in $ \mathbf{L}^{2}(\Omega) $ with $ \bu_{0} \in \mathbf{H}^{1}(\Omega) $, which yields
	\begin{equation*}
		\begin{array}{c}
			\dfrac{1}{\nu}\,\Vert \bdev(\bsi_{\star,n}(0))\Vert_{0,\Omega}^{2}\, \leq \, 
			\Big|-\disp\int_{\Omega}\bdiv(\bsi_{\star,n}(0))\cdot \bu_{0}+\mathscr{F}_{\chi}^F(\bsi_{\star,n}(0))\Big|\\
			\qquad\qquad\qquad	\,=\, \Big|\disp\int_{\Omega}\bdev(\bsi_{\star,n}(0))\cdot\boldsymbol{\nabla}\bu_{0}\Big|\,\leq\, \Vert \bdev(\bsi_{\star,n}(0))\Vert_{0,\Omega}\, \Vert \boldsymbol{\nabla}\bu_{0}\Vert_{0,\Omega}\,.
		\end{array}
	\end{equation*}
	Thus
	\[
	\Vert \bdev(\bsi_{\star,n}(0))\Vert_{0,\Omega}\,\leq\, \nu \, \Vert \boldsymbol{\nabla}\bu_{0}\Vert_{0,\Omega}\,,
	\]
	which finishes the proof.
\end{proof}
We are now in a position to establish the well-definedness of $ \mathscr{L}^F $.
\begin{lemma}\label{l_wel.S}
	For each $ \bz\in L^{\infty}(\mathrm{J};\mathbf{V}) $ and $ \chi\in L^{\infty}(\mathrm{J};\mathrm{Q}) $ satisfying  \eqref{eq:as.z.var}, there exists a unique $ (\bsi_{\star}, \bu_\star) $ solution of \eqref{eq:subp1} such that
$	\bsi_{\star}\in L^{2}(\mathrm{J};\mathbb{H}),\,\,\bdev(\bsi_{\star})\in L^{\infty}(\mathrm{J};\mathbb{L}^2(\Omega)),\,\,\bu_{\star}\in L^{\infty}(\mathrm{J};\mathbf{L}^{2}(\Omega))\cap L^{2}(\mathrm{J};\mathbf{V}),\,\,\partial_t \bu_{\star}\in L^{2}(\mathrm{J};\mathbf{L}^{2}(\Omega))$.
	Moreover, there exists positive constant $ \mathcal{C}_{\mathscr{L}^F} $, depending on $\nu$, $\beta_F$, $\Omega$, $c_{\mathscr{F}^F}$, such that there hold
	\begin{equation}\label{eq:apr1}
		\begin{array}{c}
			\hspace{-.1cm}\big\Vert\mathscr{L}^F_1(\bz,\chi)\big\Vert_{L^{2}(\mathrm{J};\mathbb{H})}
			+\big\Vert\bdev(\mathscr{L}^F_1(\bz,\chi))\big\Vert_{L^{\infty}(\mathrm{J};\mathbb{L}^2)}
			+\big\Vert\mathscr{L}^F_2(\bz,\chi)\big\Vert_{L^{\infty}(\mathrm{J};\mathbf{L}^{2})}+\big\Vert\partial_t\mathscr{L}^F_2(\bz,\chi)\big\Vert_{L^{2}(\mathrm{J};\mathbf{L}^{2})}\\[1ex]
			+\big\Vert\mathscr{L}^F_{2}(\bz,\chi)\big\Vert_{L^{2}(\mathrm{J};\mathbf{V})}\,\leq\, \mathcal{C}_{\mathscr{L}^F}\mathcal{N}_F(\bu_{0},\mathbf{g},\chi) \,.
		\end{array}
	\end{equation}
\end{lemma}
\begin{proof}
	To prove existence, we invoke Lemma~\ref{l_help}, which shows that the sequences $ \{\bsi_{\star,n}\} $ and $ \{\bu_{\star,m}\} $ from \eqref{eq:subp1.a} are bounded in $ L^{2}(\mathrm{J};\mathbb{H}) $ and $ L^{\infty}(\mathrm{J};\mathbf{L}^2(\Omega))\cap L^2(\mathrm{J};\mathbf{V}) $, respectively, with $\{\partial_t\bu_{\star,m}\}$ bounded in $ L^2(\mathrm{J};\mathbf{L}^2(\Omega)) $. Hence, there exist subsequences (still denoted the same) and limit functions $\bsi_\star\in L^2(\mathrm{J};\mathbb{H})$ and $\bu_\star\in L^{\infty}(\mathrm{J};\mathbf{L}^2(\Omega))\cap L^2(\mathrm{J};\mathbf{V})$ such that
	\begin{equation}\label{eq:conv2}
			\bu_{\star,m}\rightharpoonup \bu_\star  \qin L^{2}(\mathrm{J};\mathbf{V}),\quad
			\partial_t\bu_{\star,m}\rightharpoonup \partial_t\bu_\star  \qin L^{2}(\mathrm{J};\mathbf{L}^{2}(\Omega)),\quad
			\bsi_{\star,n}\rightharpoonup \bsi_\star  \qin L^{2}(\mathrm{J};\mathbb{H})\,.	
	\end{equation}
Next, let $n_0, m_0 \geq 1$, and take $\bta \in C^1(\mathrm{J}; \mathbb{H}_{n_0})$, $\bv \in C^1(\mathrm{J}; \mathbf{V}_{m_0})$. For $n \geq n_0$, $m \geq m_0$, testing \eqref{eq:subp1.a} with $\bta$ and $\bv$ and integrating in time yields
	\begin{equation}\label{eq:a4.}
		\begin{array}{rcll}
\disp\int_{0}^{t_F}\Big\{	\mathscr{A}^F(\bsi_{\star,n} , \bta)+\mathscr{B}^F(\bta, \bu_{\star,m})\Big\}\, dt&=&\disp\int_{0}^{t_F} \mathscr{F}_{\chi}^F(\bta)\, dt \,,\\
\hspace{-.2cm}\disp\int_{0}^{t_F}\Big\{	\disp\int_{\Omega}\partial_t \bu_{\star,m} \cdot\bv-\mathscr{B}^F(\bsi_{\star,n},\bv)-\mathscr{O}_1^F(\bz;\bsi_{\star,n}, \bv)
			\,+\, \mathscr{O}_2^F(\bz;\bu_{\star,m}, \bv)\Big\}\, dt&=&0\,.
		\end{array}	
	\end{equation}
	Because of the weak convergence in 
	\eqref{eq:conv2}, we also have
	\begin{equation}\label{eq:a4}
		\begin{array}{rcll}
\disp\int_{0}^{t_F}\Big\{	\mathscr{A}^F(\bsi_{\star} , \bta)+\mathscr{B}^F(\bta, \bu_{\star})\Big\}\, dt&=&\disp\int_{0}^{t_F} \mathscr{F}_{\chi}^F(\bta)\, dt\,,\\
\hspace{-.2cm}\disp\int_{0}^{t_F}\Big\{	\disp\int_{\Omega}\partial_t \bu_{\star} \cdot\bv-\mathscr{B}^F(\bsi_{\star},\bv)-\mathscr{O}_1^F(\bz;\bsi_{\star}, \bv)
\,+\, \mathscr{O}_2^F(\bz;\bu_{\star}, \bv)\Big\}\, dt&=&0\,. 
		\end{array}
	\end{equation}
	Since $\bta$ and $\bv$ are dense in $L^2(\mathrm{J};\mathbb{H})$ and $L^2(\mathrm{J};\mathbf{V})$, respectively, \eqref{eq:a4} implies that \eqref{eq:subp1} holds for a.e.\ $t\in\mathrm{J}$. It remains to show $\bu_{\star}(0) = \bu_0$. To this end, take $\bv \in C^1(\mathrm{J}; \mathbf{V}_{n_0})$ with $\bv(t_F) = 0$ and use the second equation in \eqref{eq:subp1}, yielding
	\begin{equation}\label{eq:a5}
		-\disp\int_{0}^{t_F}\Big\{	\disp\int_{\Omega} \bu_\star\cdot\partial_t\bv+\mathscr{B}^F(\bsi_\star,\bv)+\mathscr{O}_1^F(\bz;\bsi_\star, \bv)-\mathscr{O}_2^F(\bz;\bu_\star, \bv)\Big\}\, dt \,=\, \disp\int_{\Omega} \bu_\star(0)\cdot\bv(0)\,.
	\end{equation}
	Similarly, from the second equation of \eqref{eq:a4.} we deduce
	\begin{equation*}
		-\disp\int_{0}^{t_F}\Big\{	\disp\int_{\Omega} \bu_{\star,m}\cdot\partial_t\bv+\mathscr{B}^F(\bsi_{\star,n},\bv)+\mathscr{O}_1^F(\bz;\bsi_{\star,n}, \bv)-\mathscr{O}_2^F(\bz;\bu_{\star,m}, \bv)\Big\}\, dt\,=\, \disp\int_{\Omega} \bu_{\star,m}(0)\cdot\bv(0)\,,
	\end{equation*}
	from which, an application of \eqref{eq:conv2} and fact that $ \bu_{\star,m}(0)\rightarrow \bu_0 $, yields
	\begin{equation}\label{eq:a6}
		-\disp\int_{0}^{t_F}\Big\{	\disp\int_{\Omega} \bu_\star \cdot\partial_t\bv+\mathscr{B}^F(\bsi_\star,\bv)+\mathscr{O}_1^F(\bz;\bsi_\star, \bv)-\mathscr{O}_2^F(\bz;\bu_\star, \bv)\Big\}\, dt \,=\, \disp\int_{\Omega} \bu_0\cdot\bv(0)\,.
	\end{equation}
	Since $\bv(0)$ is arbitrary, comparing \eqref{eq:a5} and \eqref{eq:a6} yields $\bu_\star(0) = \bu_0$.  
	For uniqueness, note that \eqref{eq:subp1} is linear for given $(\bz,\chi) \in L^\infty(\mathrm{J};\mathbf{V}) \times L^\infty(\mathrm{J};\mathrm{Q})$. Hence, it suffices to show $\bu_\star = \mathbf{0}$ and $\bsi_\star = \mathbf{0}$ when $\mathscr{F}_{\chi}^F(\bta) = \mathbf{0}$ and $\bu_0 = \mathbf{0}$.  
	Taking $(\bta,\bv) = (\bsi_\star,\bu_\star)$ in \eqref{eq:subp1} and proceeding as in \eqref{eq:j1}, we obtain
	\begin{equation*}
		\begin{array}{c}
			\dfrac{1}{2}\partial_t \Vert \bu_{\star}\Vert_{0,\Omega}^{2}\,+\,\dfrac{1}{2\nu}\, \Vert \bdev(\bsi_{\star})\Vert_{0,\Omega}^{2}+\dfrac{\beta_F^{2}\nu}{2}\Vert \bu_{\star}\Vert_{\mathbf{V}}^{2} \\
			\,\leq\,
			\dfrac{1}{\nu}\left(\dfrac{1}{4}\Vert\bdev(\bsi_{\star})\Vert_{0,\Omega}^{2}+\Vert\bz\Vert_{\mathbf{V}}^{2}\Vert \bu_{\star}\Vert_{\mathbf{V}}^{2}\right)
			\,+\,|\Omega|^{(4-p)/4}\Vert\bz\Vert_{\mathbf{V}}^{p-2}\Vert\bu_{\star,m}\Vert_{\mathbf{V}}^{2}\,,
		\end{array}
	\end{equation*}
from which, using \eqref{eq:as.z.var}, integrating in time, and noting $\bu_{\star}(0) = \bu_0$, we conclude that
	\begin{equation*}
		\begin{array}{c}
			\Vert\bu_{\star}(t)\Vert_{0,\Omega}^{2}+\dfrac{1}{\nu}\disp\int_{0}^{t}\Vert \bdev(\bsi_{\star}(s))\Vert_{0,\Omega}^{2}\, ds +\beta_F^{2}\nu\disp\int_{0}^{t}\Vert \bu_{\star}(s)\Vert_{\mathbf{V}}^{2}\, ds
			\,\leq\,0 \qquad \text{for a.e.}\, t\in \mathrm{J}\,.
		\end{array}
	\end{equation*}
Thus, $\bu_\star = \mathbf{0}$ and $\bsi_\star = \mathbf{0}$ a.e. in $\mathrm{J}$. The {\it a priori} estimate \eqref{eq:apr1} then follows from \eqref{ky}.
\end{proof}
To establish the well-posedness of $\mathscr{L}^C$, we first derive a {\it priori} estimates for the solution of \eqref{eq:subp2.a}.
\begin{lemma}\label{l_wel.con.pr2}
	Assume that $ (\bz,\chi)\in \mathbf{V}_m \times \mathrm{Q}_{\tilde{s}} $ be given and satisfy
	\begin{equation}\label{eq:as2}
		\big\Vert(\bz(t),\chi(t))\big\Vert_{\mathbf{V}\times\mathrm{Q}}\,\le\,\max\Big\{\dfrac{\beta_C\kappa}{2\sqrt{3}},\,\frac{\beta_C^2\kappa}{12c_{\mathscr{O}_2^C}}\Big\} \,.
	\end{equation}
	Then, there exists a constant $ \widetilde{\mathcal{C}}_1 $, depending on $\kappa$, $\beta_C$, $c_{\mathscr{O}_2^C}$ but independent of $ \tilde{n} $, $ \tilde{m} $ and $ \tilde{s} $, such that
	\begin{equation*}
		\begin{array}{c}
			\big\Vert \varphi_{\star,\tilde{m}}\big\Vert_{L^{\infty}(\mathrm{J};L^{2})}^{2}+\big\Vert\brho_{\star,\tilde{n}}\big\Vert_{L^{\infty}(\mathrm{J};\mathbf{L}^2)}^{2}+\big\Vert\partial_t \vec{\varphi}_{\star,\tilde{m}\tilde{s}}\big\Vert_{L^{2}(\mathrm{J};L^2\times\mathrm{Q})}^{2} +\big\Vert\brho_{\star,\tilde{n}}\big\Vert_{L^{2}(\mathrm{J};\mathbf{H})}^{2} + \big\Vert \vec{\varphi}_{\star,\tilde{m}\tilde{s}}\big\Vert_{L^{2}(\mathrm{J};\mathcal{V})}^{2}\\
			\, \leq\,  \widetilde{\mathcal{C}}_1 \left(\Vert \varphi_{\star}(0)\Vert_{\mathrm{H}^{1}(\Omega)}^{2}+\disp\int_{\Omega}c_{\mathscr{F}^C}^2\, ds\right)\,:=\,\widetilde{\mathcal{C}}_1\mathcal{N}_C(\varphi_{\star}(0),c_{\mathscr{F}^C})\,.
		\end{array}
	\end{equation*}
\end{lemma}
\begin{proof}
	The argument parallels that of Lemma~\ref{l_help}, applied to \eqref{eq:subp2.a}. Choosing $(\bet_i,\vec{\psi}_{j,k}) := (\brho_{\star,\tilde{n}}(t), \vec{\varphi}_{\star,\tilde{m}\tilde{s}}(t))$, summing the equations and employing the coercivity of $\mathscr{A}^C$ (cf. \eqref{eq:ell.a.at}) and the boundedness of $\mathscr{O}_1^C$, $\mathscr{O}_2^C$, and $\mathscr{F}^C$, we deduce
	\begin{equation}\label{eq:j3}
		\begin{array}{c}
\dfrac{1}{2}\partial_t \Vert \varphi_{\star,\tilde{m}}\Vert_{0,\Omega}^{2}+\dfrac{1}{\kappa}\Vert\brho_{\star,\tilde{n}}\Vert_{0,\Omega}^{2}+a_2\Vert\lambda_{\star,\tilde{s}}\Vert_{0,\Gamma_{\mathtt{w}}}^2	\\\, \leq\,
c_{\mathscr{F}^C}\Vert\vec{\varphi}_{\star,\tilde{m}\tilde{s}}\Vert_{\mathcal{V}}
\,+\,\dfrac{1}{\kappa}\,\Vert\bz\Vert_{\mathbf{V}}\Vert\brho_{\star,\tilde{n}}\Vert_{0,\Omega}\,  \Vert \vec{\varphi}_{\star,\tilde{m}\tilde{s}}\Vert_{\mathcal{V}} +  c_{\mathscr{O}_2^C}\Vert\chi\Vert_{\mathrm{Q}}\Vert\vec{\varphi}_{\star,\tilde{m}\tilde{s}}\Vert_{\mathcal{V}}^2 \,.
		\end{array}
	\end{equation} 
Using the inf-sup condition \eqref{eq:inf.bit}, the first equation of \eqref{eq:subp2.a}, and the boundedness of $\mathscr{A}^C$ and $\mathscr{F}^C$, we obtain
	\begin{equation*}
		\begin{array}{c}
			\beta_C\,\Vert\vec{\varphi}_{\star,\tilde{m}\tilde{s}}\Vert_{\mathcal{V}}\,\leq\,	\sup\limits_{\mathbf{0}\neq \bet_i\in \mathbf{H}}\dfrac{\mathscr{B}^{C}(\bet_i, \vec{\varphi}_{\star,\tilde{m}\tilde{s}})}{\Vert \bet_i\Vert_{\mathbf{H}}}
			\,=\
			\sup\limits_{\mathbf{0}\neq \bet\in \mathbf{H}}\dfrac{-\mathscr{A}^{C}(\brho_{\star,\tilde{n}} , \bet_i)}{\Vert \bet_i\Vert_{\mathbf{H}}}\, \leq\, \dfrac{1}{\kappa}\Vert \brho_{\star,\tilde{n}}\Vert_{0,\Omega}\,,
		\end{array}
	\end{equation*}
	which taking square in the above inequality, we have
	\begin{equation}\label{eq:no.4}
		\dfrac{\beta_C^{2}\kappa}{2}\Vert\vec{\varphi}_{\star,\tilde{m}\tilde{s}}\Vert_{\mathcal{V}}^2\,\leq\,\dfrac{1}{2\kappa}\Vert\brho_{\star,\tilde{n}}\Vert_{0,\Omega}^{2}\,.
	\end{equation}
Combining with \eqref{eq:j3}, applying Young’s inequality, integrating over \( [0,t] \), using \eqref{eq:as2}, and choosing \( \epsilon \le \frac{\beta_C^2 \kappa}{12} \), we get
	\begin{equation}\label{eq:l1}
		\begin{array}{c}
			\Vert \varphi_{\star,\tilde{m}}(t)\Vert_{0,\Omega}^{2}+\disp\int_{0}^{t}\Vert\brho_{\star,\tilde{n}}(s)\Vert_{0,\Omega}^{2}\,ds + \disp\int_{0}^{t}\Vert \vec{\varphi}_{\star,\tilde{m}\tilde{s}}(s)\Vert_{\mathcal{V}}^{2}\, ds
			\, \leq\, \widetilde{\mathcal{C}}_1\, \left(\Vert \varphi_{\star,\tilde{m}}(0)\Vert_{0,\Omega}^{2}+\disp\int_{0}^{t}c_{\mathscr{F}^C}^2\right)\,,
		\end{array}
	\end{equation}
	where $ \widetilde{\mathcal{C}}_1 $ is a positive constant, depending on $ \kappa $, $ \beta_C $ and $c_{\mathscr{O}_2^C}$. 
	
Using the second equation of \eqref{eq:subp2.a} and $\div(\mathbf{H}) = \mathrm{V}'$, we deduce a constant $\widehat{C} > 0$ such that
	\begin{equation}\label{eq:k1}
		\begin{array}{c}
			\disp\int_{0}^{t}\Vert\brho_{\star,\tilde{n}}(s)\Vert_{\mathbf{H}}^{2} ds \, \leq\, \widehat{C}\, \bigg\{\Vert \varphi_{\star,\tilde{m}}(0)\Vert_{0,\Omega}^{2}+
			\disp\int_{0}^{t}\left(c_{\mathscr{F}^C}^2 
			+  	\Vert \partial_t \varphi_{\star,\tilde{m}}(s)\Vert_{0,\Omega}^{2}\right)ds
			\bigg\} \,.
		\end{array}
	\end{equation}
To estimate the last term in \eqref{eq:k1}, we time-differentiate the first equation of \eqref{eq:subp2.a}, test with $(\bet_i, (\psi_j, \xi_k)) = (\brho_{\star,\tilde{n}}, (\partial_t\varphi_{\star,\tilde{m}}, \partial_t\lambda_{\star,\tilde{s}}))$, and combine with the second, yielding:
	\begin{equation}\label{eq:k3}
		\begin{array}{c}
			\dfrac{1}{2\kappa}\partial_t\Vert\brho_{\star,\tilde{n}}\Vert_{0,\Omega}^{2}+\Vert\partial_t \varphi_{\star,\tilde{m}}\Vert_{0,\Omega}^{2}\, \leq\, c_{\mathscr{F}^C}\Vert\partial_t\vec{\varphi}_{\star,\tilde{m}\tilde{s}}\Vert_{\mathcal{V}}\\[1ex]
			\, \,+\,\dfrac{1}{\kappa}\Vert\bz\Vert_{\mathbf{V}}\, \Vert\brho_{\star,\tilde{n}}\Vert_{0,\Omega}\, \Vert\partial_t\vec{\varphi}_{\star,\tilde{m}\tilde{s}}\Vert_{\mathcal{V}}+\,\left(c_{\mathscr{C}^C}+c_{\mathscr{O}_2^C}\Vert\chi\Vert_{\mathrm{Q}}\right)\Vert\vec{\varphi}_{\star,\tilde{m}\tilde{s}}\Vert_{\mathcal{V}}\Vert\partial_t\vec{\varphi}_{\star,\tilde{m}\tilde{s}}\Vert_{\mathcal{V}}\,.
		\end{array}
	\end{equation}
Next, to bound $\|\partial_t \vec{\varphi}_{\star,\tilde{m}\tilde{s}}\|_{\mathcal{V}}$ by $\|\partial_t \varphi_{\star,\tilde{m}}\|_{0,\Omega}$, we differentiate in time the following identity
	\begin{equation}\label{eq:no.7}
		\begin{array}{c}
			\langle \bet_i \cdot\bn,
			\lambda_{\star,\tilde{s}}\rangle_{\Gamma_{\mathtt{i}}^c}	\,=\, \langle \bet_i \cdot\bn,\varphi_{\star,\tilde{m}} \rangle_{\Gamma}\,,
		\end{array}
	\end{equation}
	apply a special case of the inf-sup condition \eqref{eq:inf.bit} (cf. \eqref{part2.difBc}) and the inverse inequality, yielding
	\begin{equation*}
		\begin{array}{c}
			\beta_{2,C}\,	\Vert\partial_t \lambda_{\star,\tilde{s}}\Vert_{\mathrm{Q}}\, \leq\, \sup_{\bet_i\in \mathbf{H}}\dfrac{\langle \bet_i \cdot\bn,
				\partial_t\lambda_{\star,\tilde{s}}\rangle_{\Gamma_{\mathtt{i}}^c}}{\Vert\bet_i\Vert_{\mathbf{H}}}\, \leq\, \Vert\partial_t\varphi_{\star,\tilde{m}}\Vert_{1/2,\Gamma}\\
			\, \leq\, |\Omega|^{-1/4}\Vert\partial_t\varphi_{\star,\tilde{m}}\Vert_{L^{2}(\Gamma)}\, \leq\, |\Omega|^{-1/8}\Vert\partial_t\varphi_{\star,\tilde{m}}\Vert_{0,\Omega}\,.
		\end{array}
	\end{equation*}
	The recent result, also gives
	\begin{equation}\label{nn}
		\dfrac{\beta_{2,C}^2|\Omega|^{/4}}{2}	\Vert\partial_t \lambda_{\star,\tilde{s}}\Vert_{\mathrm{Q}}^2\,\le\,\dfrac{1}{2} \Vert\partial_t\varphi_{\star,\tilde{m}}\Vert_{0,\Omega}^2 \,.
	\end{equation}
	
Substituting \eqref{nn} into \eqref{eq:k3}, applying Young's inequality, integrating over time, and using \eqref{eq:as2} yield a constant $\widetilde{C}_1$ such that
	\begin{equation}\label{eq:k4}
		\begin{array}{c}
			\Vert\brho_{\star,\tilde{n}}(t)\Vert_{0,\Omega}^{2}+	\disp\int_{0}^{t}\left(	\Vert\partial_t \varphi_{\star,\tilde{m}}(s)\Vert_{0,\Omega}^{2}+\Vert\partial_t \lambda_{\star,\tilde{s}}\Vert_{\mathrm{Q}}^2\right)\\\, \leq\, \widetilde{C}_1 \, \bigg\{\Vert\brho_{\star,\tilde{n}}(0)\Vert_{0,\Omega}^{2}	\,+\,\disp\int_{0}^{t}\left(c_{\mathscr{F}^C}^2+\Vert\brho_{\star,\tilde{n}}(s)\Vert_{0,\Omega}^{2}+\Vert\vec{\varphi}_{\star,\tilde{m}\tilde{s}}(s)\Vert_{\mathcal{V}}^{2}\right)\bigg\}\,.
		\end{array}
	\end{equation}
	Replacing \eqref{eq:k4} in \eqref{eq:k1} and combining the result with \eqref{eq:l1} yield:
	\begin{equation*}
		\disp\int_{0}^{t} \Vert\brho_{\star,\tilde{n}}(s)\Vert_{\mathbf{H}}^{2}\, ds \, \leq\,  \widehat{C}_2 \left(\Vert \varphi_{\star,\tilde{m}}(0)\Vert_{0,\Omega}^{2}+\Vert\brho_{\star,\tilde{n}}(0)\Vert_{0,\Omega}^{2}+\disp\int_{0}^{t}c_{\mathscr{F}^C}^2 \, ds\right)\,.
	\end{equation*}

	The rest follows as in the proof of Lemma \ref{l_help}.
\end{proof}

We are in position to establish that $ \mathscr{L}^C $ (cf. \eqref{eq:opr2}) is well defined.
\begin{lemma}\label{l_wel.T}
%
%
	For each \( \bz \in L^{\infty}(\mathrm{J};\mathbf{V}) \) and \( \chi \in L^{\infty}(\mathrm{J};\mathrm{Q}) \) satisfying \eqref{eq:as2}, there exists a unique solution \((\brho_{\star}, \vec{\varphi}_\star)\) to \eqref{eq:subp2} with
	$	\brho_{\star}\in L^{2}(\mathrm{J};\mathbf{H})\cap L^{\infty}(\mathrm{J};\mathbf{L}^{2}(\Omega)),\,\,\varphi_{\star}\in L^{\infty}(\mathrm{J};\mathrm{L}^{2}(\Omega))\cap L^{2}(\mathrm{J};\mathrm{V}),\,\,\partial_t \vec{\varphi}_{\star}\in L^{2}(\mathrm{J};\mathrm{L}^{2}(\Omega)\times\mathrm{Q}),\,\,\lambda_{\star}\in L^{2}(\mathrm{J};\mathrm{Q})$.
	Moreover, there exists a constant \( \mathcal{C}_{\mathscr{L}^C} > 0 \), depending on \( \kappa, \beta_C, c_{\mathscr{O}_2^C}, |\Omega| \), such that
	\begin{equation}\label{eq:apr2}
		\begin{array}{c}
			\big\Vert \mathscr{L}^C_1(\bz,\chi)\big\Vert_{L^{2}(\mathrm{J};\mathbf{H})}+\Vert\mathscr{L}^C_1(\bz,\chi)\big\Vert_{L^{\infty}(\mathrm{J};\mathbf{L}^2)}+\big\Vert\mathscr{L}^C_2(\bz,\chi)\big\Vert_{L^{\infty}(\mathrm{J};\mathrm{L}^{2})}
			+\big\Vert\partial_t \vec{\mathscr{L}}_2^C (\bz,\chi)\big\Vert_{L^{2}(\mathrm{J};\mathrm{L}^{2}\times\mathrm{Q})}\\
			\,+\,\big\Vert\vec{\mathscr{L}}^C_{2}(\bz,\varphi)\big\Vert_{L^{2}(\mathrm{J};\mathcal{V})}
			\, \leq\,  \mathcal{C}_{\mathscr{L}^C} \mathcal{N}_C(\varphi_{0},c_{\mathscr{F}^C}) \,.
		\end{array}
	\end{equation}
\end{lemma}
\begin{proof}
	It follows similarly to the proof of Lemma \ref{l_wel.S}.
\end{proof}
\subsection{Solvability of the fixed-point equation}
Since $\mathscr{L}^F$, $\mathscr{L}^C$, and thus $\mathscr{L}$ are well-defined, we now study solvability of \eqref{eq:fixp}. First, we identify conditions ensuring $\mathscr{L}$ maps a closed ball in $\mathbf{V} \times \mathrm{Q}$ into itself.
Thus, denoting \begin{equation*}
	\begin{array}{c}
		r_1 \,:=\, \max\bigg\{ \left(\dfrac{\beta_F^2\nu}{8}|\Omega|^{(p-4)/4}\right)^{1/(p-2)},\,\dfrac{\beta_F\nu}{2\sqrt{2}} \bigg\}\,,\quad
		r_2 \,:=\, \max\Big\{\dfrac{\beta_C\kappa}{2\sqrt{3}},\,\dfrac{\beta_C^2\kappa}{12c_{\mathscr{O}_2^C}}\Big\}\qan
		r:=\min\{r_1, r_2\}\,,
	\end{array}
\end{equation*} 
we define
\begin{equation}\label{eq:def.ball}
	\mathbf{W}(r)\,:=\, \Big\{(\bz(t),\chi(t))\in \mathbf{V}\times\mathrm{Q}:\quad \big\Vert(\bz,\chi)\big\Vert_{L^{\infty}(\mathrm{J};\mathbf{V}\times\mathrm{Q})}\,  \leq\, r  \Big\}\,.
\end{equation}

Given $ (\bz,\chi)\in \mathbf{W}(r) $, and bearing in mind the definition of 
$ \mathscr{L} $ (cf. \eqref{eq:def.opr.E}), integration from $ 0 $ to $ t\in \mathrm{J} $ along with the a {\it priori} estimates \eqref{eq:apr1} and \eqref{eq:apr2}, yields
\begin{equation}\label{eq:l2}
	\begin{array}{c}
		\big\Vert \mathscr{T}(\bz,\chi)\big\Vert_{L^{\infty}(\mathrm{J};\mathbf{V}\times\mathrm{Q})}^{2}
		\, \leq\, 2\left(\big\Vert \mathscr{L}^F_2(\bz , \mathscr{L}^C_3(\bz, \chi))\big\Vert_{L^{\infty}(\mathrm{J};\mathbf{V})}^{2}+\big\Vert \mathscr{L}^C_3(\bz, \chi)\big\Vert_{L^{\infty}(\mathrm{J};\mathrm{Q})}^{2}\right)\\[1ex]
		\,\le\,\mathcal{C}_{\mathscr{T}}\widetilde{\mathcal{N}}(\bu_{0},\varphi_{0},\mathbf{g},c_{\mathscr{F}^C})\,,
	\end{array}
\end{equation}

where $ \mathcal{C}_{\mathscr{T}} = 2\max\big\{\mathcal{C}_{\mathscr{L}^F}^{2},  (\mathcal{C}_{\mathscr{L}^F}^{2}+1)\mathcal{C}_{\mathscr{L}^C}^2 \big\} $ and
$\widetilde{\mathcal{N}}(\bu_{0},\varphi_{0},\mathbf{g},c_{\mathscr{F}^C}):=
\mathcal{N}_F(\bu_{0},\bg,0)+\mathcal{N}_C(\varphi_{0},c_{\mathscr{F}^C})$.

We have therefore proved the following result.	
\begin{lemma}\label{l_cons.W}
	Assume that the data are sufficiently small so that there hold 
	\begin{equation}\label{eq:as1.data}
		\begin{array}{c}
			\mathcal{C}_{\mathscr{T}}\,\widetilde{\mathcal{N}}(\bu_{0},\varphi_{0},\mathbf{g},c_{\mathscr{F}^C})\, \le\, r \,,
		\end{array}
	\end{equation}
	then, $ \mathscr{T}(\mathbf{W}(r))\subseteq\mathbf{W}(r) $.
\end{lemma}	
\begin{proof}
	It is a direct consequence of the estimate \eqref{eq:l2}.
\end{proof}
We now address the Lipschitz continuity of $\mathscr{L}^F$ and $\mathscr{L}^C$. Recall from \cite[Lemma 4.4]{cgg-JSC-2023} that for each $p \in [3,4]$, there exists $\ell_o > 0$, depending only on $\mathrm{F}$, $|\Omega|$, and $p$, such that for all $\bz, \by, \bw, \bv \in \mathbf{V}$
\begin{equation}\label{eq:bound.o2}
	\big|\mathscr{O}_2^F (\bz;\bw,\bv)-\mathscr{O}_2^F (\by;\bw,\bv)\big|\, \leq\, \ell_o \, \Big(\Vert\bz\Vert_{\mathbf{V}}+\Vert\by\Vert_{\mathbf{V}}\Big)^{p-3}\Vert\bz-\by\Vert_{\mathbf{V}}\Vert\bw\Vert_{\mathbf{V}}\Vert\bv\Vert_{\mathbf{V}}\,.
\end{equation}
We now utilize the preceding result to confirm the stated property of $ \mathscr{L}^F $.
\begin{lemma}\label{l_lip.S}
	There exists a constant $\mathcal{L}_{\mathscr{L}^F} > 0$, depending only on $\nu$, $a_1$, $\beta_F$, $r$, and $\ell_o$, such that
	\begin{equation}\label{eq:lip.S}
		\begin{array}{c}
\big\Vert\mathscr{L}^F_2(\bz,\chi)-\mathscr{L}^F_2(\by,\omega)\big\Vert_{L^{\infty}(\mathrm{J};\mathbf{L}^{2}(\Omega))}\,+\,\big\Vert\partial_t\mathscr{L}^F_2(\bz,\chi)-\partial_t\mathscr{L}^F_2(\by,\omega)\big\Vert_{L^{2}(\mathrm{J};\mathbf{L}^{2}(\Omega))}\\[1ex]
			+\big\Vert\mathscr{L}^F(\bz,\chi)-\mathscr{L}^F(\by,\omega)\big\Vert_{L^{2}(\mathrm{J};\mathbb{H}\times\mathbf{V})}
			\, \leq\, \mathcal{L}_{\mathscr{L}^F}\,\Big\{ \Vert\chi-\omega\Vert_{L^{2}(\mathrm{J};\mathrm{Q})}+\Vert \partial_t(\chi-\omega)\Vert_{L^{2}(\mathrm{J};\mathrm{Q})} \\[1ex]
			\,+\,\Vert\bz-\by\Vert_{L^{\infty}(\mathrm{J};\mathbf{V})}\left(\big\Vert \mathscr{L}^F_1(\by,\omega) \big\Vert_{L^{2}(\mathrm{J};\mathbb{H})}+\big\Vert\mathscr{L}^F_{2}(\by,\omega)\big\Vert_{L^{2}(\mathrm{J};\mathbf{V})}\right) \Big\}\,.
		\end{array}
	\end{equation}
\end{lemma}	
\begin{proof}
	Given $ (\bz,\chi) $ and $ (\by, \omega) $, we let $ \mathscr{L}^F(\bz(t),\chi(t))=(\bsi_{\star}(t),\bu_{\star}(t)) $ and $ \mathscr{L}^F(\by(t),\omega(t))=(\bsi_{\circ}(t),\bu_{\circ}(t)) $ for $t\in\mathrm{J}$ as the unique solutions to \eqref{eq:subp1}.
Upon subtracting the two systems, properly collecting the resulting expressions, and denoting \( e_{\bsi} := \bsi_{\star} - \bsi_{\circ} \), \( e_{\bu} := \bu_{\star} - \bu_{\circ} \), we arrive at
	\begin{equation}\label{eq:hh1}
		\begin{split}
\mathscr{A}^F(e_{\bsi} , \bta)+\mathscr{B}^F(\bta, e_{\bu})=& \big(\mathscr{F}_{\chi}^F-\mathscr{F}_{\omega}^F\big)(\bta)\,,\\
\disp \int_{\Omega}\partial_t e_{\bu} \cdot\bv-\mathscr{B}^F(e_{\bsi},\bv)-\mathscr{O}_1^F(\bz;e_{\bsi}, \bv)
-\mathscr{O}_1^F(\by;\bsi_\circ, \bv)+\mathscr{O}_2^F(\bz;\bu_\star-\bu_{\circ}, \bv)=&\mathscr{O}_1^F(\bz-\by;\bsi_{\circ},\bta)\\
- \big[\mathscr{O}_2^F(\bz;\bu_\circ, \bv)-\mathscr{O}_2^F(\by;\bu_\circ, \bv)\big]\,. & 
		\end{split}
	\end{equation}
	
Setting $(\bta,\bv) = (e_{\bsi}, e_{\bu})$, summing the first two rows of \eqref{eq:hh1}, and using the ellipticity of $\mathscr{A}^F$ along with the boundedness of $\mathscr{O}_1^F$ and $\mathscr{O}_2^F$ yield
	%
	\begin{equation}\label{eq:h6}
		\begin{array}{c}
			\dfrac{1}{2}\partial_t \big\Vert e_{\bu}\big\Vert_{0,\Omega}^{2}+\dfrac{1}{\nu}\big\Vert \bdev(e_{\bsi})\big\Vert_{0,\Omega}^{2}\\\,\leq\, \Big\{\dfrac{1}{\nu}\Vert\bz\Vert_{\mathbf{V}}\big\Vert \bdev(e_{\bsi})\big\Vert_{0,\Omega}
			\,+\,|\Omega|^{(4-p)/4}\, \Vert\bz\Vert_{\mathbf{V}}^{p-2}\, \Vert e_{\bu}\Vert_{\mathbf{V}}\Big\}\Vert e_{\bu}\Vert_{\mathbf{V}}\\
			\,+\, \Big\{
			\dfrac{1}{\nu}\Vert\bdev(\bsi_{\circ})\Vert_{0,\Omega}+\ell_o \, (2\Vert\bz\Vert_{\mathbf{V}})^{p-3}\Vert\bu_{\circ}\Vert_{\mathbf{V}}\Big\}\, \Vert\bz-\by\Vert_{\mathbf{V}}\,\Vert e_{\bu}\Vert_{\mathbf{V}}
			+ c_{\mathscr{F}^F} \, \Vert \chi-\omega\Vert_{\mathrm{Q}}\, \Vert e_{\bsi}\Vert_{\mathbb{H}}\,.
		\end{array}
	\end{equation}
	
	\smallskip
	A result analogous to \eqref{eq:h3} follows by using the first row of \eqref{eq:hh1} as
	\begin{equation}\label{eq:h5}
		\dfrac{\beta_F^{2}\nu}{2}\,\Vert e_{\bu}\Vert_{\mathbf{V}}^{2}\,\leq\, 2\nu\, c_{\mathscr{F}^F}^{2}\, \Vert\chi-\omega\Vert_{\mathrm{Q}}^{2}+\dfrac{1}{2\nu}\Vert\bdev(e_{\bsi})\Vert_{0,\Omega}^{2}\,.	
	\end{equation}

Combining \eqref{eq:h5} and \eqref{eq:h6}, applying Young’s inequality, integrating in time, using \eqref{eq:as.z.var}, and $\bu_{\star}(0)=\bu_{\circ}(0)$, we obtain a constant $C_0>0$, depending on $\nu$, $c_{\mathscr{F}^F}$, $\beta_F$, $r$, $\ell_o$, and $\epsilon$, such that
	\begin{equation}\label{eq:h7}
		\begin{array}{c}
\big\Vert e_{\bu}(t)\big\Vert_{0,\Omega}^{2}+\disp\int_{0}^{t}\left(\big\Vert \bdev(e_{\bsi}(s))\big\Vert_{0,\Omega}^{2}+\Vert e_{\bu}(s)\Vert_{\mathbf{V}}^{2}\right)ds
\,\leq\,C_0 \, \bigg\{\disp\int_{0}^{t} \Vert\chi(s)-\omega(s)\Vert_{\mathrm{Q}}^{2}\, ds\\
\,+\,\disp\int_{0}^{t}\left(\Vert\bdev(\bsi_{\circ}(s))\Vert_{0,\Omega}^{2}+\Vert\bu_{\circ}(s)\Vert_{\mathbf{V}}^{2}\right)\Vert\bz(s)-\by(s)\Vert_{\mathbf{V}}^{2}\, ds\,+\,\epsilon\disp\int_{0}^{t}\Vert e_{\bsi}(s)\Vert_{\mathbb{H}}^{2}\, ds\bigg\}\,.
		\end{array}
	\end{equation}
	To bound $\int_{0}^{t}\Vert e_{\bsi}\Vert_{\mathbb{H}}^{2}$, as in \eqref{eq:h7}, we follow the approach used in deriving \eqref{eq:f7} to obtain:
	\begin{equation}\label{eq:r1}
		\begin{array}{c}
\disp\int_{0}^{t}\Vert\bdiv(e_{\bsi}(s))\Vert_{0,4/3;\Omega}^{2}ds \, \leq\, 
C_1 \,  \Big\{\disp\int_{0}^{t}\Vert\partial_t (e_{\bu}(s)) \Vert_{0,\Omega}^{2}
\,+\,\disp\int_{0}^{t}\Vert e_{\bu}(s)\Vert_{\mathbf{V}}^{2}+\disp\int_{0}^{t}\Vert\bdev(e_{\bsi}(s))\Vert_{0,\Omega}^{2} \\
\,+\,\disp\int_{0}^{t}\left(\Vert\bdev(\bsi_{\circ}(s))\Vert_{0,\Omega}^{2}+\Vert\bu_{\circ}(s)\Vert_{\mathbf{V}}^{2}\right)\Vert\bz(s)-\by(s)\Vert_{\mathbf{V}}^{2} \Big\}\,.
		\end{array}
	\end{equation}
To estimate the first term on the right of \eqref{eq:r1}, we time-differentiate the first equation in \eqref{eq:hh1}, test with $(e_{\bsi}, \partial_t e_{\bu})$, apply Young's inequality, integrate in time, and use $\bsi_{\star}(0) = \bsi_{\circ}(0)$ to obtain
	\begin{equation}\label{eq:l3}
		\begin{array}{c}
			\disp\int_{0}^{t}\Vert\partial_t  e_{\bu}(s) \Vert_{0,\Omega}^{2}\, \leq\,C_2 \,\Big\{ \disp\int_{0}^{t}\Big(\Vert\bdev(e_{\bsi}(s))\Vert_{0,\Omega}^{2}+\Vert e_{\bu}(s)\Vert_{\mathbf{V}}^{2}\Big)\\
		\hspace{-.3cm}	+\disp\int_{0}^{t}\left(\Vert\bdev(\bsi_{\circ}(s))\Vert_{0,\Omega}^{2}+\Vert\bu_{\circ}(s)\Vert_{\mathbf{V}}^{2}\right)\Vert\bz(s)-\by(s)\Vert_{\mathbf{V}}^{2}
			+ \disp\int_{0}^{t}\left(\Vert \partial_t(\chi(s)-\omega(s))\Vert_{\mathrm{Q}}^{2}+\epsilon_1\Vert  e_{\bsi}(s)\Vert_{\mathbb{H}}^{2}\right)\Big\}\,.
		\end{array}
	\end{equation}
Substituting \eqref{eq:l3} into \eqref{eq:r1}, combining with \eqref{eq:h7}, \eqref{kll}, and \eqref{kjj}, and choosing small $\epsilon$ yields
	\begin{equation}\label{eq:r2}
		\begin{array}{c}
			\disp\int_{0}^{t}\Vert  e_{\bsi}(s)\Vert_{\mathbb{H}}^{2}\, \leq\, \widetilde{C}_2 \,\Big\{  \disp\int_{0}^{t}\left(\Vert\bdev(\bsi_{\circ}(s))\Vert_{0,\Omega}^{2}+\Vert\bu_{\circ}(s)\Vert_{\mathbf{V}}^{2}\right)\Vert\bz(s)-\by(s)\Vert_{\mathbf{V}}^{2}\\
			\,+\,\disp\int_{0}^{t} \Vert\chi(s)-\omega(s)\Vert_{\mathrm{Q}}^{2}+\disp\int_{0}^{t}\Vert \partial_t(\chi(s)-\omega(s))\Vert_{\mathrm{Q}}^{2}\Big\}\,,
		\end{array}
	\end{equation}
	
	which completes the proof.
\end{proof}
\begin{lemma}\label{l_lip.T}
	There exists a constant $ \mathcal{L}_{\mathscr{L}^C} >0$, depending on $ \beta_C $, $ \kappa $, $ r_2 $, such that
	\begin{equation}\label{eq:lip.T}
		\begin{array}{c}
			\big\Vert\mathscr{L}^C_2(\bz,\chi)-\mathscr{L}^C_2(\by,\omega)\big\Vert_{L^{\infty}(\mathrm{J};\mathbf{L}^{2}(\Omega))}
			+\big\Vert\partial_t\big(\vec{\mathscr{L}}_2^C(\bz,\chi)-\vec{\mathscr{L}}_2^C(\by,\omega)\big)\big\Vert_{L^{\infty}(\mathrm{J};\mathbf{L}^{2}(\Omega)\times\mathrm{Q})}\\
			\qquad\qquad\qquad\,+\, \big\Vert\mathscr{L}^C(\bz,\chi)-\mathscr{L}^C(\by,\omega)\big\Vert_{L^{2}(\mathrm{J};\mathbf{H}\times\mathcal{V})} \\
			\, \leq\, \mathcal{L}_{\mathscr{L}^C}\, \Big\{\Vert\bz-\by\Vert_{L^{\infty}(\mathrm{J};\mathbf{V})}\,\Vert \mathscr{L}^C_1(\by,\omega)\Vert_{L^{2}(\mathrm{J};\mathbf{H})}+\Vert\chi-\omega\Vert_{L^{\infty}(\mathrm{J};\mathrm{Q})}\Vert \mathscr{L}^C_3(\by,\omega)\Vert_{L^{2}(\mathrm{J};\mathrm{Q})}   \Big\}\,.
		\end{array}
	\end{equation}
\end{lemma}
\begin{proof}
It follows similarly to the proof of Lemma \ref{l_lip.S}.
\end{proof}
With Lemmas \ref{l_lip.S} and \ref{l_lip.T} established, we now prove the continuity of the fixed-point operator $\mathscr{T}$ on the closed ball $\mathbf{W}(r)$ (cf. \eqref{eq:def.ball}). For $(\bz,\chi), (\by,\omega) \in \mathbf{W}(r)$, using the definition of $\mathscr{T}$ (cf. \eqref{eq:def.opr.E}), integrating from $0$ to $t\in \mathrm{J}$, and applying the continuity bounds from Lemmas \ref{l_lip.S} and \ref{l_lip.T}, we get
\begin{equation*}
	\begin{array}{c}
		\disp\int_{0}^{t}\big\Vert \mathscr{T}(\bz(s),\chi(s))-\mathscr{T}(\by(s),\omega(s))\big\Vert_{\mathbf{V}\times\mathrm{Q}}^{2} \, ds \\
		\,\le\,
		\mathcal{L}_{\mathscr{L}^F}\,  \Vert\bz-\by\Vert_{L^{\infty}(\mathrm{J};\mathbf{V})}\Big\{\big\Vert \mathscr{L}^F_1(\by,\mathscr{L}^C_3(\by,\omega)) \big\Vert_{L^{2}(\mathrm{J};\mathbb{H})}+\big\Vert\mathscr{L}^F_{2}(\by,\mathscr{L}^C_3(\by,\omega))\big\Vert_{L^{2}(\mathrm{J};\mathbf{V})}\Big\}\\
		\,+\, (\mathcal{L}_{\mathscr{L}^F}+1)\mathcal{L}_{\mathscr{L}^C}\, \Big\{\Vert\bz-\by\Vert_{L^{\infty}(\mathrm{J};\mathbf{V})}\,\Vert \mathscr{L}^C_1(\by,\omega)\Vert_{L^{2}(\mathrm{J};\mathbf{H})}+\Vert\chi-\omega\Vert_{L^{\infty}(\mathrm{J};\mathrm{Q})}\Vert \mathscr{L}^C_3(\by,\omega)\Vert_{L^{2}(\mathrm{J};\mathrm{Q})}   \Big\}\,.
	\end{array}
\end{equation*}

From the {\it a priori} estimates for $\mathscr{T}^F$ (Lemma \ref{l_wel.S}, Eq.~\eqref{eq:apr1}) and $\mathscr{T}^C$ (Lemma \ref{l_wel.T}, Eq.~\eqref{eq:apr2}), we deduce the existence of a constant $\mathcal{L}_{\mathscr{T}}$, depending only on $\mathcal{L}_{\mathscr{T}^F}$, $\mathcal{L}_{\mathscr{T}^C}$, $\mathcal{C}_{\mathscr{T}^F}$, and $\mathcal{C}_{\mathscr{T}^C}$, such that for all $(\bz,\chi), (\by,\omega) \in \mathbf{W}(r)$, the following holds
\begin{equation}\label{eq:lip.Xi}
	\begin{array}{c}
\big\Vert \mathscr{T}(\bz,\chi)-\mathscr{T}(\by,\omega)\big\Vert_{L^{2}(\mathrm{J};\mathbf{V}\times\mathrm{Q})}  \leq
		\mathcal{L}_{\mathscr{T}}\, \widetilde{\mathcal{N}}\big(\bu_{0},\varphi_{0},\mathbf{g},c_{\mathscr{F}^C}\big)\,  \Big(\Vert\bz-\by\Vert_{L^{\infty}(\mathrm{J};\mathbf{V})}+\Vert\chi-\omega\Vert_{L^{\infty}(\mathrm{J};\mathrm{Q})}\Big)\,.
	\end{array}
\end{equation}

Owing to the above analysis, we now establish the main result of this section.
\begin{theorem}
Assume the data are small enough to satisfy \eqref{eq:as1.data}, and that also holds
	\begin{equation}\label{eq:as2.data}
		\begin{array}{c}
			\mathcal{L}_{\mathscr{T}}\, \widetilde{\mathcal{N}}\big(\bu_{0},\varphi_{0},\mathbf{g},c_{\mathscr{F}^C}\big)\,\leq\, 1\,.
		\end{array}
	\end{equation}
	
Then, problem \eqref{eq:v1a-eq:v4} admits a unique solution $ (\bsi(t),\bu(t))\in \mathbb{H}\times\mathbf{V} $ and $ (\brho(t),\vec{\varphi}(t))\in \mathbf{H}\times\mathcal{V} $ with $ (\bu(t),\lambda(t))\in \mathbf{W}(r) $ for a.e.\ $ t\in \mathrm{J} $. Furthermore, the following a~{\it priori} estimates hold:
	\begin{subequations}
		\begin{alignat}{2}
			\Vert\bsi\Vert_{L^{2}(\mathrm{J};\mathbb{H})}+\Vert\bsi\Vert_{L^{\infty}(\mathrm{J};\mathbb{L}^2)}+\Vert\bu\Vert_{L^{\infty}(\mathrm{J};\mathbf{L}^{2})}+\Vert\partial_t\bu\Vert_{L^{2}(\mathrm{J};\mathbf{L}^{2})}+\Vert\bu\Vert_{L^{2}(\mathrm{J};\mathbf{V})}
			\leq \mathcal{C}_{1}\,\widetilde{\mathcal{N}}\big(\bu_{0},\varphi_{0},\mathbf{g},c_{\mathscr{F}^C}\big)\,,\label{eq:apr.est1}\\[1ex]
			\Vert \brho\Vert_{L^{2}(\mathrm{J};\mathbf{H})}+\Vert \brho\Vert_{L^{\infty}(\mathrm{J};\mathbf{L}^2)}+\Vert\varphi\Vert_{L^{\infty}(\mathrm{J};\mathrm{L}^{2})}+\Vert\partial_t \vec{\varphi}\Vert_{L^{2}(\mathrm{J};\mathrm{L}^{2}\times\mathrm{Q})}+\Vert\vec{\varphi}\Vert_{L^{2}(\mathrm{J};\mathrm{V}\times\mathrm{Q})}
			\, \leq\,  \mathcal{C}_{2} \mathcal{N}_T(\varphi_{0},c_{\mathscr{F}^C})\,.\label{eq:apr.est2}
		\end{alignat}
	\end{subequations}
\end{theorem}
\begin{proof}
Recalling from Lemma \ref{l_cons.W} that \eqref{eq:as1.data} ensures $\mathscr{T}$ maps $\mathbf{W}(r)$ into itself, and given the equivalence between \eqref{eq:v1a-eq:v4} and \eqref{eq:fixp}, along with the Lipschitz continuity of $\mathscr{T}$ (cf. \eqref{eq:lip.Xi}) and assumption \eqref{eq:as2.data}, the Banach fixed-point theorem guarantees a unique solution $(\bu(t),\lambda(t)) \in \mathbf{W}(r)$ to \eqref{eq:fixp} for a.e.\ $t \in \mathrm{J}$. Moreover, the a~{\it priori} bounds \eqref{eq:apr.est1}--\eqref{eq:apr.est2} follow directly from \eqref{eq:apr1} and \eqref{eq:apr2}.
\end{proof}
\section{Fully-Discrete Scheme}\label{s:4}
This section presents the Galerkin scheme approximating \eqref{eq:v1a-eq:v4} and analyzes its solvability via a discrete fixed-point method.
\subsection{Preliminaries}\label{s.4.1}
Let \( \{\mathcal{T}_h\}_{h>0} \) be a sequence of triangulations of \( \Omega \) into triangles \( T \) with diameters \( h_T \), and define \( h := \max_{T \in \mathcal{T}_h} h_T \) as usual. The induced partition of \( \Gamma_{\mathrm{in}}^{\mathrm{c}} \) is denoted by \( \Gamma_{\mathrm{in}, h}^{\mathrm{c}} \), assumed to have an even number of edges (modified if necessary). The coarser partition \( \Gamma_{\mathrm{in}, 2h}^{\mathrm{c}} \) is formed by merging adjacent edge pairs. For \( k \geq 0 \), let \( \mathrm{P}_k(T) \) be the space of polynomials up to degree \( k \) on \( T \), with vector and tensor counterparts \( \mathbf{P}_k(T) \) and \( \mathbb{P}_k(T) \). The local Raviart–Thomas space of order \( k \) is defined as \( \mathbf{RT}_k(T) := \mathbf{P}_k(T) + \mathrm{P}_k(T) \mathbf{x} \). Using these, we define the finite element spaces for each decoupled problem as follow:
\begin{subequations}
	\begin{alignat}{2}
		\mathbb{H}_{h}&\,:=\, \Big\{
		\bta_h \in \mathbb{H}(\bdiv;\Omega):\quad \bta_{i,h}|_{T}\in\mathbf{RT}_{0}(T),\,\,\,\forall\, i\in \{1,2\},\quad\forall\, T\in\mathcal{T}_{h}\Big\}\cap \mathbb{H}_{\Gamma_{\mathtt{o}}}(\bdiv_{4/3};\Omega)\,,\\
		\mathbf{V}_{h}&\,:=\,\Big\{\bv_{h}\in\mathbf{L}^{2}(\Omega):\quad\bv_{h}|_{T}\in\mathbf{P}_{0}(T)\quad\forall\, T\in\mathcal{T}_{h}\Big\}\,,\\
			\mathbf{H}_{h}&\,:=\,  \Big\{
		\bet_h\in\mathbf{H}(\div;\Omega):\quad \bet_h |_{T}\in \mathbf{RT}_{0}(T)\quad\forall\, T\in\mathcal{T}_h
		\Big\}\,,\\
		\mathrm{V}_h &\,:=\,\Big\{
		\psi_h\in\mathrm{L}^{2}(\Omega):\quad\psi_h |_{T}\in \mathrm{P}_{0}(T)\quad\forall\, T\in \mathcal{T}_h \Big\}\,.
	\end{alignat}
\end{subequations}
%
Similarly as for continuous case, we define the space $\mathcal{V}_h\, :=\,\mathrm{V}_h\times\mathrm{Q}_h$		 and set the notation $		\vec{\varphi}_{h}\, :=\, (\varphi_{h} , \lambda_{h})\,,\,\, \vec{\psi}_{h}\,:=\, (\psi_{h} , \xi_{h}) \,,\,\, \vec{\phi}_{h}\,:=\, (\phi_{h} , \chi_{h})\in \mathcal{V}_{h}$.
Next,  applying the backward Euler method with uniform time steps  $ t_n =n \Delta t$, $ n=1,\cdots,N $ and defining $ f^{n}:=f(\cdot,t_n) $, $ \delta_t f^{n}:=(f^{n}-f^{n-1})/\Delta t $  for a generic function $ f $, we derive the fully discrete scheme:
Given initial data $ (\bu_h^0 , \varphi_h^0)\in \mathbf{V}_h \times\mathrm{V}_h $,
find $ (\bsi_h^{n},\bu_h^{n})\in \mathbb{H}_h \times\mathbf{V}_h$ and $ (\brho_h^{n},\vec{\varphi}_h^{n})\in \mathbf{H}_h  \times\mathcal{V}_h$, for each $ n=1,\cdots,N $, such that 
\begin{equation}\label{eq:fuldis-v1a-eq:v4}
	\begin{split}
		\mathscr{A}^{F}(\bsi_h^n , \bta_h)+\mathscr{B}^{F}(\bta_h, \bu_h^n)&= \mathscr{F}_{\lambda_h^n}^{F}(\bta_h)\,,\\
		\disp \int_{\Omega}\delta_t \bu_h^n \cdot\bv_h-\mathscr{B}^{F}(\bsi_h^n,\bv_h)-\mathscr{O}^{F}_1(\bu_h^n;\bsi_h^n, \bv_h)+\mathscr{O}^{F}_2(\bu_h^n;\bu_h^n, \bv_h)&=0   \,,\\
		\mathscr{A}^{C}(\brho_h^n , \bet_h)+\mathscr{B}^{C}(\bet_h,\vec{\varphi}_h^n)&= 0\,,\\
		\disp \int_{\Omega}\delta_t \varphi_h^n \, \psi_h-	\mathscr{B}^{C}(\brho_h^n,\vec{\psi}_h)+\mathscr{C}^{C}(\vec{\varphi}_h^n, \vec{\psi}_h)-\mathscr{O}^{C}_1 (\bu_h^n;\brho_h^n,\vec{\psi}_h)+\mathscr{O}^{C}_2(\lambda_h^n;\vec{\varphi}_h^n, \vec{\psi}_h)&=\mathscr{F}^C(\vec{\psi}_h) \,,
	\end{split}
\end{equation}
for any $ (\bta_h, \bv_h)\in\mathbb{H}_h \times\mathbf{V}_h$ and $ (\bet_h, \vec{\psi}_h)\in \mathbf{H}_h \times\mathrm{V}_h$.

We adopt the discrete version of the strategy from Sec.~\ref{s:fix} to analyze the solvability of \eqref{eq:fuldis-v1a-eq:v4}, defining the discrete operator \( \mathscr{L}^F_{\mathtt{d}}: \mathbf{V}_{h} \times \mathrm{Q}_{h} \to \mathbb{H}_h \times \mathbf{V}_h \) by
\begin{equation*}
	\mathscr{L}^F_{\mathtt{d}}(\bz_h^{n}, \chi_h^{n})\, =\,\big(\mathscr{L}^F_{1,\mathtt{d}}(\bz_h^{n}, \chi_h^{n}),\, \mathscr{L}^F_{2,\mathtt{d}}(\bz_h^{n}, \chi_h^{n})\big)\, :=\,
	(\bsi_{h,\star}^{n}, \bu_{h,\star}^{n})\,,
\end{equation*}

for each $ 1\leq n \leq N $, and $ (\bz_h^{n}, \chi_h^{n})\in  \mathbf{V}_{h}\times \mathrm{Q}_{h} $, where $ (\bsi_{h,\star}^{n}, \bu_{h,\star}^{n})\in \mathbb{H}_h \times\mathbf{V}_h $ satisfying
\begin{equation}\label{eq:dis.sub1}
	\begin{array}{rcll}
		\mathscr{A}^F(\bsi_{h,\star}^{n} , \bta_h)+\mathscr{B}^F(\bta_h, \bu_{h,\star}^{n})&=& \mathscr{F}^F_{\chi_h^{n}}(\bta_h)\,,\\
		\disp \int_{\Omega}\delta_t \bu_{h,\star}^{n} \cdot\bv_h-\mathscr{B}^F(\bsi_{h,\star}^{n},\bv_h)-\mathscr{O}_1^F(\bz_h^{n};\bsi_{h,\star}^{n}, \bv_h)+\mathscr{O}_2^F(\bz_h^{n};\bu_{h,\star}^{n}, \bv_h)&=&0  \,,
	\end{array}
\end{equation}
for any $ (\bta_h, \bv_h)\in\mathbb{H}_h \times\mathbf{V}_h $, with $\bu_{h,\star}^{0}=\Pcalbf_{0}^{h}\bu_0$.

In addition, we also let $ \mathscr{L}^C_{\mathtt{d}}: \mathbf{V}_{h}\times \mathrm{Q}_{h}\rightarrow  \mathbf{H}_h \times\mathcal{V}_h $ be the discrete operator given by
\begin{equation*}
	\mathscr{L}^C_{\mathtt{d}}(\bz_h^{n}, \chi_h^{n})\, =\, \big(\mathscr{L}^C_{1,\mathtt{d}}(\bz_h^{n}, \chi_h^{n}),\, \mathscr{L}^C_{2,\mathtt{d}}(\bz_h^{n}, \chi_h^{n}),\, \mathscr{L}^C_{3,\mathtt{d}}(\bz_h^{n}, \chi_h^{n})\big)\,:=\,
	(\brho_{h,\star}^{n}, \vec{\varphi}_{h,\star}^{n}) \,,
\end{equation*}

for each $1\leq n\leq N $ and $ (\bz_h^{n}, \chi_h^{n})\in\mathbf{V}_{h}\times \mathrm{Q}_{h} $, where $ (\brho_{h,\star}^{n}, \vec{\varphi}_{h,\star}^{n})\in  \mathbf{H}_h \times\mathcal{V}_h $  satisfying
\begin{equation}\label{eq:dis.sub2}
	\begin{split}
		\mathscr{A}^{C}(\brho_{h,\star}^{n} , \bet_h)+\mathscr{B}^{C}(\bet_h,\vec{\varphi}_{h,\star}^{n})=& \mathscr{F}_1^C(\bet_h)\,,\\
		\disp \int_{\Omega}\delta_t \varphi_{h,\star}^{n} \, \psi_h-	\mathscr{B}^{C}(\brho_{h,\star}^{n},\vec{\psi}_h)+\mathscr{C}^{C}(\lambda_{h,\star}^{n}, \xi_h)-\mathscr{O}^{C}_1 (\bz_h^{n};\brho_{h,\star}^{n},\psi_h)
		+\mathscr{O}^{C}_2(\chi_h^{n};\lambda_{h,\star}^{n}, \xi_h)=&\mathscr{F}_2^C(\xi_h)  \,,
	\end{split}
\end{equation}

for all $ (\bet_h, \vec{\psi}_h)\in \mathbf{H}_h \times\mathcal{V}_h$, with $\varphi_{h,\star}^0 \,=\,\mathcal{P}_0^h(\varphi_{0})$.
\smallskip

Finally, we define $ \mathscr{T}_{\mathtt{d}} :\mathbf{V}_{h}\times\mathrm{Q}_{h}\rightarrow \mathbf{V}_{h}\times\mathrm{Q}_{h}$ as
\begin{equation}\label{eq:def.dis.opr.E}
	\mathscr{T}_{\mathtt{d}}(\bz_h^{n}, \chi_h^{n})\, =\,  \big( \mathscr{L}^F_{2,\mathtt{d}}(\bz_h^{n} , \mathscr{L}^C_{3,\mathtt{d}}(\bz_h^{n}, \chi_h^{n})),\, \mathscr{L}^C_{3,\mathtt{d}}(\bz_h^{n}, \chi_h^{n})\big)\,,
\end{equation}
and realize that solving  \eqref{eq:fuldis-v1a-eq:v4} is equivalent to finding a fixed point of $ 	\mathscr{T}_{\mathtt{d}} $,
i.e., given initial data $ (\bu_h^0 , \varphi_h^0)\in \mathbf{V}_h \times\mathrm{V}_h $, find $ (\bu_{h}^{n},\lambda_{h}^n)\in \mathbf{V}_h\times\mathrm{Q}_h $ such that
\begin{equation}\label{eq:dis.fixp}
	\begin{array}{c}
		\mathscr{T}_{\mathtt{d}}(\bu_h^{n} , \lambda_h^{n})\, =\,(\bu_h^{n} , \lambda_h^{n})\qquad\forall \, 1\leq n\leq N \,.
	\end{array}
\end{equation}
\subsection{Well-posedness of the discrete operators \( \mathscr{L}^F_{\mathtt{d}} \) and \( \mathscr{L}^C_{\mathtt{d}} \)}
Following Sec.~\ref{sec.2.2}, we establish the well-posedness of discrete systems \eqref{eq:dis.sub1} and \eqref{eq:dis.sub2}. 
First, we recall from \cite[Lemma 4.3]{cgo-NMPDE-2021} the existence of a constant \( \beta_{F,\mathtt{d}} > 0 \), independent of \( h \), such that
\begin{equation}\label{eq:dis.inf.b}
	\sup_{\mathbf{0}\neq \bta_h\in \mathbb{H}_h}\dfrac{\mathscr{B}^F(\bta_h, \bv_h)}{\Vert \bta_h\Vert_{\mathbb{H}}}\,\geq\,\beta_{F,\mathtt{d}}\,\Vert \bv_h\Vert_{\mathbf{V}}\qquad\forall\, \bv_h\in \mathbf{V}_h \,.
\end{equation}

The following result establishes the discrete version of \eqref{eq:inf.bit}.
\begin{lemma}
	There exists a constant $\beta_{C,\mathtt{d}}$, independent of $h$, such that
	\begin{equation}\label{eq:dis.inf.bit}
		\sup_{\mathbf{0}\neq \bet_h\in \mathbf{H}_h}\dfrac{\mathscr{B}^C(\bet_h, \vec{\psi}_h)}{\Vert \bet_h\Vert_{\mathbf{H}}}\,\geq\,\beta_{C,\mathtt{d}}\,\Vert \vec{\psi}_h\Vert_{\mathcal{V}}\qquad\forall\, \psi_h\in \mathcal{V}_h \,.
	\end{equation}
\end{lemma}
\begin{proof}
First, from \cite[Lemma 5.2]{gos-IMA-2012}, there exists a constant \(\beta_{1,C,\mathtt{d}} > 0\), independent of \(h\), such that
	\begin{equation}\label{l1}
		\sup_{\mathbf{0}\neq \bet_h\in \mathbf{H}_h}\dfrac{\mathscr{B}^C(\bet_h, \vec{\psi}_h)}{\Vert \bet_h\Vert_{\div;\Omega}}\,\geq\,\beta_{1,C,\mathtt{d}}\,\left(\Vert\psi_h\Vert_{0,\Omega}+\Vert\xi_{h}\Vert_{\mathrm{Q}}\right)\quad\forall\, \vec{\psi}_h \in\mathcal{V}_h \,.
	\end{equation}
Proceeding as in~\cite[Sec.~4.4.1]{gors-CMA-2021} and using
$\Vert \bet_h \Vert_{\mathbf{H}} \le c_{\Omega} \Vert \bet_h \Vert_{\div;\Omega}$
for all \(\bet_h \in \mathbf{H}(\div;\Omega)\) with constant \(c_{\Omega}>0\), along with~\eqref{l1}, we deduce that
	\begin{equation}\label{l2}
		\sup_{\mathbf{0}\neq \bet_h\in \mathbf{H}_h}\dfrac{\mathscr{B}^C(\bet_h, \vec{\psi}_h)}{\Vert \bet_h\Vert_{\mathbf{H}}}\,\geq\,\dfrac{\beta_{1,C,\mathtt{d}}}{c_{\Omega}}\,\Vert\xi_{h}\Vert_{\mathrm{Q}}\quad\forall\, \vec{\psi}_h \in\mathcal{V}_h \,.
	\end{equation}
From~\cite[Lemma 4.1]{cov-Calcolo-2020}, there exists a constant \( \beta_{2,C,\mathtt{d}} \), independent of \( h \), such that
	\begin{equation}\label{l3}
		\sup_{\mathbf{0}\neq \bet_h\in \mathbf{H}_h}\dfrac{\mathscr{B}^C(\bet_h, \vec{\psi}_h)}{\Vert \bet_h\Vert_{\mathbf{H}}}\,\geq\,	\sup_{\mathbf{0}\neq \bet_h\in \mathbf{H}_h}\dfrac{\mathscr{B}^C(\bet_h, (\psi_h,0))}{\Vert \bet_h\Vert_{\mathbf{H}}}\,\ge\,\beta_{2,C,\mathtt{d}}\,\Vert\psi_{h}\Vert_{\mathrm{V}}\,.
	\end{equation}
	Therefore, a direct combination of~\eqref{l2} and~\eqref{l3} leads to \eqref{eq:dis.inf.bit} with $ \beta_{C,\mathtt{d}}=\frac{1}{2}\min\{\beta_{1,C,\mathtt{d}}/c_{\Omega}, \, \beta_{2,C,\mathtt{d}}\}$.
\end{proof}
The continuity of the forms in \eqref{eq:fuldis-v1a-eq:v4} and coercivity of \( \mathscr{A}^F \), \( \mathscr{A}^C \) hold on discrete spaces. As in the continuous case (cf. \eqref{n1}-\eqref{n2}), the pairs \( (\bsi_{h,\star}^{n}, \bu_{h,\star}^{n}) \in \mathbb{H}_h \times \mathbf{V}_h \) and \( (\brho_{h,\star}^{n}, \vec{\varphi}_{h,\star}^{n}) \in \mathbf{H}_h \times \mathcal{V}_h \) uniquely solve~\eqref{eq:dis.sub1} and~\eqref{eq:dis.sub2}, so \( \mathscr{L}_{\mathtt{d}}^F \) and \( \mathscr{L}_{\mathtt{d}}^C \) are well-defined.
We now present discrete counterparts of Lemmas~\ref{l_wel.S} and~\ref{l_wel.T}.
\begin{lemma}\label{l_wel.dis.S}
	Assume that for each $ 1\leq n\leq N $, and $ (\bz_h^{n},\chi_h^{n})\in \mathbf{V}_{h}\times \mathrm{Q}_{h} $ satisfy
	\begin{equation}\label{eq:dis.as.z.var}
		\Vert\bz_h^n\Vert_{\mathbf{V}}\,\le\,\max\Big\{ \left(\dfrac{\beta_{F,\mathtt{d}}^2\nu}{8}|\Omega|^{(p-4)/4}\right)^{1/(p-2)},\,\dfrac{\beta_{F,\mathtt{d}}\nu}{2\sqrt{2}} \Big\}\,.
	\end{equation}
	
	Then, there exist the positive constant $ \mathcal{C}_{\mathscr{L}^F,\mathtt{d}} $, depending only on $ \nu $, $ \beta_{F,\mathtt{d}} $, such that there hold
	\begin{equation}\label{eq:dis.apr1}
		\begin{array}{c}
			\Vert\mathscr{L}^F_{2,\mathtt{d}}(\bz_h^n,\chi_h^n)\Vert_{0,\Omega}^{2}+\Vert\bdev(\mathscr{L}^F_{1,\mathtt{d}}(\bz_h^n,\chi_h^n))\Vert_{0,\Omega}^{2}
			\\
			+\Delta t\disp\sum_{j=1}^{n}\Big(\Vert\mathscr{L}^F_{1,\mathtt{d}}(\bz_h^j,\chi_h^j)\Vert_{\mathbb{H}}^{2}+\Vert\delta_t\mathscr{L}^F_{2,\mathtt{d}}(\bz_h^j,\chi_h^j)\Vert_{0,\Omega}^{2}\,+\,\Vert\mathscr{L}^F_{2,\mathtt{d}}(\bz_h^j,\chi_h^j)\Vert_{\mathbf{V}}^{2}\Big)\\
			\,\leq\, \mathcal{C}_{\mathscr{L}^F,\mathtt{d}}\,\Big\{\Vert \bu_{0}\Vert_{\mathbf{H}^{1}(\Omega)}^{2}+\Delta t\disp\sum_{j=1}^{n}\left(\Vert\mathbf{g}^j\Vert_{1/2,\Gamma_{\mathtt{o}}^c}^2+\Vert\chi_h^{j}\Vert_{\mathrm{Q}}^{2}+\Vert\delta_t\mathbf{g}^j\Vert_{1/2,\Gamma_{\mathtt{o}}^c}^2+\Vert\delta_t\chi_h^{j}\Vert_{\mathrm{Q}}^{2}\right) \Big\}\\
			\,:=\, \mathcal{C}_{\mathscr{L}^F,\mathtt{d}}\,\mathcal{N}_{F,\mathtt{d}}(\bu_{0},\mathbf{g},\chi_{h})\,.
		\end{array}
	\end{equation}
\end{lemma}
\begin{proof}
Let \( (\bz_h^n, \chi_h^n)\in \mathbf{V}_{h}\times\mathrm{Q} \) for \( n=1,\dots,N \). As in the continuous case, testing \eqref{eq:dis.sub1} with \( (\bta_h , \bv_{h}) := (\bsi_{h,\star}^{n},\bu_{h,\star}^{n}) \in \mathbb{H}_h \times \mathbf{V}_h \) and using the discrete inf-sup condition for \( \mathscr{B}^F \) (cf.~\eqref{eq:dis.inf.b}) yields
	\begin{equation}\label{eq:dis.j1}
		\begin{array}{c}
			\dfrac{1}{2\Delta t}\left(\Vert \bu_{h,\star}^n\Vert_{0,\Omega}^{2}-\Vert \bu_{h,\star}^{n-1}\Vert_{0,\Omega}^{2}\right)\,+\,\dfrac{1}{\nu}\, \Vert \bdev(\bsi_{h,\star}^{n})\Vert_{0,\Omega}^{2}
			\,\leq\, c_{\mathscr{F}^F}\left(\Vert\mathbf{g}^n\Vert_{1/2,\Gamma_{\mathtt{o}}^c}+a_1\Vert\chi_h^n\Vert_{\mathrm{Q}}\right)\Vert\bsi_{h,\star}^{n}\Vert_{\mathbb{H}}
			\\ \,+\,\Big(\dfrac{1}{\nu}\Vert\bz_{h}^n\Vert_{\mathbf{V}}\,\Vert\bdev(\bsi_{h,\star}^{n})\Vert_{0,\Omega}
			\,+\,|\Omega|^{(4-p)/4}\Vert\bz_{h}^n\Vert_{\mathbf{V}}^{p-2}\Vert\bu_{h,\star}^{n}\Vert_{\mathbf{V}}\Big)\Vert\bu_{h,\star}^{n}\Vert_{\mathbf{V}}\,.
		\end{array}
	\end{equation}
	and
	\begin{equation*}
		\dfrac{\beta_{F,\mathtt{d}}^{2}\nu}{2}\,\Vert \bu_{h,\star}^{n}\Vert_{\mathbf{V}}^{2}\,\leq \, 2\nu c_{\mathscr{F}^F}^2\left(\Vert\mathbf{g}^n\Vert_{1/2,\Gamma_{\mathtt{o}}^c}+a_1\Vert\chi_h^n\Vert_{\mathrm{Q}}\right)^{2}+\dfrac{1}{2\nu}\Vert\bdev(\bsi_{h,\star}^{n})\Vert_{0,\Omega}^{2}\,.
	\end{equation*}
Substituting into \eqref{eq:dis.j1}, applying Young’s inequality, summing over \( n \), and using assumption \eqref{eq:dis.as.z.var}, we obtain
	\begin{equation}\label{eq:dis.f4}
		\begin{array}{c}
			\Vert\bu_{h,\star}^{n}\Vert_{0,\Omega}^{2}+\Delta t\disp\sum_{j=1}^{n}\left(\Vert \bdev(\bsi_{h,\star}^{j})\Vert_{0,\Omega}^{2} + \Vert \bu_{h,\star}^{j}\Vert_{\mathbf{V}}^{2}\right)\\
			\,\leq\,\mathcal{C}_1 \, \Big\{\Vert\bu_{0}\Vert_{0,\Omega}^{2}+\Delta t\disp\sum_{j=1}^{n}\left(\Vert\mathbf{g}^j\Vert_{1/2,\Gamma_{\mathtt{o}}^c}+a_1\Vert\chi_h^j\Vert_{\mathrm{Q}}\right)^{2}  +\epsilon\Delta t\disp\sum_{j=1}^{n}\Vert\bsi_{h,\star}^{j}\Vert_{\mathbb{H}}^{2} \Big\}\,,
		\end{array}
	\end{equation}
	where $ \mathcal{C}_1 $ is a positive constant depending on $ \beta_{F,\mathtt{d}} $, $ \nu $, $ \epsilon  $ and $|\Omega|$.
	
Similarly, using the second row of \eqref{eq:dis.sub1} and the inclusion \( \bdiv(\mathbb{H}_h) \subseteq \mathbf{V}_h \) (cf.~\cite[Lemma 3.7]{g-springer-2014}), an analogue of \eqref{eq:f7} holds:
	\begin{equation}\label{eq:dis.f7}
		\begin{array}{c}
			\Delta t\disp\sum_{j=1}^{n}\Vert\bdiv(\bsi_{h,\star}^j)\Vert_{0,4/3;\Omega}^{2} \, \leq\, \widetilde{C}\,\Delta t\sum_{j=1}^{n}\, \Big\{\Vert \bdev(\bsi_{h,\star}^{j})\Vert_{0,\Omega}^{2}\, +\,\Vert \bu_{h,\star}^j\Vert_{\mathbf{V}}^{2}+\Vert\delta_t  \bu_{h,\star}^{j}\Vert_{0,\Omega}^{2}\Big\}\,.
		\end{array}
	\end{equation}
To bound the last term, we follow the approach of \eqref{eq:f6}, using the time-discrete derivative of the first row of \eqref{eq:dis.sub1} with \( (\bta_h,\bv_{h}) = (\bsi_{h,\star}^{n}, \delta_t \bu_{h,\star}^{n}) \), applying Young's inequality, and summing over \( n \), we obtain	
	\begin{equation}\label{eq:dis.f6}
		\begin{array}{c}
			\Delta t\disp\sum_{j=1}^{n}\Vert\delta_t \bu_{h,\star}^{j}\Vert_{0,\Omega}^{2}+ \Vert \bdev(\bsi_{h,\star}^{n})\Vert_{0,\Omega}^{2} \,\leq\,\mathcal{C}_2 \, \Big\{ \Vert \bdev(\bsi_{h,\star}^{0})\Vert_{0,\Omega}^{2}\\
			+\Delta t\disp\sum_{j=1}^{n}\Big(\Vert\delta_t\mathbf{g}^j\Vert_{1/2,\Gamma_{\mathtt{o}}^c}^2+\Vert\delta_t\chi_h^{j}\Vert_{\mathrm{Q}}^{2}+ \epsilon_1 \Vert\bsi_{h,\star}^{j}\Vert_{\mathbb{H}}^{2}\Big)\,+\,\Delta t\disp\sum_{j=1}^{n}\left(\Vert\bdev(\bsi_{h,\star}^{j})\Vert_{0,\Omega}^{2}+\Vert\bu_{h,\star}^{j}\Vert_{\mathbf{V}}^{2}\right)\Big\}\,.
		\end{array}
	\end{equation}
Substituting \eqref{eq:dis.f6} into \eqref{eq:dis.f7} and combining with \eqref{eq:dis.f4} and \eqref{kkk}, we get
	\begin{equation}\label{k1}
		\begin{array}{c}
			\Delta t	\disp\sum_{j=1}^{n}\Vert\bsi_{h,\star}^{j}\Vert_{\mathbb{H}}^{2} \, \leq\,\mathcal{C}_3\,\Big\{\Vert \bdev(\bsi_{h,\star}^{0})\Vert_{0,\Omega}^{2}+\Vert\bu_{0}\Vert_{0,\Omega}^{2}\\
			\,+\, \Delta t\disp\sum_{j=1}^{n}\left(\Vert\mathbf{g}^j\Vert_{1/2,\Gamma_{\mathtt{o}}^c}^2+\Vert\chi_h^{j}\Vert_{\mathrm{Q}}^{2}+\Vert\delta_t\mathbf{g}^j\Vert_{1/2,\Gamma_{\mathtt{o}}^c}^2+\Vert\delta_t\chi_h^{j}\Vert_{\mathrm{Q}}^{2}\right) \Big\}\,.
		\end{array}
	\end{equation}
	Finally, as in \eqref{eq:a3} and \eqref{eq:a4}, testing the first row of \eqref{eq:dis.sub1} at \( n=0 \) with \( \bta_h = \bsi_{h,\star}^{0} \), applying Green's formula, and noting \( \bu_{h,\star}^{0} = \Pcalbf_0^h\bu_{0} \), we bound \( \Vert \bdev(\bsi_{h,\star}^{0})\Vert_{0,\Omega} \) by \( \Vert \bu_{0} \Vert_{\mathbf{H}^1(\Omega)} \). Combining this with \eqref{k1} and \eqref{eq:dis.f4} yields \eqref{eq:dis.apr1}.
\end{proof}
\begin{lemma}\label{l_wel.dis.T}
	Assume that for each $ 1\leq n\leq N $, and $ (\bz_{h}^{n},\chi_h^n)\in \mathbf{V}_h \times\mathrm{Q}_h $ satisfy
	\begin{equation}\label{eq:dis.as2}
		\Vert (\bz_h^n,\chi_h^n)\Vert_{\mathbf{V}\times\mathrm{Q}}\, \leq\, \max\Big\{\dfrac{\beta_{C,\mathtt{d}}\kappa}{2\sqrt{3}},\,\frac{\beta_{C,\mathtt{d}}^2\kappa}{12c_{\mathscr{O}_2^C}}\Big\}\,.
	\end{equation}
	Then, there exists a positive constant $ \mathcal{C}_{\mathscr{L}^C,\mathtt{d}} $
	only on $  \kappa , \beta_{C,\mathtt{d}} , c_{\mathscr{O}_2^C}, |\Omega|$, such that there hold
	\begin{equation}\label{eq:dis.apr2}
		\begin{array}{c}
			\Vert\mathscr{L}^C_{2,\mathtt{d}}(\bz_h^n,\chi_h^n)\Vert_{0,\Omega}^{2}+\Vert\mathscr{L}^C_{1,\mathtt{d}}(\bz_h^n,\chi_h^n)\Vert_{0,\Omega}^{2}+\Delta t \,\disp\sum_{j=1}^{n}
			\bigg(\Vert \mathscr{L}^C_{1,\mathtt{d}}(\bz_h^j,\chi_h^j)\Vert_{\mathbf{H}}^{2}\\
			\,+\,\Vert\delta_t \vec{\mathscr{L}}_{2,\mathtt{d}}^C (\bz_h^j,\chi_h^j)\Vert_{\mathrm{L}^{2}(\Omega)\times\mathrm{Q}}^{2}+\Vert\vec{\mathscr{L}}_{2,\mathtt{d}}^C(\bz_h^j,\varphi_h^j)\Vert_{\mathcal{V}}^{2}\bigg)\\
			\, \leq\,  \mathcal{C}_{\mathscr{L}^C,\mathtt{d}} \Big\{\Vert \varphi(0)\Vert_{\mathrm{H}^{1}(\Omega)}^{2}+\disp\sum\limits_{j=0}^{n}c_{\mathscr{F}^C}^2\Big\}\,:=\,\mathcal{C}_{\mathscr{L}^C,\mathtt{d}}\,\mathcal{N}_{C,\mathtt{d}}(\varphi_{0},c_{\mathscr{F}^C})\,\,.
		\end{array}
	\end{equation}
\end{lemma}
\begin{proof}
The proof follows similarly to that of Lemma \ref{l_wel.dis.S}.
\end{proof}

\subsubsection{Discrete fixed-point solvability}
We study the solvability of \eqref{eq:dis.fixp} by setting 
$$ r_{\mathtt{d}}:=\min\{r_{1,\mathtt{d}},r_{2,\mathtt{d}}\}\,,\quad\text{where}\quad r_{1,\mathtt{d}}\,:=\,\max\Big\{ \left(\dfrac{\beta_{F,\mathtt{d}}^2\nu}{8}|\Omega|^{(p-4)/4}\right)^{1/(p-2)},\,\dfrac{\beta_{F,\mathtt{d}}\nu}{2\sqrt{2}} \Big\}\qan  
$$
\vspace{-.4cm}
$$ r_{2,\mathtt{d}}\,:=\,\max\Big\{\dfrac{\beta_{C,\mathtt{d}}\kappa}{2\sqrt{3}},\,\frac{\beta_{C,\mathtt{d}}^2\kappa}{12c_{\mathscr{O}_2^C}}\Big\}\,,$$
 and defining
$\mathbf{W}_{\mathtt{d}} := \big\{ (\bz_h^n , \chi_h^n) \in \mathbf{V}_h \times \mathrm{Q}_h : \max_{j=1,\dots,n} \Vert (\bz_h^j , \chi_h^j) \Vert_{\mathbf{V} \times \mathrm{Q}} \leq r_{\mathtt{d}} \big\},$
and derive conditions ensuring that \( \mathscr{T}_{\mathtt{d}} \) maps \( \mathbf{W}_{\mathtt{d}} \) into itself.
For \( (\bz_h^n , \chi_h^n) \in \mathbf{W}_{\mathtt{d}} \), using the definition of \( \mathscr{T}_{\mathtt{d}} \) (cf. \eqref{eq:def.dis.opr.E}) and the a {\it priori} bound \eqref{eq:dis.apr1}, there exists a constant \( \mathcal{C}_{\mathscr{T},\mathtt{d}} \), depending only on \( \mathcal{C}_{\mathscr{T},\mathtt{d},F} \) and \( \mathcal{C}_{\mathscr{T},\mathtt{d},T} \), such that
\begin{equation}\label{eq:dis.l2}
	\begin{array}{c}
		\Delta t \, \disp\sum_{j=0}^{n}\big\Vert \mathscr{T}_{\mathtt{d}}(\bz_h^{j},\chi_h^{j})\big\Vert_{\mathbf{V}\times\mathrm{Q}}^{2}
		\, \leq\,
		\mathcal{C}_{\mathscr{T},\mathtt{d}}\,\widetilde{\mathcal{N}}_{\mathtt{d}}\big(\bu_{0},\varphi_{0},\mathbf{g},c_{\mathscr{F}^C}\big)\,,\quad\text{where}\quad\\
		\qquad\widetilde{\mathcal{N}}_{\mathtt{d}}\big(\bu_{0},\varphi_{0},\mathbf{g},c_{\mathscr{F}^C}\big)\,:=\,\mathcal{N}_{F,\mathtt{d}} \big(\bu_{0},\mathbf{g},0\big)^{2}+\mathcal{N}_{C,\mathtt{d}} (\varphi_{0},c_{\mathscr{F}^C})^2\,.
	\end{array}
\end{equation}
\smallskip

We have then proved the following result.
\begin{lemma}\label{l_dis.cons.Xi}
	Assume that the data are sufficiently small so that
	\begin{equation}\label{eq:dis.as.cont.Xi}
		\begin{array}{c}
			\mathcal{C}_{\mathscr{T},\mathtt{d}}\,\widetilde{\mathcal{N}}_{\mathtt{d}}\big(\bu_{0},\varphi_{0},\mathbf{g},c_{\mathscr{F}^C}\big)<1\,,
		\end{array}
	\end{equation}
	then, there hold $ \mathscr{T}_{\mathtt{d}}\,(\mathbf{W}_{\mathtt{d}})\subseteq \mathbf{W}_{\mathtt{d}} $.
\end{lemma}
Next, we establish the Lipschitz-continuity properties of $ \mathscr{T}_{\mathtt{d},F} $ and $ \mathscr{T}_{\mathtt{d},C} $.
\begin{lemma}\label{l_dis.lip.S}
	There exists a constant $ \mathcal{L}_{\mathscr{L}^F,\mathtt{d}}>0 $, depending on $ \nu $, $ c_{\mathscr{F}^F} $, $ \beta_{F,\mathtt{d}} $, $ \ell_o $, such that there hold
	\begin{equation}\label{eq:dis.lip.S}
		\begin{array}{c}
			\big\Vert\mathscr{L}^F_{2,\mathtt{d}}(\bz_h^n,\chi_h^n)-\mathscr{L}^F_{2,\mathtt{d}}(\by_h^n,\omega_h^n)\big\Vert_{0,\Omega}\\
			\,+\,\Delta t \,\disp\sum_{j=0}^{n}\Big(\big\Vert\delta_t\mathscr{L}^F_{2,\mathtt{d}}(\bz_h^j,\chi_h^j)-\delta_t\mathscr{L}^F_{2,\mathtt{d}}(\by_h^j,\omega_h^j)\big\Vert_{0,\Omega}+\big\Vert\mathscr{L}^F_{\mathtt{d}}(\bz_h^j,\chi_h^j)-\mathscr{L}^F_{\mathtt{d}}(\by_h^j,\omega_h^j)\big\Vert_{\mathbb{H}\times\mathbf{V}}\Big)\\
			\, \leq\, \mathcal{L}_{\mathscr{L}^F,\mathtt{d}}\, \Delta t \,\bigg\{  \Vert\bz_h-\by_h\Vert_{L^{\infty}(\mathrm{J}_n;\mathbf{V})}\disp\sum_{j=0}^{n}\left(\big\Vert \mathscr{L}^F_{1,\mathtt{d}}(\by_h^j,\omega_h^j) \big\Vert_{\mathbb{H}}+\big\Vert\mathscr{L}^F_{2,\mathtt{d}}(\by_h^j,\omega_h^j)\big\Vert_{\mathbf{V}}\right)\\
			\,+\,\disp\sum_{j=1}^{n} \Vert\chi_h^j-\omega_h^j\Vert_{\mathrm{Q}}+\Vert \delta_t(\chi_h^j-\omega_h^j)\Vert_{\mathrm{Q}}\bigg\}\,.
		\end{array}
	\end{equation}
\end{lemma}	
\begin{proof}
The proof is similar to Lemma \ref{l_lip.S} but without time integration. For \( 1 \leq n \leq N \), given \( (\bz_h^n, \chi_h^n) \) and \( (\by_h^{n}, \omega_h^{n}) \), let
$\mathscr{T}_{\mathtt{d},F}(\bz_h^n,\chi_h^n) = (\bsi_{h,\star}^{n}, \bu_{h,\star}^n), \quad \mathscr{T}_{\mathtt{d},F}(\by_h^{n},\omega_h^{n}) = (\bsi_{h,\circ}^{n}, \bu_{h,\circ}^n),$
with \( \bu_{h,\star}^0 = \bu_{h,\circ}^0 \). Defining \( e_{\bsi}^{n} := \bsi_{h,\star}^{n} - \bsi_{h,\circ}^{n} \) and \( e_{\bu}^{n} := \bu_{h,\star}^{n} - \bu_{h,\circ}^{n} \), we obtain
	\begin{equation}\label{eq:dis.hh1}
		\begin{array}{l}
\mathscr{A}^F(e_{\bsi}^n , \bta_h)+\mathscr{B}^F(\bta_h, e_{\bu}^n)= \big(\mathscr{F}^F_{\chi_h^{n}}-\mathscr{F}^F_{\omega_h^{n}}\big)(\bta_h)\,,\\
\disp \int_{\Omega}\delta_t e_{\bu}^n \cdot\bv_h-\mathscr{B}^F(e_{\bsi}^n,\bv_h)-\mathscr{O}_1^F(\bz_h^n;e_{\bsi}^n, \bv_h)-\mathscr{O}_1^F(\by_h^n;\bsi_{h,\circ}^n, \bv_h)\\
\qquad\qquad+\mathscr{O}_2^F(\bz_h^n;e_{\bu}^n, \bv_h)=\mathscr{O}_1^F(\bz_h^n -\by_h^n;\bsi_{h,\circ}^n,\bta_h)- \big[\mathscr{O}_2^F(\bz_h^n;\bu_{h,\circ}^n, \bv_h)-\mathscr{O}_2^F(\by_h^n;\bu_{h,\circ}^n, \bv_h)\big] \,.
		\end{array}
	\end{equation}
Choosing $(\bta_h , \bv_h) = (e_{\bsi}^n , e_{\bu}^n)$, applying the coercivity of $\mathscr{A}^{F}$ and the boundedness of $\mathscr{O}_{1}^{F}$, $\mathscr{O}_2^F$  yield
	\begin{equation}\label{eq:dis.h6}
		\begin{array}{c}
\dfrac{1}{2}\delta_t \big\Vert e_{\bu}^{n}\big\Vert_{0,\Omega}^{2}+\dfrac{1}{\nu}\big\Vert \bdev(e_{\bsi}^{n})\big\Vert_{0,\Omega}^{2}
\,\leq\,\Big\{\dfrac{1}{\nu}\Vert\bz_{h}^n\Vert_{\mathbf{V}}\,\big\Vert \bdev(e_{\bsi}^{n})\big\Vert_{0,\Omega}\\
+|\Omega|^{(4-p)/4}\, \Vert\bz_{h}^n\Vert_{\mathbf{V}}^{p-2}\, \Vert e_{\bu}^{n}\Vert_{\mathbf{V}}\Big\}\Vert\Xi_{\bu}^{n}\Vert_{\mathbf{V}}
			\,+\, \Big\{\dfrac{1}{\nu}\Vert\bdev(\bsi_{h,\circ}^{n})\Vert_{0,\Omega}\\+\ell_o \, (2\Vert\bz_{h}^n\Vert_{\mathbf{V}})^{p-3}\Vert\bu_{h,\circ}^{n}\Vert_{\mathbf{V}}\Big\}\, \Vert\bz_h^{n}-\by_h^{n}\Vert_{\mathbf{V}}\,\Vert e_{\bu}^{n}\Vert_{\mathbf{V}}
			\, +\, c_{\mathscr{F}^F} \, \Vert \chi_h^{n}-\omega_h^{n}\Vert_{\mathrm{Q}}\, \Vert e_{\bsi}^{n}\Vert_{\mathbb{H}}\,.
		\end{array}
	\end{equation}
	
Reasoning as in \eqref{eq:h5} and using the discrete inf-sup condition \eqref{eq:dis.inf.b}, we get
	\begin{equation}\label{eq:dis.h5}
		\dfrac{\beta_{F,\mathtt{d}}^{2}\,\nu}{2}\,\Vert e_{\bu}^{n}\Vert_{\mathbf{V}}^{2}\,\leq\, 2\nu\, c_{\mathscr{F}^F}^{2}\, \Vert\chi_h^n-\omega_h^n\Vert_{\mathrm{Q}}^{2}+\dfrac{1}{2\nu}\Vert\bdev(e_{\bsi}^{n})\Vert_{0,\Omega}^{2}\,.	
	\end{equation}
	
Using \eqref{eq:dis.h5}, \eqref{eq:dis.h6}, Young’s inequality, and the bound \( \max_{1\le j\le n}\Vert\bz_h^j\Vert_{\mathbf{V}} \le r_{1,\mathtt{d}} \), we obtain a constant \( C_0 > 0 \), depending on \( \mathcal{C}_A \), \( \beta_{F,\mathtt{d}} \), and \( \alpha_{\mathtt{d},F} \), such that
	\begin{equation*}
		\begin{array}{c}
			\big\Vert e_{\bu}^{n}\big\Vert_{0,\Omega}^{2}+\Delta t\disp\sum_{j=0}^{n}\big\Vert \bdev(e_{\bsi}^{j})\big\Vert_{0,\Omega}^{2}+\Delta t\disp\sum_{j=0}^{n}\Vert e_{\bu}^{j}\Vert_{\mathbf{V}}^{2}
			\,\leq\,C_0 \, \Big\{\epsilon\,\Delta t\disp\sum_{j=0}^{n}\Vert e_{\bsi}^{j}\Vert_{\mathbb{H}}^{2}\\[1ex]
			\,+\,\Delta t\disp\sum_{j=0}^{n}\left(\Vert\bdev(\bsi_{h,\circ}^{j})\Vert_{0,\Omega}^{2}+\Vert\bu_{h,\circ}^{j}\Vert_{\mathbf{V}}^{2}\right)\Vert\bz_h^{j}-\by_h^{j}\Vert_{\mathbf{V}}^{2}\,+\,\Delta t\disp\sum_{j=0}^{n} \Vert\chi_h^j-\omega_h^{j}\Vert_{\mathrm{Q}}^{2}\Big\}\,.
		\end{array}
	\end{equation*}
	The remainder follows as in Lemma \ref{l_lip.S} (cf. \eqref{eq:r1}–\eqref{eq:r2}).
\end{proof}
\begin{lemma}\label{l_dis.lip.T}
	There exists a constant $ \mathcal{L}_{\mathscr{L}^C,\mathtt{d}}>0 $, depending on $ \beta_{C,\mathtt{d}} $, $ \kappa $, such that
	\begin{equation}\label{eq:dis.lip.T}
		\begin{array}{c}
\Vert\mathscr{L}^C_{2,\mathtt{d}}(\bz_h^j,\chi_h^j)-\mathscr{L}^C_{2,\mathtt{d}}(\by_h^j,\omega_h^j)\Vert_{0,\Omega}^{2}
+\Vert\delta_t\big(\vec{\mathscr{L}}^C_{2,\mathtt{d}}(\bz_h^j,\chi_h^j)-\vec{\mathscr{L}}^C_{2,\mathtt{d}}(\by_h^j,\omega_h^j)\big)\Vert_{\mathbf{L}^{2}(\Omega)\times\mathrm{Q}}^{2}\\
\,+\,\Delta t \disp\sum_{j=1}^{n} \Vert\mathscr{L}^C_{\mathtt{d}}(\bz_h^j,\chi_h^j)-\mathscr{L}^C_{\mathtt{d}}(\by_h^j,\omega_h^j)\Vert_{\mathbf{H}\times\mathcal{V}}^{2}
\, \leq\, \mathcal{L}_{\mathscr{L}^C,\mathtt{d}}\,\Delta t \,  \Big\{\Vert\chi_h-\omega_h\Vert_{L^{\infty}(\mathrm{J};\mathrm{Q})}^{2}\disp\sum_{j=1}^{n}\Vert \mathscr{L}^C_{3,\mathtt{d}}(\by_h^j,\omega_h^j)\Vert_{\mathrm{Q}}^{2}\\
\,+\,
\Vert\bz_h-\by_h\Vert_{L^{\infty}(\mathrm{J}_n;\mathbf{V})}^{2}\,\disp\sum_{j=1}^{n}\Vert \mathscr{L}^C_{1,\mathtt{d}}(\by_h^j,\omega_h^j)\Vert_{\mathbf{H}}^{2}\Big\}\,.
\end{array}
	\end{equation}
\end{lemma}
\begin{proof}
	It proceeds analogously to the proof of Lemma \ref{l_lip.T}. Further details are omitted.
\end{proof}
With Lemmas \ref{l_dis.lip.S} and \ref{l_dis.lip.T} confirmed, it is now possible to verify the Lipschitz continuity of the fixed point operator \( \mathscr{T}_{\mathtt{d}} \) within the closed ball \( \mathbf{W}_{\mathtt{d}} \).
\begin{lemma}
	Let $\widetilde{\mathcal{N}}_{\mathtt{d}}\big(\bu_{0},\varphi_{0},\mathbf{g},c_{\mathscr{F}^C}\big) $ be defined by the second row of \eqref{eq:dis.l2}. Then, there exists $ 	\mathcal{L}_{\mathscr{T},\mathtt{d}} $ is a constant depending on $ \mathcal{C}_{\mathscr{T},\mathtt{d},F} $, $ \mathcal{C}_{\mathscr{T},\mathtt{d},C} $, $ \mathcal{L}_{\mathscr{T},\mathtt{d},F} $ and $ \mathcal{L}_{\mathscr{T},\mathtt{d},C} $ such that
	\begin{equation}\label{eq:dis.lip.E}
		\begin{array}{c}
			\Delta t \,\disp\sum_{j=0}^{n}\big\Vert \mathscr{T}_{\mathtt{d}}(\bz_h^{j},\chi_{h}^j)-\mathscr{T}_{\mathtt{d}}(\by_h^{j},\omega_h^j)\big\Vert_{0,4;\Omega}^{2} 
			\, \leq\,
			\mathcal{L}_{\mathscr{T},\mathtt{d}}\,\widetilde{\mathcal{N}}_{\mathtt{d}}\big(\bu_{0},\varphi_{0},\mathbf{g},c_{\mathscr{F}^C}\big) \,,
		\end{array}
	\end{equation}
\end{lemma}
\begin{proof}
	Recalling the definition of $ \mathscr{T}_{\mathtt{d}} $ (cf. \eqref{eq:def.dis.opr.E}) and using \eqref{eq:dis.lip.S}, \eqref{eq:dis.lip.T}, we conclude that
	\begin{equation*}
		\begin{array}{c}
			\Delta t \,\disp\sum_{j=0}^{n}\big\Vert \mathscr{T}_{\mathtt{d}}(\bz_h^{j},\chi_h^{j})-\mathscr{T}_{\mathtt{d}}(\by_h^{j},\omega_h^{j})\big\Vert_{\mathbf{V}\times\mathrm{Q}}^{2}  \\
			\,\leq\,
			\Delta t \,\mathcal{L}_{\mathscr{L}^F,\mathtt{d}}\,  \Vert\bz_h-\by_h\Vert_{L^{\infty}(\mathrm{J}_n;\mathbf{V})}^{2}\,  \disp\sum_{j=0}^{n}\Big(\big\Vert \mathscr{L}^F_{1,\mathtt{d}}(\by_h^{j},\mathscr{L}^C_{3,\mathtt{d}}(\by_h^{j},\omega_h^{j})) \big\Vert_{\mathbb{H}}^{2}
			+\big\Vert\mathscr{L}^F_{2,\mathtt{d}}(\by_h^j,\mathscr{L}^C_{3,\mathtt{d}}(\by_h^j,\omega_h^j))\big\Vert_{\mathbf{V}}^{2}\Big)\\
			\,+\,\Delta t \, (\mathcal{L}_{\mathscr{L}^F}+1)\mathcal{L}_{\mathscr{L}^C}\, \Big\{\Vert\bz_h-\by_h\Vert_{L^{\infty}(\mathrm{J}_n;\mathbf{V})}^{2}\,\disp\sum_{j=0}^{n}\Vert \mathscr{L}^C_{1,\mathtt{d}}(\by_h^j,\omega_h^j)\Vert_{\mathbf{H}}\\
			\,+\,\Vert\chi_h-\omega_h\Vert_{L^{\infty}(\mathrm{J}_n;\mathrm{Q})}\disp\sum_{j=0}^{n}\Vert \mathscr{L}^C_{3,\mathtt{d}}(\by_h^j,\omega_h^j)\Vert_{\mathrm{Q}}^{2}   \Big\}\,,
		\end{array}
	\end{equation*}
applying Lemmas \ref{l_wel.dis.S} and \ref{l_wel.dis.T} then yields \eqref{eq:dis.lip.E} with
	$	\mathcal{L}_{\mathscr{T},\mathtt{d}}:=\mathcal{L}_{\mathscr{T},\mathtt{d},F}\max\big\{\mathcal{C}_{\mathscr{T},\mathtt{d},F},\max\{\mathcal{C}_{\mathscr{T},\mathtt{d},F},\mathcal{L}_{\mathscr{T},\mathtt{d},F}\}\mathcal{L}_{\mathscr{T},\mathtt{d},C}\big\}.$
	
\end{proof}
We are now in position to establish the following main result.
\begin{theorem}\label{t:dis.exis}
	Assume that data satisfies \eqref{eq:dis.as.cont.Xi} and
		$	\mathcal{L}_{\mathscr{T},\mathtt{d}}\, \widetilde{\mathcal{N}}_{\mathtt{d}}\big(\bu_{0},\varphi_{0},\mathbf{g},c_{\mathscr{F}^C}\big)\,\leq\, 1\,.$
Then, the fully discrete Galerkin scheme \eqref{eq:fuldis-v1a-eq:v4} has a unique solution 
$(\bsi_h^{n}, \bu_h^{n}) \in \mathbb{H}_h \times \mathbf{V}_h, \,\,
(\brho_h^{n}, \vec{\varphi}_h^{n}) \in \mathbf{H}_h \times \mathcal{V}_h,$
with 
\(
(\bu_{h}^{n}, \lambda_{h}^{n}) \in \mathbf{W}_{\mathtt{d}}.
\)
Moreover, there holds
	\begin{subequations}
		\begin{alignat}{2}
		\Vert\bu_h^{n}\Vert_{0,\Omega}+\Vert\bdev(\bsi_h^n)\Vert_{0,\Omega}+\Big(\Delta t \, \disp\sum_{j=1}^{n}\Vert(\bsi_h^{j},\bu_h^{j})\Vert_{\mathbb{H}\times\mathbf{V}}^{2}+\Vert\delta_t\bu_h^{j}\Vert_{0,\Omega}^{2}\Big)^{1/2}
		\leq \mathcal{C}_{1,\mathtt{d}}\widetilde{\mathcal{N}}_{\mathtt{d}}(\bu_{0},\varphi_{0},\mathbf{g},c_{\mathscr{F}^C})\,,\label{eq:dis.apr.est1}\\
			\Vert\varphi_h^{n}\Vert_{0,\Omega}+\Vert \brho_h^{j}\Vert_{0,\Omega}+\Big(\Delta t\disp\sum_{j=1}^{n}\Vert \brho_h^{j}\Vert_{\mathbf{H}}^{2}+\Vert\delta_t \vec{\varphi}_h^{j}\Vert_{\mathrm{L}^{2}\times\mathrm{Q}}^{2}+\Vert\vec{\varphi}_h^{j}\Vert_{\mathcal{V}}^{2}\Big)^{1/2}
		\leq  \mathcal{C}_{2,\mathtt{d}}\, \mathcal{N}_{C,\mathtt{d}}(\varphi_{0},c_{\mathscr{F}^C})\,.\label{eq:dis.apr.est2}
	\end{alignat}
	\end{subequations}
\end{theorem}
\begin{proof}
Assumption \eqref{eq:dis.as.cont.Xi} ensures \( \mathscr{T}_{\mathtt{d}} \) maps \(\mathbf{W}_{\mathtt{d}}\) into itself, and \eqref{eq:dis.lip.E} shows it is continuous. Thus, Brouwer's theorem guarantees existence of a solution to \eqref{eq:dis.fixp} and \eqref{eq:fuldis-v1a-eq:v4}. Under 	$	\mathcal{L}_{\mathscr{T},\mathtt{d}}\, \widetilde{\mathcal{N}}_{\mathtt{d}}\big(\bu_{0},\varphi_{0},\mathbf{g},c_{\mathscr{F}^C}\big)\,\leq\, 1$, uniqueness follows from the Banach fixed-point theorem. The a~{\it priori} estimates \eqref{eq:dis.apr.est1} and \eqref{eq:dis.apr.est2} then follow from \eqref{eq:dis.apr1} and \eqref{eq:dis.apr2}.
\end{proof}
\section{A {\it priori} error analysis}\label{sec.5}
This section derives global errors
\[
\|(\bsi(t_n),\bu(t_n)) - (\bsi_h^n,\bu_h^n)\|_{\mathbb{H}\times\mathbf{V}} \qan \|(\brho(t_n), \vec{\varphi}(t_n)) - (\brho_h^n , \vec{\varphi}_h^n)\|_{\mathbf{H}\times\mathcal{V}},
\]
where $(\bsi(t_n),\bu(t_n))$, $(\brho(t_n), \vec{\varphi}(t_n))$ with $(\bu(t_n),\lambda(t_n)) \in \mathbf{W}$ solve \eqref{eq:v1a-eq:v4} at time step $n$, and $(\bsi_h^n,\bu_h^n) \in \mathbb{H}_h \times \mathbf{V}_h$, $(\brho_h^n, \vec{\varphi}_h^n) \in \mathbf{H}_h \times \mathcal{V}_h$ with $(\bu_h^n, \lambda_h^n) \in \mathbf{W}_{\mathtt{d}}$ solve \eqref{eq:fuldis-v1a-eq:v4}.
To this end, in what follows we recall the interpolation operators.
Given \( k > 1 \), define the spaces 
\begin{equation}\label{key}
	\begin{array}{c}
		\mathbf{Z}_{k} \,:=\,\Big\{\bet\in\mathbf{H}(\div_{k};\Omega):\quad\bet|_{T}\in \mathbf{W}^{1,k}(T)\quad\forall\, T\in \mathcal{T}_h   \Big\}\qan\\
		\mathbb{Z}_{k} \,:=\,\Big\{\bta\in\mathbb{H}(\bdiv_{k};\Omega):\quad\bta|_{T}\in \mathbb{W}^{1,k}(T)\quad\forall\, T\in \mathcal{T}_h   \Big\}\,,
	\end{array}
\end{equation}
and let
\( \Pibf_h : \mathbf{Z}_{4/3} \to \mathbf{H}_h \), 
\( \Pibb_h : \mathbb{Z}_{4/3} \to \mathbb{H}_h \) 
be the vector and tensor Raviart--Thomas interpolants, respectively. Introducing the \( \mathrm{L}^2 \)-projectors 
\( \Pcal_h : \mathrm{L}^2(\Omega) \to \mathrm{V}_h \) and 
\( \Pcalbf_h : \mathbf{L}^2(\Omega) \to \mathbf{V}_h \) 
for the scalar and vector-valued cases, respectively, we recall the commuting diagram property from~\cite[Lemma 1.41]{eg-book-2004}:
\begin{equation}\label{eq:diag}
	\bdiv\big(\Pibb_h(\bta)\big)\, =\, 	\Pcalbf_h\big(\bdiv(\bta)\big) \quad\forall\,\bta\in\mathbb{Z}_{4/3} \qan
	\div\big(\Pibf_h(\bet)\big)\, =\, 	\Pcal_h\big(\div(\bet)\big)\quad\forall\,\bet\in\mathbf{Z}_{4/3} \,.
\end{equation}
\subsection{Error estimate}\label{sec.5.2}
For $\bz \in \mathbf{V}$ and $\xi \in \mathrm{Q}$, define
$\mathsf{M} := \mathbb{H} \times \mathbf{V}$, 
$\mathsf{Z} := \mathbf{H} \times \mathcal{V}$,
and the bilinear forms $\mathbf{A}^F : \mathsf{M} \times \mathsf{M} \to \mathbb{R}$ and $\mathbf{A}^C : \mathsf{Z} \times \mathsf{Z} \to \mathbb{R}$ by
\begin{equation}\label{def.A}
	\begin{array}{l}
		\mathbf{A}_{\bz}^F\big((\bsi,\bu),\,(\bta,\bv)\big):=
		\mathscr{A}^F(\bsi,\bta)+\mathscr{B}^F(\bta,\bu)
		-\mathscr{B}^F(\bsi,\bv)-\mathscr{O}_1^F(\bz;\bsi,\bv)+\mathscr{O}_2^F(\bz;\bu,\bv)\,,\\[1ex]
		\mathbf{A}_{\bz,\xi}^C \big((\brho,\vec{\varphi}),\,(\bet,\vec{\psi})\big):=\mathscr{A}^C(\brho,\bet)+\mathscr{B}^C(\bet, \vec{\varphi})-\mathscr{B}^C(\brho, \vec{\psi})+\mathscr{C}^C(\vec{\varphi},\vec{\psi})\\[.5ex]
		\quad\qquad\qquad\qquad\qquad-\mathscr{O}_1^C(\bz;\brho,\psi)+\mathscr{O}_2^C(\xi;\vec{\varphi},\vec{\psi})\,,
	\end{array}
\end{equation}
we observe that \eqref{eq:v1a-eq:v4} and \eqref{eq:fuldis-v1a-eq:v4} can be reformulated as follow:
\begin{subequations}
	\begin{alignat}{2}
	\hspace{-.2cm}	\disp\int_{\Omega}\partial_t \bu^{n}\cdot\bv+\mathbf{A}_{\bu}^F\big((\bsi,\bu),\,(\bta,\bv)\big)=\mathscr{F}_{\lambda}^F(\bta),\,\,\,
	\disp\int_{\Omega}\delta_t \bu_h^{n}\cdot\bv_h+\mathbf{A}_{\bu_h^n}^{F}\big((\bsi_h^n,\bu_h^n),(\bta_h ,\bv_h)\big)=\mathscr{F}_{\lambda_h^n}^F(\bta_h)\,,\label{eq:subprob1}\\
	\hspace{-.2cm}\disp\int_{\Omega}\partial_t \varphi^{n}\,\psi+\mathbf{A}_{\bu,\lambda}^C\big((\brho^n ,\vec{\varphi}^n), (\bet,\vec{\psi})\big)=\mathscr{F}^{C}(\vec{\psi}),\,\,\,
	\disp\int_{\Omega}\delta_t \varphi_h^{n}\,\psi_h+\mathbf{A}_{\bu_h^n,\lambda_h^n}^C\big((\brho_h^n ,\vec{\varphi}_h^n), (\bet_h,\vec{\psi}_h)\big)=\mathscr{F}^{C}(\vec{\psi}_h)\,.\label{eq:subprob1a}
	\end{alignat}
\end{subequations}
for all $(\bta,\bv)\in\mathsf{M} $, $(\bta_h,\bv_h)\in\mathsf{M}_h$ and
%
 $(\bet,\vec{\psi})\in\mathsf{Z}$, $(\bet_h,\vec{\psi}_h)\in\mathsf{Z}_h $.
Now, we set error functions as
$\btet_{\bsi}^{n}:=	\Pibb_h(\bsi^{n})-\bsi_{h}^{n},\,\,\,
\btet_{\bu}^{n}:=\Pcalbf_h(\bu^{n})-\bu_{h}^{n},\,\,\btet_{\brho}^{n}:=\Pibf_h(\brho^{n})-\brho_{h}^{n},\,\,\theta_{\varphi}^{n}:=
\mathcal{P}_h(\varphi^{n})-\varphi_{h}^{n},\,\,\theta_{\lambda}^n :=\widetilde{\lambda}^n -\lambda_h^n \,,$
and estimate them, starting with the commuting diagram properties from the second column of~\eqref{eq:diag}, which give
\begin{equation}\label{gg1}
	\begin{array}{c}
		\mathscr{B}^F(\bvatet_{\bsi}^{n},\bv_{h})\,=\,0 \quad\forall\,\bv_h \in \mathbf{V}_h\,,
	\end{array}
\end{equation}
By adding and subtracting \( \bsi_h^{n} \) and \( \bu_h^n \) in the second components of \( \mathscr{O}_{1}^F(\bu^n; \bsi^{n}, \bv_h) \) and \( \mathscr{O}_2^F(\bu^n; \bu^n, \bv_h) \), and \( \delta_t \bu^n \) and \( \delta_t \Pcalbf_h(\bu^n) \) in
\(
\int_{\Omega} (\partial_t \bu^n - \delta_t \bu_h^{n}) \cdot \bv_h,
\)
followed by algebraic rearrangements, we rewrite \eqref{eq:subprob1} as
\begin{equation}\label{eq:err.prob1}
	\begin{array}{c}
\disp\int_{\Omega}\delta_t\btet_{\bu}^{n} \cdot\bv_h +\mathbf{A}_{\btet_{\bu}^n}^{F}\big((\btet_{\bsi}^n , \btet_{\bu}^n), (\bta_h , \bv_h)\big)\,=\, 
-\mathscr{A}^{F}(\bvatet_{\bsi}^{n},\bta_h)
+\big(\mathscr{F}_{\lambda^n}^F-\mathscr{F}_{\lambda_{h}^n}^F\big)(\bta_h)\\
-\mathscr{O}_{1}^F(\bu^{n};\bvatet_{\bsi}^n, \bv_h)+\mathscr{O}_{1}^F(\bvatet_{\bu}^n+\btet_{\bu}^n;\bsi_h^{n},\bv_h)-\mathscr{O}_2^F(\bu^{n};\bvatet_{\bu}^n, \bv_h)\\
-\big(\mathscr{O}_2^F (\bu^{n};\bu_{h}^{n},\bv_h)-\mathscr{O}_2^F (\bu_h^{n};\bu_{h}^{n},\bv_h)\big)
-\disp \int_{\Omega} \delta_t\bvatet_{\bu}^n \cdot\bv_h -
		\disp \int_{\Omega}\left(\partial_t \bu^{n}-\delta_t \bu^n\right) \cdot\bv_h \,.
	\end{array}
\end{equation}
Similarly, using \eqref{eq:subprob1a}, the second row of \eqref{def.A}
  gives
\begin{equation}\label{eq:err.prob2}
	\begin{array}{c}
\disp \int_{\Omega}\delta_t\theta_{\varphi}^{n} \,\psi_h +\mathbf{A}^{C}_{\btet_{\bu}^n,\theta_{\lambda}^n}\big((\btet_{\brho}^n ,\theta_{\vec{\varphi}}^n ), (\bet_h , \vec{\psi}_h)\big)
\,=\, -\mathscr{A}^{C}(\bvatet_{\brho}^{n},\bet_h)-\mathscr{O}_{1}^C(\bu^{n};\bvatet_{\brho}^n, \psi_h)\\
+\mathscr{O}_{1}^C(\bvatet_{\bu}^n+\btet_{\bu}^n;\brho_h^{n},\psi_h)-\mathscr{O}_{2}^C(\lambda^n;\vartheta_{\vec{\varphi}}^n , \vec{\psi}_h)-\mathscr{O}_2^C(\vartheta_{\lambda}^n+\theta_{\lambda}^n ; \vec{\varphi}_h^n , \vec{\psi}_h)\\
-\disp \int_{\Omega} \delta_t\vartheta_{\varphi}^n \,\psi_h -
\disp \int_{\Omega}\left(\partial_t \varphi^{n}-\delta_t \varphi^n\right) \,\psi_h\,.
	\end{array}
\end{equation}
The following two lemmas establish {\it a priori} error estimates for \eqref{eq:subprob1} and \eqref{eq:subprob1a}
\begin{lemma}
	Given $ \varpi^{n}\in \mathrm{V} $ and $ \omega_{h}^n\in \mathrm{V}_h $ for each $ 1\leq n\leq N $, let $ (\bsi^{n},\bu^n)\in \mathsf{M} $ and $ (\bsi_h^{n}, \bu_h^n)\in \mathsf{M}_h$ be solutions of the first and second equations of \eqref{eq:subprob1}, respectively, and assume that data satisfy
	\begin{equation}\label{eq:as1.conv}
		\begin{array}{c}
			\widetilde{\mathcal{N}}(\bu_{0},\varphi_{0},\mathbf{g},c_{\mathscr{F}^C})\,\leq\, \dfrac{1}{2\mathcal{C}_0 \,\mathcal{C}_{1}}\min\{1,|\Omega|^{1/4}\}\,.
		\end{array}
	\end{equation}
	Then, the following estimate holds
	\begin{equation}\label{eq:conv1}
		\begin{array}{c}
			\hspace{-.2cm}\Vert\btet_{\bu}^{n}\Vert_{0,\Omega}^{2}+\Delta t\disp\sum_{i=1}^{n}\Vert (\btet_{\bsi}^{i},\btet_{\bu}^i)\Vert_{\mathbb{H}\times\mathbf{V}}^{2} \,\leq\,\mathcal{C}_{\mathtt{est},F}\Big\{\Delta t^{2} +h^2 
		+ \Delta t\disp\sum_{i=1}^{n}\big(\Vert\lambda^{i}-\lambda_{h}^{i}\Vert_{\mathrm{Q}}^{2}+\Vert\delta_t(\lambda^{i}-\lambda_{h}^{i})\Vert_{\mathrm{Q}}^{2}\big)\Big\}\,.
		\end{array}
	\end{equation}
\end{lemma}
\begin{proof}
Taking $(\bta_h, \bv_h) = (\btet_{\bsi}^{n}, \btet_{\bu}^n)$ in \eqref{eq:err.prob1} and using the definition of $\mathbf{A}_{\bz}^F$ and coercivity of $\mathscr{A}^F$, we get
	\begin{equation}\label{eq:no.A}
		\begin{array}{c}
\dfrac{1}{2}\delta_t \Vert \btet_{\bu}^{n}\Vert_{0,\Omega}^{2}+\dfrac{1}{\nu}\,\Vert \bdev(\btet_{\bsi}^{n})\Vert_{0,\Omega}^{2}
\,\leq\, 
\big|\big(\mathscr{F}_{\lambda^n}-\mathscr{F}_{\lambda_h^n}\big)(\btet_{\bsi}^n)\big|+\bigg| \disp \int_{\Omega} \delta_t\bvatet_{\bu}^{n} \cdot\btet_{\bu}^{n}\bigg|\\
\qquad\,+\, \Big|	\disp \int_{\Omega}\left(\partial_t \bu^{n}-\delta_t \bu^n\right) \cdot\btet_{\bu}^{n}\Big|+\big| \mathscr{A}^{F}(\bvatet_{\bsi}^{n},\btet_{\bsi}^{n}) \big|\,+\, 
\big|\mathscr{O}_{1}^F(\bu^{n};\bvatet_{\bsi}^n +\btet_{\bsi}^n,\btet_{\bu}^{n})\big|\\
+\big|\mathscr{O}_{1}^F(\bvatet_{\bu}^n+\btet_{\bu}^n;\bsi_h^{n},\btet_{\bu}^{n})\big|
+\big|\mathscr{O}_2^F(\bu^{n};\bvatet_{\bu}^n+\btet_{\bu}^n, \btet_{\bu}^{n})\big|+\big|\big(\mathscr{O}_2^F (\bu^{n};\bu_{h}^{n},\btet_{\bu}^{n})-\mathscr{O}_2^F (\bu_h^{n};\bu_{h}^{n},\btet_{\bu}^{n})\big)\big|\,.
\end{array}
\end{equation}
Now, we bound the first term on right-hand side of \eqref{eq:no.A} using the continuity of $\mathscr{F}^F$ (cf. \eqref{bund.Ff}) as
		\begin{equation}\label{eq:no.B}
			\big|\big(\mathscr{F}_{\lambda^n}-\mathscr{F}_{\lambda_h^n}\big)(\btet_{\bsi}^n)\big|\, \leq\, c_{\mathscr{F}^F}a_1\Vert\lambda^n-\lambda_h^n\Vert_{\mathrm{Q}}\,\Vert\btet_{\bsi}^n\Vert_{\mathbb{H}} \,.
		\end{equation}
The second and third terms are estimated as in standard finite elements (see \cite{t-Sprin-1984}):
\begin{subequations}
\begin{alignat}{2}
\bigg| \disp \int_{\Omega} \delta_t\bvatet_{\bu}^{n} \cdot\btet_{\bu}^{n}\bigg|\, \leq\,	\Delta t^{-1/2}\Vert \partial_t\bvatet_{\bu}\Vert_{L^{2}(\mathrm{J}_{n},\mathbf{L}^{2}(\Omega))}\, \Vert \btet_{\bu}^{n}\Vert_{0,\Omega}\,,\label{eq:no.C1}\\
	\bigg|	\disp \int_{\Omega}\left(\partial_t \bu^{n}-\delta_t \bu^n\right) \cdot\btet_{\bu}^{n}\bigg|
\, \leq\, 	
\Delta t^{1/2}\left(\disp\int_{t_{n-1}}^{t_n}\Vert \partial_{tt}\bu(s)\Vert_{0,\Omega}^{2}\, ds\right)^{1/2} \,\big\Vert \btet_{\bu}^{n}\big\Vert_{0,\Omega}\,.\label{eq:no.C}
\end{alignat}
\end{subequations}
		Continuity of $\mathscr{A}^{F}$ gives
		\begin{equation}\label{eq:no.Ea}
					\big|\mathscr{A}^{F}(\bvatet_{\bsi}^{n},\btet_{\bsi}^{n})\big|\, \leq \, \dfrac{1}{\nu}\Vert\bvatet_{\bsi}^{n}\Vert_{0,\Omega}\,\Vert \bdev(\btet_{\bsi}^{n})\Vert_{0,\Omega}\,.
		\end{equation} 	
		Using the boundedness of $\mathscr{O}_1^F$ (see \eqref{eq1:bound.c.ct}),  \eqref{eq2:bound.c.ct} and \eqref{eq:bound.o2}, we obtain
		\begin{equation}\label{eq:no.D}
			\begin{array}{c}
 \big|\mathscr{O}_{1}^F(\bu^{n};\bvatet_{\bsi}^n +\btet_{\bsi}^n,\btet_{\bu}^{n})\big|+\big|\mathscr{O}_{1}^F(\bvatet_{\bu}^n+\btet_{\bu}^n;\bsi_h^{n},\btet_{\bu}^{n})\big|\\
\hspace{-.8cm}\,\le\, \dfrac{1}{\nu}\Big\{
\Vert\bu^n\Vert_{\mathbf{V}}\big(\Vert\bdev(\bvatet_{\bsi}^n) \Vert_{0,\Omega}+\Vert\bdev(\btet_{\bsi}^n)\Vert_{0,\Omega}\big)\,+\,\big(\Vert\bvatet_{\bu}^n\Vert_{\mathbf{V}}+\Vert\btet_{\bu}^n\Vert_{\mathbf{V}}\big)\Vert\bdev(\bsi_h^n)\Vert_{0,\Omega}\Big\}\Vert\btet_{\bu}^{n}\Vert_{\mathbf{V}}\,.
			\end{array}
		\end{equation}
\begin{equation}\label{eq:no.D1}
\begin{array}{c}
\big|\mathscr{O}_2^F(\bu^{n};\bvatet_{\bu}^n+\btet_{\bu}^n, \btet_{\bu}^{n})\big|+\big|\big(\mathscr{O}_2^F (\bu^{n};\bu_{h}^{n},\btet_{\bu}^{n})-\mathscr{O}_2^F (\bu_h^{n};\bu_{h}^{n},\btet_{\bu}^{n})\big)\big|\\
\hspace{-.25cm}\,\le\, \Big(|\Omega|^{(4-p)/4}\Vert\bu^n\Vert_{\mathbf{V}}^{p-2}\,\big(\Vert\bvatet_{\bu}^n\Vert_{\mathbf{V}}+\Vert\btet_{\bu}^n\Vert_{\mathbf{V}}\big)
+\ell_o \big(\Vert\bu^n\Vert_{\mathbf{V}}+\Vert\bu_h^n\Vert_{\mathbf{V}}\big)^{p-3}\Vert\bu^n-\bu_h^n\Vert_{\mathbf{V}}\Vert\bu_{h}^{n}\Vert_{\mathbf{V}}\Big)\Vert\btet_{\bu}^{n}\Vert_{\mathbf{V}}\,.
			\end{array}
		\end{equation}	
	Consequently, replacing the estimates \eqref{eq:no.B}-\eqref{eq:no.D1} back into \eqref{eq:no.A}, implies
	\begin{equation}\label{eq:m1}
		\begin{array}{c}
			\dfrac{1}{2}\delta_t\big\Vert\btet_{\bu}^{n}\big\Vert_{0,\Omega}^{2}+\dfrac{1}{\nu}\,\big\Vert\bdev(\btet_{\bsi}^{n})\big\Vert_{0,\Omega}^{2}\,\le\, a_1 c_{\mathscr{F}^F}\big\Vert\lambda^{n}-\lambda_h^{n}\big\Vert_{\mathrm{Q}}\,\big\Vert\btet_{\bsi}^{n}\big\Vert_{\mathbb{H}}\\[1ex]
			\, +\, \dfrac{1}{\nu} \Big( \Vert\bdev(\bvatet_{\bsi}^{n})\Vert_{0,\Omega}+\Vert\bu^{n}\Vert_{\mathbf{V}}\big\Vert\btet_{\bu}^{n}\big\Vert_{\mathbf{V}}\Big)\, \big\Vert \bdev(\btet_{\bsi}^{n})\big\Vert_{0,\Omega}\\[1ex]
			\,+\,\dfrac{1}{\nu}\Vert\bu_h^{n}\Vert_{\mathbf{V}}\, \Vert\bdev(\bvatet_{\bsi}^{n})\Vert_{0,\Omega}\Vert\btet_{\bu}^n\Vert_{\mathbf{V}}
			\,+\, \Big\{   
			\dfrac{1}{\nu}\Vert\bdev(\bsi_h^n)\Vert_{0,\Omega}\\[1ex]
			\,+\,|\Omega|^{(4-p)/4}\Vert\bu^n\Vert_{\mathbf{V}}^{p-2}+\ell_o \big(\Vert\bu^n\Vert_{\mathbf{V}}+\Vert\bu_h^n\Vert_{\mathbf{V}}\big)^{p-3}\Vert\bu_h^n\Vert_{\mathbf{V}}\Big\}\big(\Vert\bvatet_{\bu}^n\Vert_{\mathbf{V}}+\Vert\btet_{\bu}^n\Vert_{\mathbf{V}}\big)\Vert\btet_{\bu}^n\Vert_{\mathbf{V}}\\[1ex]
			\,+\,\Big(\Delta t^{-1/2}\Vert \partial_t\bvatet_{\bu}\Vert_{L^{2}(\mathrm{J}_{n},\mathbf{L}^{2}(\Omega))} +\Delta t^{1/2}\Vert\partial_{tt}\bu\Vert_{L^{2}(\mathrm{J}_n;\mathbf{L}^{2}(\Omega))}\Big)\,\Vert \btet_{\bu}^{n}\Vert_{0,\Omega}\,.
		\end{array}
	\end{equation}
Using the discrete inf-sup condition of $\mathscr{B}^{F}$ and \eqref{eq:err.prob1} with $\bv_{h}=\mathbf{0}$, we obtain a result similar to \eqref{eq:dis.h5} as:	
	\begin{equation}\label{eq:hj}
		\dfrac{\beta_{F,\mathtt{d}}^{2}\,\nu}{2}\,\Vert \btet_{\bu}^{n}\Vert_{0,4;\Omega}^{2}\,\leq\, \dfrac{1}{2\nu}\,\big(\Vert\bdev(\btet_{\bsi}^{n})\Vert_{0,\Omega}^{2}+\bdev(\bvatet_{\bsi}^{n})\Vert_{0,\Omega}^{2}\big)+2\nu a_1^2 c_{\mathscr{F}^F}^2 \Vert\lambda^n -\lambda_h^n\Vert_{\mathrm{Q}}^2 \,.
	\end{equation}
	
Using \eqref{eq:hj}, Young's inequality, and bounding \( \|\bu^n\|_{\mathbf{V}} \), \( \|\bu_h^n\|_{\mathbf{V}} \) by \( r_{1} \), \( r_{1,\mathtt{d}} \), summing over \( n \), with \( \btet_{\bu}^0 = 0 \), and multiplying by \( 2\Delta t \), we get a constant \( \mathcal{C}_0 > 0 \) depending on \( \mathcal{C}_A \), \( \ell_{A_F} \), \( \beta_{F,\mathtt{d}} \), and \( \alpha_{\mathtt{d},F} \) such that
	\begin{equation}\label{eq:m2}
		\begin{array}{c}
			\Vert\btet_{\bu}^{n}\Vert_{0,\Omega}^{2}+\Delta t\disp\sum_{j=1}^{n}\Vert\bdev(\btet_{\bsi}^{j})\Vert_{0,\Omega}^{2}+\Delta t \,\disp\sum_{j=1}^{n}\Vert \btet_{\bu}^{j}\Vert_{\mathbf{V}}^{2}\\
			\, \leq\,
			\mathcal{C}_{0}\Big(
			\Delta t\disp\sum_{j=1}^{n}\Vert\lambda^{j}-\lambda_h^{j}\Vert_{\mathrm{Q}}^{2}+\Big(\Delta t\disp\sum_{j=1}^{n}\Vert\btet_{\bu}^{j}\Vert_{\mathbf{V}}^{2}\Big)\,\big(\Vert\bdev(\bsi_h)\Vert_{L^{\infty}(\mathrm{J} ;\mathbb{L}^{2}(\Omega))}+\Vert\bu_h\Vert_{L^{\infty}(\mathrm{J};\mathbf{V})}\big)\\
			\,+\,\Delta t\disp\sum_{j=1}^{n} \big(\Vert\bvatet_{\bsi}^j\Vert_{0,\Omega}^{2}
			+\Vert\bvatet_{\bu}^j\Vert_{\mathbf{V}}^{2}+\epsilon\Vert\btet_{\bsi}^j\Vert_{\mathbb{H}}^2\big)
			+\Vert \partial_t\bvatet_{\bu}\Vert_{L^{2}([0,t_n];\mathbf{L}^{2}(\Omega))}^{2} +\Delta t^{2} \,\Vert\partial_{tt}\bu\Vert_{L^{2}([0,t_n];\mathbf{L}^{2}(\Omega))}^{2}
			\Big)\,.
		\end{array}
	\end{equation}
	
Applying the \emph{a priori} estimate \eqref{eq:dis.apr.est1}, assumption \eqref{eq:as1.conv}, and the bounds for \(\Vert\btet_{\bsi}^n\Vert_{\mathbb{H}}\) from \eqref{eq:r1}-\eqref{eq:r2} (cf. Lemma~\ref{l_lip.S}), along with the approximation properties of \(\Pibb_h\) \cite[Lemma 4.1]{cgo-NMPDE-2021} and the \(\mathrm{L}^2\)-projector \(\Pcalbf_h\) \cite[Prop.~1.135]{eg-book-2004}, we conclude \eqref{eq:conv1}, completing the proof.
\end{proof}
\begin{lemma}
	Given $ \bu^{n}\in \mathbf{V} $ and $ \bu_{h}^{n}\in \mathbf{V}_h $ for each $ 1\leq n\leq N $, let $ (\brho^{n},\vec{\varphi}^n)\in \mathsf{Z} $ and $ (\brho_h^{n},\vec{\varphi}_h^n)\in \mathsf{Z}_h $ be the solutions of the first and second equations of \eqref{eq:subprob1a}, respectively, and assume that data satisfy
	\begin{equation}\label{eq:as2.conv}
		\begin{array}{c}
			\mathcal{N}_{C}(\varphi_{0},c_{\mathscr{F}^C})\,\leq\,\dfrac{1}{2\widetilde{C}_0 \, \mathcal{C}_{2}}\min\big\{1,|\Omega|^{1/8}\beta_{2,C,\mathtt{d}}\big\}\,.
		\end{array}
	\end{equation}
	Then, the following estimate holds
	\begin{equation}\label{eq:conv2a}
		\begin{array}{c}
			\Vert\theta_{\varphi}^{n}\Vert_{0,\Omega}^{2}+\Delta t\disp\sum_{j=1}^{n}\left(\Vert\btet_{\brho}^{j}\Vert_{0,\Omega}^{2}+\Vert \theta_{\vec{\varphi}}^{j}\Vert_{\mathcal{V}}^{2}\right)
			\, \leq\,
			\mathcal{C}_{\mathtt{est},C}\Big\{\Delta t^{2}+h^2+\Big(\Delta t\disp\sum_{j=1}^{n}\Vert\btet_{\bu}^{j}\Vert_{\mathbf{V}}^{2}\Big)\,\Vert\brho\Vert_{L^{\infty}(\mathrm{J};\mathbf{L}^{2}(\Omega))}\Big\}\,.
		\end{array}
	\end{equation}
\end{lemma}
\begin{proof}
Choosing \( (\bet_h , \vec{\psi}_h) := (\btet_{\brho}^n , \theta_{\vec{\varphi}}^n) \) in \eqref{eq:err.prob2} and applying the definition of \( \mathbf{A}^{C} \) and coercivity of \( \mathscr{A}^C \), we obtain
	\begin{equation}\label{eq:no.A.2}
		\begin{array}{c}
\dfrac{1}{2}\delta_t \big\Vert \theta_{\varphi}^{n}\big\Vert_{0,\Omega}^{2}+\dfrac{1}{\kappa}\,\big\Vert \btet_{\brho}^{n}\big\Vert_{0,\Omega}^{2}
\,\leq\, \bigg| \disp \int_{\Omega} \delta_t\vartheta_{\varphi}^{n} \,\theta_{\varphi}^{n}\bigg|
+ \Big|	\disp \int_{\Omega}\left(\partial_t \varphi^{n}-\delta_t \varphi^n\right) \,\theta_{\varphi}^{n}\Big|\\	\qquad\,+\,\big|\mathscr{A}^C(\bvatet_{\brho}^{n},\btet_{\brho}^{n})\big|+\big|\mathscr{C}^C(\theta_{\vec{\varphi}}^n +\vartheta_{\vec{\varphi}}^n, \theta_{\vec{\varphi}}^n)+\big|\mathscr{O}_{1}^C(\bu^{n};\bvatet_{\brho}^n+\btet_{\brho}^n, \theta_{\varphi}^n)\\
+\big|\mathscr{O}_{1}^C(\bvatet_{\bu}^n+\btet_{\bu}^n;\brho_h^{n},\theta_{\varphi}^n)\big|+\big|\mathscr{O}_{2}^C(\lambda^n;\vartheta_{\vec{\varphi}}^n +\theta_{\vec{\varphi}}^n, \theta_{\vec{\varphi}}^n)\big|+\big|\mathscr{O}_2^C(\vartheta_{\lambda}^n+\theta_{\lambda}^n ; \vec{\varphi}_h^n , \theta_{\vec{\varphi}}^n)\big|\,.
		\end{array}
	\end{equation}
	Similarly as for the derivation of \eqref{eq:m1}, we obtain
	\begin{equation}\label{eq:m1.}
		\begin{array}{c}
			\dfrac{1}{2}\delta_t\big\Vert\theta_{\varphi}^{n}\big\Vert_{0,\Omega}^{2}+\dfrac{1}{\kappa}\,\big\Vert\btet_{\brho}^{n}\big\Vert_{0,\Omega}^{2}\\
			\, \le\, 
			\dfrac{1}{\kappa}\Vert\bu^n\Vert_{\mathbf{V}}\left(\Vert\bvatet_{\brho}^n\Vert_{0,\Omega}+\Vert\btet_{\brho}^n\Vert_{0,\Omega}\right)\Vert\theta_{\varphi}^n\Vert_{\mathrm{V}}+c_{\mathscr{O}_2^C}\Vert\lambda^n\Vert_{\mathrm{Q}}\left(\Vert\vartheta_{\vec{\varphi}}^n\Vert_{\mathcal{V}}+\Vert\theta_{\vec{\varphi}}^n\Vert_{\mathcal{V}}\right)\Vert\theta_{\vec{\varphi}}^n\Vert_{\mathcal{V}}\\
			\,+\,\dfrac{1}{\kappa}\Vert\brho_{h}^n\Vert_{0,\Omega}\left(\Vert\bvatet_{\bu}^n\Vert_{\mathbf{V}}+\Vert\btet_{\bu}^n\Vert_{\mathbf{V}}\right)\Vert\theta_{\varphi}^n\Vert_{\mathrm{V}}+c_{\mathscr{O}_2^C}\Vert\vec{\varphi}_h^n\Vert_{\mathcal{V}}\left(\Vert\vartheta_{\lambda}^n\Vert_{\mathrm{Q}}+\Vert\theta_{\lambda}^n\Vert_{\mathrm{Q}}\right)\Vert\theta_{\vec{\varphi}}^n\Vert_{\mathcal{V}}\\
			\,+\, \Big(\Delta t^{-1/2}\Vert \partial_t\vartheta_{\varphi}\Vert_{L^{2}(\mathrm{J}_{n},\mathrm{L}^{2}(\Omega))} +\Delta t^{1/2}\Vert\partial_{tt}\varphi\Vert_{L^{2}(\mathrm{J}_n;\mathrm{L}^{2}(\Omega))}\Big)\Vert\theta_{\varphi}^n\Vert_{0,\Omega}\,.
		\end{array}
	\end{equation}
	Analogously to \eqref{eq:hj}, applying the discrete inf-sup condition for \( \mathscr{B}^C\) and \eqref{eq:err.prob2} with \( \vec{\psi} = (0,0) \), we deduce
	\begin{equation*}
		\dfrac{\kappa\beta_{C,\mathtt{d}}^2}{4}\Vert\theta_{\vec{\varphi}}^n\Vert_{\mathcal{V}}^2\,\le\, \dfrac{1}{2\kappa}\left(\Vert\btet_{\brho}^n\Vert_{0,\Omega}^2+\Vert\bvatet_{\brho}^n\Vert_{0,\Omega}^2\right)\,,
	\end{equation*}
which, substituting into \eqref{eq:m1.}, applying Young's inequality, bounding \( \|\bu^n\|_{\mathbf{V}} \), \( \|\lambda^n\|_{\mathrm{Q}} \) by \( r_2 \), summing over \( n \), and using \( \theta_{\varphi}^0 = 0 \), we obtain a constant \( \widetilde{\mathcal{C}}_0 > 0 \), depending only on \( \kappa \), \( c_{\mathscr{O}_2^C} \), and \( \beta_{C,\mathtt{d}} \), such that
	\begin{equation}\label{eq:m2.}
		\begin{array}{c}
			\Vert\theta_{\varphi}^{n}\Vert_{0,\Omega}^{2}+\Delta t\disp\sum_{j=1}^{n}\left(\Vert\btet_{\brho}^{j}\Vert_{0,\Omega}^{2}+\Vert \theta_{\vec{\varphi}}^{j}\Vert_{\mathcal{V}}^{2}\right)
			\, \leq\,
			\widetilde{	\mathcal{C}}_{0}\Big\{
			\left(\Delta t\disp\sum_{j=1}^{n}\left(\Vert\btet_{\bu}^{j}\Vert_{\mathbf{V}}^{2}+
			\Vert\theta_{\varphi}^{j}\Vert_{\mathrm{V}}^{2}\right)\right)\,\Vert\brho\Vert_{L^{\infty}(\mathrm{J};\mathbf{L}^{2}(\Omega))}\\
			+\left(\Delta t\sum_{j=1}^{n}\Vert\theta_{\vec{\varphi}}^n\Vert_{\mathcal{V}}^2\right)\Vert\vec{\varphi}_h\Vert_{L^{\infty}(\mathrm{J}_n;\mathcal{V})}
			\,+\,\Delta t\disp\sum\limits_{j=1}^n\left(\Vert\bvatet_{\brho}^j\Vert_{0,\Omega}^2+\Vert\bvatet_{\bu}^j\Vert_{\mathbf{V}}^2+\Vert\vartheta_{\vec{\varphi}}^j\Vert_{\mathcal{V}}^2\right)\\
			\,+\,\Vert \partial_t\vartheta_{\varphi}\Vert_{L^{2}([0,t_n];\mathrm{L}^{2}(\Omega))}^{2} +\Delta t^{2} \,\Vert\partial_{tt}\varphi\Vert_{L^{2}([0,t_n];\mathrm{L}^{2}(\Omega))}^{2}\Big\}\,.
		\end{array}
	\end{equation}
Hence, combining the \emph{a priori} bounds \eqref{eq:apr.est2} and \eqref{eq:dis.apr.est2}, the assumption \eqref{eq:as2.conv}, and the approximation properties of \( \Pibf_h \) \cite[Lemma 3.1]{cmo-ETNA-2018}, \( \Pcal_h \) \cite[Proposition 1.135]{eg-book-2004}, and the Lagrange multiplier interpolation \cite[Section 5.4]{gos-IMA-2012}, yields \eqref{eq:conv2a}, completing the proof.
\end{proof}
To complete the error analysis, we combine \eqref{eq:conv1} and \eqref{eq:conv2a} by multiplying the former by $\frac{1}{2\mathcal{C}_{\mathtt{est},F}}$, adding it to \eqref{eq:conv2a}, and bounding \( \Vert\brho\Vert_{L^{\infty}(\mathrm{J};\mathbf{L}^{2}(\Omega))} \) using \eqref{eq:apr.est2}. This yields constants \( \mathcal{C}_{\mathtt{st}} \) and \( \mathcal{D}_{\mathtt{st}} \), depending only on \( \mathcal{C}_{\mathtt{est},F} \), \( \mathcal{C}_{\mathtt{est},C} \), and \( \mathcal{C}_{2} \) (cf. \eqref{eq:apr.est2}), such that
\begin{equation}\label{eq:fin.est}
	\begin{array}{c}
		\Vert\btet_{\bu}^{n}\Vert_{0,\Omega}^{2}+\Vert\theta_{\varphi}^{n}\Vert_{0,\Omega}^{2}+\Delta t\disp\sum_{j=1}^{n}\left(\Vert (\btet_{\bsi}^{j},\btet_{\bu}^j)\Vert_{\mathbb{H}\times\mathbf{V}}^{2}+\Vert (\btet_{\brho}^{j},\theta_{\vec{\varphi}}^j)\Vert_{\mathbf{H}\times\mathcal{V}}^{2}\right)\\
		\,\leq\,\mathcal{C}_{\mathtt{st}}\big(\Delta t^{2} +h^2\big)\,+\, \mathcal{D}_{\mathtt{st}}\mathcal{N}_{T}(\varphi_{0},\varphi_{D})\,\Delta t\disp\sum_{j=1}^{n}\Vert\btet_{\bu}^{j}\Vert_{\mathbf{V}}^{2}
		\,.
	\end{array}
\end{equation}
Consequently, we are now in a position to state the main result of this section.
\begin{theorem}\label{theorem-rates-of-convergence}
	Assume that the data satisfy \eqref{eq:as1.conv}, \eqref{eq:as2.conv} and
		$\mathcal{D}_{\mathtt{st}}\mathcal{N}_{T}(\varphi_{0},\varphi_{D})\, \leq\, \dfrac{1}{2}$.
	Then, letting $ \mathcal{C}:=2\mathcal{C}_{\mathtt{st}} $, there holds
	\begin{equation*}
		\Vert\btet_{\bu}^{n}\Vert_{0,\Omega}^{2}+\Vert\theta_{\varphi}^{n}\Vert_{0,\Omega}^{2}+\Delta t\disp\sum_{j=1}^{n}\left(\Vert (\btet_{\bsi}^{j},\btet_{\bu}^j)\Vert_{\mathbb{H}\times\mathbf{V}}^{2}+\Vert (\btet_{\brho}^{j},\theta_{\vec{\varphi}}^j)\Vert_{\mathbf{H}\times\mathcal{V}}^{2}\right) \,\leq\,\mathcal{C}\,\big(\Delta t^{2} +h^{2}\big) \,.
	\end{equation*}
\end{theorem}
\begin{proof}
	It suffices to employ $\mathcal{D}_{\mathtt{st}}\mathcal{N}_{T}(\varphi_{0},\varphi_{D})\, \leq\, \dfrac{1}{2}$ in \eqref{eq:fin.est}.
\end{proof}	
\section{Numerical results}\label{sec.7}
We present two test cases to evaluate the method. The first, with a manufactured solution, confirms the convergence rates of Theorem~\ref{theorem-rates-of-convergence}; the second, inspired by reverse osmosis, lacks an exact solution. Both are implemented in $\mathtt{FreeFem}$++, using the decoupling strategy from Section~\ref{s:4}. Iterations stop when the global \( \ell^2 \)-norm of the update falls below a set tolerance. Errors are defined by
${\tt e}(\bsi) := \|\bsi - \bsi_h\|_{\mathbb{H}},\;
{\tt e}(\bu) := \|\bu - \bu_h\|_{\mathbf{V}},\;
{\tt e}(p) := \|p - p_h\|_{0,\Omega},\;
{\tt e}(\brho) := \|\brho - \brho_h\|_{\mathbf{H}},\;
{\tt e}(\varphi) := \|\varphi - \varphi_h\|_{\mathrm{V}},\;
{\tt e}(\lambda) := \|\lambda - \lambda_h\|_{\mathrm{Q}}.$
The experimental convergence rate ${\tt r}(\star)$ for each $\star \in \{\bsi,\bu,p,\brho,\varphi,\lambda\}$ is computed in the standard way.

\subsection{Example 1: Validation}
In our first example, 
we consider the parameters $ \nu,\kappa\,=\, 0.1 $, $ p\,=\,3 $, $ a_0 = 0.5 $, $ a_1 =2 $, the domain $ \Omega =(0,1)^2 $ and the final time $ t_F =0.5  $. In addition, we choose the data so that the exact solution is given by $\bu=[	-\sin(t)\cos(\pi x)\sin(\pi y),	\sin(t)\sin(\pi x)\cos(\pi y)]^{t}$, $p \,=\, \sin(t)(x-1)\exp(y)$ and $\varphi\, = \, \exp(-t)\sin(x)\sin(y)$.
defined over the computational domain $ \textbf{x} \,:=\, (x , y)^{\mathtt{t}}\in\Omega:=(0,1)^{2} $. 
Table \ref{Tab1} summarizes the convergence history of the fully mixed finite element method \eqref{eq:fuldis-v1a-eq:v4}, computed using successively refined meshes. As predicted by Theorem \ref{theorem-rates-of-convergence}, all unknowns converge at the expected rate of $O(h)$
\subsection{Example 2: Spacer-Filled Channel with Physical Parameters}
We consider a 2D narrow horizontal channel with two parallel membrane walls spaced at $h = 0.7221$, filled with spacers in submerged and zigzag configurations. The channel length is $L = 8.6652$, and the spacer diameter $d$ satisfies $d/h = 0.5$, with inter-filament spacing three times the channel height. The domain is discretized using triangular meshes of size $0.0138$, refined near membranes and spacers where velocity and mass fraction gradients are highest.
The model and discretization parameters are taken as
$\nu \,=\, 0.8, \quad \mathtt{F}\,=\, 3,\,\, \kappa \,=\, 10^{-3}, \,\, a_0 \,=\, 2\times 10^{-2}, \,\, a_1 \,=\, 1.8\times 10^{4}$,
$\varphi_{\mathtt{in}}\,=\, 6\times 10^{-10},\,\, 
\bu_{\mathtt{in}}\,=\,\left(10(4y/\mathrm{h}-4(y/\mathrm{h})^2),\, (a_0-a_1 \,\varphi_{\mathtt{in}})(2y/\mathrm{h}-1)\right)^{\mathtt{t}}, \,\,\Delta t \,=\, 10^{-2}$.
Fig.~\ref{fig3} shows the velocity magnitude at $t = 1$ for both submerged and zigzag configurations, consistent with \cite{askkd-JMS-2015}. Concentration contours at various times are shown in Figs.~\ref{fig4} and \ref{fig5} for submerged and zigzag spacers, respectively. In the zigzag case (Fig.~\ref{fig5}), transverse filaments disturb and compress the concentration boundary layer alternately on both membranes. In contrast, Fig.~\ref{fig4} shows that the submerged configuration leaves the boundary layer largely undisturbed due to the filaments' lack of membrane attachment.
\begin{table}[t!]
	\caption{Example 1, history of spatial convergence using $ \Delta t \,=\, h $.}\label{Tab1}		
	\centerline{
		{\scriptsize\begin{tabular}{|ccccccccccc|}			
				\hline\hline
				$h$ & $\mathtt{e}(\boldsymbol{\sigma})$ & ${\tt r}(\boldsymbol{\sigma})$ & 
				$\mathtt{e}(\mathbf{u})$ & ${\tt r}(\mathbf{u})$ & 
				$\mathtt{e}(\boldsymbol{\rho})$ & ${\tt r}(\boldsymbol{\rho})$ & $\mathtt{e}(\varphi)$ &
				${\tt r}(\varphi) $ & $ \mathtt{e}(\lambda) $ & $ {\tt r}(\lambda) $\\
				\hline\hline
				{}&  {}& {}& {}& {}& {}& {}& {}& {}& {}&{}\\
				{$ 1/2^{3} $}&  {1.3591}& {-}& {0.0859}& {-}& {0.0371}& {-}& {0.0080}& {-}& {0.0025}&{-}\\[1mm]
				{$ 1/2^{4} $}&  {0.7080}& {0.940}& {0.0420}& {1.032}& {0.0160}& {1.209}& {0.0036}& {1.126}& {0.0010}& {1.264}\\[1mm]	
				
				{$ 1/2^{5} $}& {0.3694}& {0.982}& {0.0212}& {1.030}& {0.0073}& {1.180}& {0.0016}& {1.178}& {0.0005}& {1.125}\\[1mm]
				
				{$ 1/2^{6} $}& {0.1873}& {0.975}& {0.0106}& {0.993}& {0.0035}& {1.049}& {0.0008}& {1.013}& {0.0002}& {1.065}\\[1mm]		
				\hline
	\end{tabular}}}
\end{table}
\begin{figure}[t!]
	\centering
	\includegraphics[width=.8\textwidth]{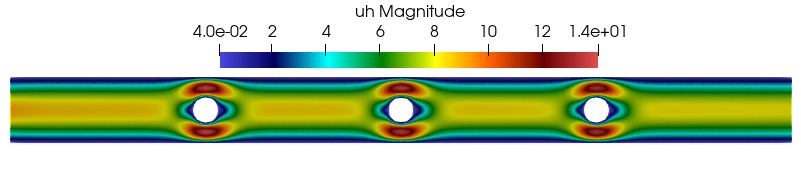}\\[-.3cm]
	\includegraphics[width=.8\textwidth]{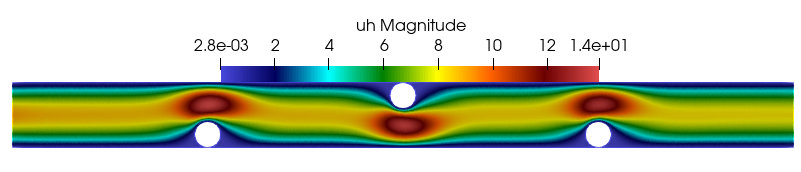}
		\vspace{-.5cm}
	\caption{Velocity magnitude contour in the feed channel with submerged (top) and zigzag (down) spacer for time $ t=1 $.}\label{fig3}
\end{figure}

\begin{figure}[t!]
	\centering
	\includegraphics[width=.8\textwidth]{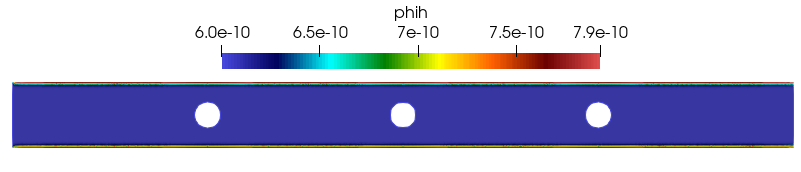}\\[-.2cm]
	\includegraphics[width=.8\textwidth]{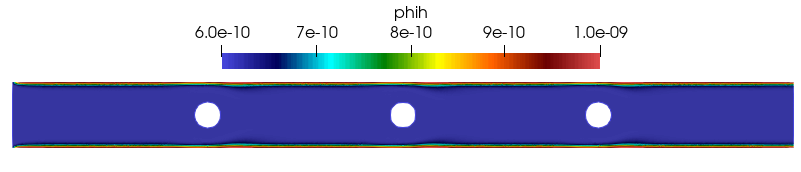}\\[-.2cm]
	\includegraphics[width=.8\textwidth]{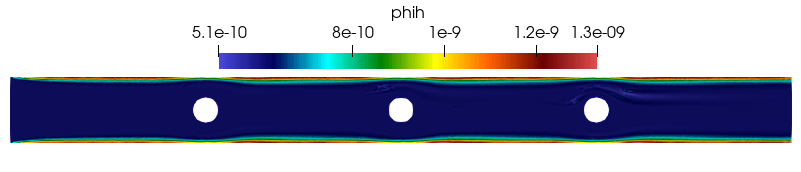}
		\vspace{-.5cm}
	\caption{Salt concentration profiles in the feed channel with submerged spacer for times $ t \,=\, 0.2 $, $ 0.5 $, $ 1 $ (from top to down).}\label{fig4}
\end{figure}

\begin{figure}[t!]
	\centering
	\includegraphics[width=.8\textwidth]{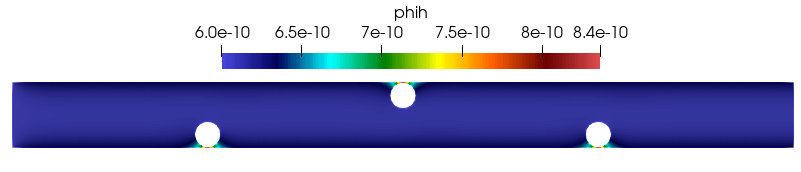}\\[-.2cm]
	\includegraphics[width=.8\textwidth]{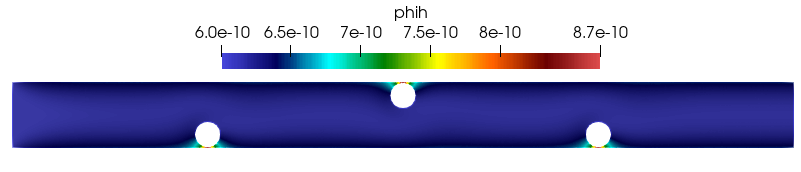}\\[-.2cm]
	\includegraphics[width=.8\textwidth]{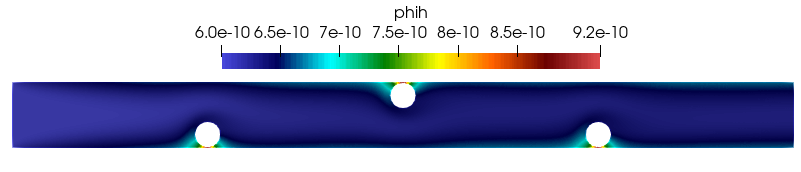}
	\vspace{-.5cm}
	\caption{Salt concentration profiles in the feed channel with zigzag spacer for times $ t \,=\, 0.2 $, $ 0.5 $, $ 1 $ (from top to down)}\label{fig5}
\end{figure}

\bibliographystyle{siam}
\bibliography{bibliography}

\end{document}